\documentclass[11pt, reqno]{amsart}
\usepackage{txfonts}
\usepackage{amsmath,amssymb,amsthm,mathrsfs,enumerate,bm,xcolor,multirow,pbox}
\usepackage{mathrsfs}
\usepackage{algorithm,algorithmic,float}
\usepackage{graphicx,color,framed,tikz,caption,subcaption}
\usepackage{enumitem}
\setlist{leftmargin=9mm}
\usepackage[colorlinks,linkcolor=black,citecolor=black,urlcolor=black]{hyperref}
\allowdisplaybreaks[4]
\numberwithin{equation}{section}
\newcommand{\N}{\mathbb{N}}
\newcommand{\R}{\mathbb{R}}

\newcommand{\pnorm}[2]{\lVert #1\rVert_{#2}}

\newcommand{\biggpnorm}[2]{\bigg\lVert#1\bigg\rVert_{#2}}
\newcommand{\abs}[1]{\lvert#1\rvert}
\newcommand{\bigabs}[1]{\big\lvert#1\big\rvert}
\newcommand{\biggabs}[1]{\bigg\lvert#1\bigg\rvert}
\newcommand{\iprod}[2]{\langle#1,#2\rangle}

\renewcommand{\epsilon}{\varepsilon}

\renewcommand{\d}[1]{\mathrm{d}#1}

\newcommand{\floor}[1]{\lfloor #1 \rfloor}

\newcommand{\bigop}{\mathcal{O}_{\mathbf{P}}}

\newcommand{\smallo}{\mathfrak{o}}
\newcommand{\bigo}{\mathcal{O}}

\renewcommand{\tilde}{\widetilde}

\DeclareMathOperator{\E}{\mathbb{E}}
\DeclareMathOperator{\Prob}{\mathbb{P}}

\DeclareMathOperator{\var}{Var}

\let\liminf\relax
\DeclareMathOperator*\liminf{\underline{lim}}
\let\limsup\relax
\DeclareMathOperator*\limsup{\overline{lim}}
\DeclareMathOperator*{\argmax}{arg\,max\,}

\newcommand{\beq}{\begin{equation}}
\newcommand{\eeq}{\end{equation}}
\newcommand{\beqa}{\begin{equation} \begin{aligned}}
\newcommand{\eeqa}{\end{aligned} \end{equation}}
\newcommand{\beqas}{\begin{equation*} \begin{aligned}}
\newcommand{\eeqas}{\end{aligned} \end{equation*}}

\newcommand{\bit}{\begin{itemize}}
	\newcommand{\eit}{\end{itemize}}
\newcommand{\bmat}{\begin{bmatrix}}
	\newcommand{\emat}{\end{bmatrix}}

\theoremstyle{definition}\newtheorem{problem}{Problem}[section]
\theoremstyle{definition}\newtheorem{definition}[problem]{Definition}
\theoremstyle{definition}
\theoremstyle{remark}\newtheorem{assumption}{Assumption}

\theoremstyle{remark}\newtheorem{remark}{Remark}
\theoremstyle{definition}
\theoremstyle{plain}\newtheorem{theorem}[problem]{Theorem}
\theoremstyle{plain}\newtheorem{question}{Question}
\theoremstyle{plain}\newtheorem{lemma}[problem]{Lemma}
\theoremstyle{plain}\newtheorem{proposition}[problem]{Proposition}
\theoremstyle{plain}\newtheorem{corollary}[problem]{Corollary}
\theoremstyle{plain}

\AtBeginDocument{%
	\def\MR#1{}
}

\begin{document}

\title[Thompson sampling: precise arm-pull dynamics and adaptive inference]{Thompson sampling: precise arm-pull dynamics and adaptive inference}

\author[Q. Han]{Qiyang Han}

\address[Q. Han]{
Department of Statistics, Rutgers University, Piscataway, NJ 08854, USA.
}
\email{qh85@stat.rutgers.edu}

\date{\today}


\begin{abstract}
Adaptive sampling schemes are well known to create complex dependence that may invalidate conventional inference methods based on i.i.d. data. A recent line of work, however, shows that this need not be the case for UCB-type algorithms in multi-armed bandits. A central emerging theme is a `stability' property with asymptotically deterministic arm-pull counts in these algorithms, making inference essentially as easy as in the i.i.d. setting.

In this paper, we study the precise arm-pull dynamics in another canonical class of Thompson-sampling type algorithms. We show that the phenomenology is qualitatively different: the arm-pull count is asymptotically deterministic if and only if the arm is suboptimal or is the unique optimal arm; otherwise it converges in distribution to a non-degenerate limit given by the unique invariant distribution of a stochastic differential equation (SDE). This dichotomy uncovers a unifying principle behind many existing (in)stability results: an arm is stable if and only if its interaction with statistical noise is asymptotically negligible.

As an application, we show that normalized arm means obey the same dichotomy, with Gaussian limits for stable arms and a semi-universal, non-Gaussian limit for unstable arms. This not only enables the construction of confidence intervals for the unknown mean rewards despite non-normality, but also reveals the potential of developing tractable inference procedures by understanding the precise stochastic dynamics of arm pulls beyond the stable regime.

The proofs rely on two new approaches tailored to suboptimal and optimal arms. For suboptimal arms, we develop an `inverse process' approach that characterizes the inverse of the arm-pull count process as an approximate integrator in a deterministic Stieltjes integral. For optimal arms, we adopt an `unnatural reparametrization' of the arm-pull and noise processes that reduces the difficulty of formalizing a natural SDE with a singular initial condition to proving the uniqueness of the invariant distribution of another SDE. We prove the latter by establishing the strong Feller property and irreducibility of the associated Markov semigroup, via a set of analytic tools including the parabolic H\"ormander condition and the Stroock-Varadhan support theorem.
\end{abstract}

\maketitle
\setcounter{tocdepth}{1}
\tableofcontents

\sloppy

\section{Introduction}

\subsection{Motivating questions}
The multi-armed bandit problem \cite{thompson1933likelihood,robbins1952some} is a fundamental theoretical paradigm that captures the inherent exploration-exploitation trade-off in modern sequential decision-making. To fix notation, we consider a standard \(K\)-armed bandit: at each round \(t\in [T]\), the player selects an arm \(A_t\in [K]\) based on the past rewards \(R_1,\ldots,R_{t-1}\), and then observes a reward generated according to
\begin{align}\label{eqn:bandit_seq_model}
R_t \equiv \mu_{A_t}+\sigma\cdot \xi_t .
\end{align}
Here \(\mu_a\) denotes the mean reward of arm \(a\in [K]\), \(\{\xi_t\}_{t\in [T]}\) are i.i.d. standardized noise variables, and $\sigma$ is the noise level. We write \(\mu_\ast\equiv\max_{b\in [K]}\mu_b\) for the optimal mean reward, and \(\Delta_a\equiv\mu_\ast-\mu_a\ge 0\) for the suboptimality gap of arm \(a\in [K]\). We also write \(\mathcal{A}_0\equiv\{a\in [K]:\Delta_a=0\}\) and \(\mathcal{A}_+\equiv\{a\in [K]:\Delta_a>0\}\) as the sets of optimal and suboptimal arms, respectively.

In this paper, we focus on a canonical Bayesian approach to the $K$-armed bandit problem \eqref{eqn:bandit_seq_model}, namely Thompson sampling \cite{thompson1933likelihood}. Roughly speaking, Thompson sampling places a prior on the unknown mean rewards in \eqref{eqn:bandit_seq_model} and maintains a posterior distribution for each arm to balance exploration and exploitation. Owing to its flexibility for more complex models and strong empirical performance, Thompson sampling has become one of the most popular bandit algorithms; see, e.g., \cite{chapelle2011empirical,agrawal2012analysis,kaufmann2012thompson,korda2013thompson,russo2014learning,russo2016information,russo2018tutorial,bubeck2023first}.

A large body of theoretical work \cite{agrawal2012analysis,agrawal2017near,kaufmann2012thompson,lattimore2020bandit} on Thompson sampling has been devoted to understanding its regret behavior. This line of research, together with the powerful technical toolkits developed therein, establishes the (near) regret optimality of various Thompson-sampling algorithms, both in the classical asymptotic sense of Lai-Robbins \cite{lai1985asymptotically} and in the worst-case minimax sense of \cite{auer2002nonstochastic}.

From a statistical perspective, while the (near) regret optimality of Thompson sampling is desirable for reward maximization, this optimality criterion does not directly pinpoint its utility in the equally important problem of statistical inference. Indeed, it is classical knowledge that adaptive sampling schemes can induce complicated dependence structures that invalidate conventional inference methods developed for i.i.d. data; see, e.g., \cite{white1958limiting,dickey1979distribution,lai1982least}. Such complications have also been observed for bandit algorithms; see, e.g., \cite{deshpande2018accurate,zhang2020inference,deshpande2023online,lin2024statistical,ying2024adaptive,khamaru2025near,lin2025semi,halder2025stable}. This gap naturally leads to the following question:

\begin{question}\label{question:inference}
	Can standard Thompson sampling be used for statistical inference of the unknown mean rewards in \eqref{eqn:bandit_seq_model}?
\end{question}

A recent line of work \cite{khamaru2024inference,han2024ucb,fan2024precise,fan2025statistical} suggests a new approach to statistical inference in bandits for a closely related, general class of Upper Confidence Bound (UCB) algorithms by studying the precise behavior of arm pulls. A central theme in these works is that many UCB algorithms enjoy a \emph{stability} property, in that the arm-pull counts \(\{n_{a;T}\}\) are asymptotically deterministic. Consequently, inference for the unknown reward means \(\{\mu_a\}\) can be carried out using the empirical means \(\hat{\mu}_{a;T}\) via conventional asymptotic normality, essentially as if the data were collected in an i.i.d. manner.  Here, for each arm \(a\in [K]\), we write
\begin{align}\label{def:num_arm_pull}
n_{a;T}\equiv \sum_{s \in [T]} \bm{1}_{A_s=a},\quad \hat{\mu}_{a;T}\equiv \frac{1}{n_{a;T}}\sum_{s \in [T]} \bm{1}_{A_s=a} R_s,\quad \forall a \in [K],
\end{align}
where it is understood that \(0/0=0\).

Unfortunately, the stability of arm pulls appears substantially more subtle for Thompson sampling. For instance, in the two-armed case \(K=2\), if both arms are optimal, then neither arm-pull count \(\{n_{a;T}\}\) concentrates under the model \eqref{eqn:bandit_seq_model} and beyond; see, e.g., \cite{kalvit2021closer,halder2025stable}. In contrast, if one arm is suboptimal, then both arm-pull counts are provably asymptotically deterministic \cite{fan2022typical}, and hence satisfy the stability notion described above.

This mixed and fragmented picture of arm-pull dynamics under Thompson sampling motivates the following question:

\begin{question}\label{question:arm_pull}
	Can we provide a complete characterization of the arm pull dynamics for Thompson sampling?
\end{question}

\subsection{Precise arm-pull dynamics}
\begin{algorithm}[!t]
	\caption{Generalized Gaussian Thompson sampling}\label{alg:ts}
	\raggedright
	\textbf{Input:} (i) Number of epoch $T$, (ii) noise level $\sigma$, (iii) sampling distribution $\mathsf{Z}$.\\
	\textbf{Initialization:} Set $\bar{n}_{a;0}\gets 1$ and $\bar{\mu}_{a;0}\gets 0$ for all $a\in[K]$.\\
	\textbf{For} $t=1,\ldots,T$ \textbf{do}:
	\begin{enumerate}
		\item Pull arm $A_{t} \in \argmax_{a \in [K]}\big\{ \bar{\mu}_{a;t-1}+ \frac{\sigma }{  \sqrt{\bar{n}_{a;t-1}}} Z_{a;t} \big\}$, where $Z_{\cdot;\cdot} \stackrel{\mathrm{i.i.d.}}{\sim}\mathsf{Z}$, and observe the reward $R_t$ according to (\ref{eqn:bandit_seq_model}).
		\item For all $a \in [K]$, update 
		\begin{align*}
		\bar{n}_{a;t}\gets \bar{n}_{a;t-1}+\bm{1}_{A_t=a},\quad \bar{\mu}_{a;t}\gets \frac{1}{ \bar{n}_{a;t}}\big(\bar{n}_{a;t-1}\bar{\mu}_{a;t-1}+\bm{1}_{A_t=a} R_t\big).
		\end{align*}
	\end{enumerate}
	\textbf{End for}\\[0.2em]
\end{algorithm}

The first goal of this paper is to give an affirmative answer to Question \ref{question:arm_pull} by analyzing a generalized Gaussian Thompson sampling algorithm for the model \eqref{eqn:bandit_seq_model}, detailed in Algorithm \ref{alg:ts}. When the sampling distribution $\mathsf{Z}\sim \mathcal{N}(0,1)$, Algorithm \ref{alg:ts} reduces to the usual Gaussian Thompson sampling corresponding to a $\mathcal{N}(0,1)$ prior on the unknown mean rewards\footnote{Our analysis and results in fact hold for any normal prior $\mathcal{N}(\mu_{\mathrm{p}},\sigma_{\mathrm{p}}^2)$. To keep notation simple, we focus on the canonical prior $\mathcal{N}(0,1)$.}.

Our first set of main results provide a complete characterization of the arm-pull dynamics $\{n_{a;T}\}$ for Algorithm \ref{alg:ts}:
\begin{itemize}
	\item(\emph{Suboptimal arms}). For each suboptimal arm $a \in \mathcal{A}_+$ with $\Delta_a>0$, Theorem \ref{thm:suboptimal_arm_pull} shows that $n_{a;T}$ is asymptotically deterministic in the sense that
	\begin{align}\label{eqn:intro_suboptimal_arm}
	\frac{n_{a;T}/\sigma^2}{[\bar{\Phi}^-(1/T)/\Delta_a]^2} \stackrel{\Prob}{\to} 1.
	\end{align}
	Here $\bar{\Phi}(\cdot)\equiv\Prob(\mathsf{Z}>\cdot)$ and $\bar{\Phi}^-$ denotes its generalized inverse.
	\item (\emph{Optimal arms}). For optimal arms $a \in \mathcal{A}_0$, Theorem \ref{thm:optimal_arm_pull} shows a qualitatively different behavior: the vector of arm-pull proportions satisfies the joint weak convergence
	\begin{align}\label{eqn:intro_optimal_arm}
	\bigg(\bigg\{\frac{n_{a;T}}{T}\bigg\}_{a \in \mathcal{A}_0}\bigg)\rightsquigarrow \mathbb{L}_{\abs{\mathcal{A}_0}}^{[r],\ast}.
	\end{align}
	Here the limiting random variable $\mathbb{L}_{\abs{\mathcal{A}_0}}^{[r],\ast}$ is specified as (part of) the unique invariant distribution of a stochastic differential equation (SDE) (Eqn. \eqref{def:sde_time_change_renor}). Moreover, the law of $\mathbb{L}_{\abs{\mathcal{A}_0}}^{[r],\ast}$ depends only on the number of optimal arms $\abs{\mathcal{A}_0}$, and is degenerate with a point mass at $1$ when $\abs{\mathcal{A}_0}=1$.
\end{itemize}

At a deeper level, our proofs of the dichotomy in \eqref{eqn:intro_suboptimal_arm}-\eqref{eqn:intro_optimal_arm} reveal a fundamental principle underlying the notion of \emph{stability} mentioned above: an arm has an asymptotically deterministic pull count (and is thus stable) if and only if its interaction with the statistical noise $\{\xi_t\}$ is asymptotically negligible. In particular, we show that (i) the effect of statistical noise on all suboptimal arms and on the unique optimal arm vanishes as $T\to\infty$, whereas (ii) when there are multiple optimal arms, noise interacts with the dynamics in a nontrivial way even as $T\to\infty$, and this interaction is captured by the invariant distribution of an SDE. This principle also unifies several recent stability and instability results for related bandit algorithms. For example, the stability of UCB-type algorithms in \cite{khamaru2024inference,han2024ucb,fan2024precise,fan2025statistical} can be viewed as a consequence of the dominance of exploration rates over statistical noise, while the instability results for other UCB-type algorithms in \cite{praharaj2025instability} correspond to the opposite regime where exploration is insufficient relative to noise. We refer the reader to Section \ref{subsection:stability} for further discussion.

It is also worth noting that \eqref{eqn:intro_suboptimal_arm} sheds further light on potential limitations of the Lai-Robbins lower bound \cite{lai1985asymptotically} (see, e.g., \eqref{ineq:lai_robbins}) as a descriptor of typical bandit behavior. In particular:
\begin{itemize}
	\item In the special case of Gaussian Thompson sampling, \eqref{eqn:intro_suboptimal_arm} shows that the typical growth of $n_{a;T}$ is indeed well captured by the logarithmic rate suggested by the Lai-Robbins lower bound for $\E n_{a;T}$.
	\item For general sampling schemes $\mathsf{Z}$, the pull count $n_{a;T}$ may grow \emph{arbitrarily slowly}, so its typical behavior can differ substantially from that of $\E n_{a;T}$ predicted by the Lai-Robbins lower bound.
\end{itemize}
This observation complements the heavy-tail phenomenon for the pseudo-regret of a broad class of bandit algorithms \cite{fan2024fragility} to an extreme extent, that the pseudo-regret may have such heavy tails that its first moment can already be arbitrarily larger in order than its typical magnitude.

\subsection{Adaptive inference}

The second goal of this paper is to leverage the precise characterizations of the arm-pull dynamics in \eqref{eqn:intro_suboptimal_arm} and \eqref{eqn:intro_optimal_arm} to give an affirmative answer to Question \ref{question:inference}.

As discussed above, when an arm’s pull count is provably asymptotically deterministic, the arm is stable in the sense of \cite{lai1982least}, and conventional inference based on asymptotic normality becomes available. For the generalized Gaussian Thompson sampling scheme in Algorithm \ref{alg:ts}, our theory \eqref{eqn:intro_suboptimal_arm}-\eqref{eqn:intro_optimal_arm} implies that stability holds if and only if the arm is suboptimal or is the unique optimal arm. Consequently, the main inferential challenge arises precisely when there are multiple optimal arms.

Our inference proposal hinges on the fact that \eqref{eqn:intro_suboptimal_arm}-\eqref{eqn:intro_optimal_arm} is strong enough to yield a distributional limit for the normalized empirical arm mean (formalized in Theorem \ref{thm:arm_mean_dist}):
\begin{align}\label{eqn:intro_arm_mean_dist}
\big(n_{a;T}/\sigma^2\big)^{1/2}\big(\hat{\mu}_{a;T}-\mu_a\big) \rightsquigarrow
\begin{cases}
\mathcal{N}(0,1), & \text{if } a \in \mathcal{A}_+ \text{ is suboptimal},\\
\mathscr{N}_{\abs{\mathcal{A}_0}}, & \text{if } a \in \mathcal{A}_0 \text{ is optimal}.
\end{cases}
\end{align}
Here the law of \(\mathscr{N}_{\abs{\mathcal{A}_0}}\) depends only on the number of optimal arms \(\abs{\mathcal{A}_0}\), and it coincides with \(\mathcal{N}(0,1)\) when \(\abs{\mathcal{A}_0}=1\). Although the analytic properties of \(\mathscr{N}_{\abs{\mathcal{A}_0}}\) remain largely unknown, its distribution is straightforward to simulate, which enables the construction of valid confidence intervals for the unknown means \(\{\mu_a\}\) based on the empirical means \(\hat{\mu}_{a;T}\) and the arm-pull counts $\{n_{a;T}\}$.

From a broader perspective, the literature on inference with adaptively collected data has developed two main methodological approaches. The first exploits the martingale structure of the data and debiasing techniques to construct confidence intervals via martingale central limit theorems; see, e.g., \cite{deshpande2018accurate,zhang2020inference,bibaut2021post,hadad2021confidence,zhan2021off,deshpande2023online,syrgkanis2023post,lin2024statistical,ying2024adaptive,khamaru2025near,lin2025semi}. The second builds on concentration inequalities for self-normalized martingales (cf.\ \cite{de2004self,de2009self}), typically yielding considerably wider confidence intervals; see, e.g., \cite{abbasi2011improved,shin2019bias,waudby2024anytime}.

Our approach is qualitatively different from both above lines, and, as mentioned above, is actually more closely aligned with recent proposals \cite{khamaru2024inference,han2024ucb,fan2024precise,fan2025statistical} that aim to recover conventional, asymptotic-normality-based inference through arm stability. Our inference proposal based on the theory \eqref{eqn:intro_arm_mean_dist} highlights a further potential of this paradigm: inference can be enabled by understanding the precise stochastic dynamics of arm pulls beyond the stable regime.

\subsection{Proof techniques}

The proof for our theory in \eqref{eqn:intro_suboptimal_arm}-\eqref{eqn:intro_optimal_arm} relies on two new technical approaches, developed separately for suboptimal and optimal arms.

To prove \eqref{eqn:intro_suboptimal_arm} for suboptimal arms $a\in\mathcal{A}_+$, we develop an `inverse-process' approach to characterize the asymptotically deterministic behavior of $\{n_{a;T}\}$. At a high level, instead of working directly with the process $t\mapsto n_{a;t}$, we study its inverse process, which we show can be approximately characterized as the integrator in a deterministic Stieltjes integral. A subsequent inversion of this Stieltjes integral then yields the asymptotic behavior of $n_{a;T}$ as $T\to\infty$. Technically, our approach moves beyond the limitations of the existing method in \cite{fan2022typical}, which relies crucially on the unique optimal arm setting and thereby avoids the essential difficulty arising from noise interactions among multiple optimal arms.

To prove \eqref{eqn:intro_optimal_arm} for optimal arms $a\in\mathcal{A}_0$, our method relies on an `unnatural reparametrization' of the arm-pull count and noise processes with a time change and renormalization. This reduces the problem of rigorously formalizing a `natural' SDE with a singular initial condition to establishing uniqueness of the invariant distribution of a time-changed and renormalized SDE. We prove the latter by establishing the strong Feller property and (topological) irreducibility of the associated Markov semigroup, using analytic tools including a suitable localized form of the parabolic H\"ormander condition \cite{hormander1967hypoelliptic,hairer2011malliavin} and the Stroock-Varadhan support theorem \cite{stroock1972support,benarous1994holder,millet1994simple}.

\subsection{Organization}

The rest of the paper is organized as follows. Section \ref{section:main_results} formally presents our theory \eqref{eqn:intro_suboptimal_arm}-\eqref{eqn:intro_optimal_arm}, along with the minimal background needed on the associated stochastic differential equation. Section \ref{section:adaptive_inf} details the distributional theory \eqref{eqn:intro_arm_mean_dist} and the resulting inference method, together with some illustrative numerical demonstrations. Proof outlines for \eqref{eqn:intro_suboptimal_arm} and \eqref{eqn:intro_optimal_arm} are given in Sections \ref{section:proof_outline_suboptimal} and \ref{section:proof_outline_optimal}, respectively, while complete proofs are deferred to Sections \ref{section:proof_suboptimal}-\ref{section:proof_adaptive_inf} and the appendices.

\subsection{Notation}\label{section:notation}

For any two integers $m,n \in \mathbb{Z}$, let $[m:n]\equiv \{m,m+1,\ldots,n\}$ when $m\leq n$ and $\emptyset$ otherwise. Let $[m:n)\equiv [m:n-1]$, $(m:n]\equiv [m+1:n]$, and we write $[n]\equiv [1:n]$. For $a,b \in \R$, $a\vee b\equiv \max\{a,b\}$ and $a\wedge b\equiv\min\{a,b\}$. For $a \in \R$, let $a_\pm \equiv (\pm a)\vee 0$. For $a>0$, let $\log_+(a)\equiv 1\vee \log(a)$. For $x \in \R^n$, let $\pnorm{x}{p}$ denote its $p$-norm $(0\leq p\leq \infty)$, and $B_{n;p}(R)\equiv \{x \in \R^n: \pnorm{x}{p}< R\}$. We simply write $\pnorm{x}{}\equiv\pnorm{x}{2}$ and $B_n(R)\equiv B_{n;2}(R)$. For two integers $k_1>k_2$, we interpret $\sum_{k=k_1}^{k_2}\equiv 0, \prod_{k=k_1}^{k_2}\equiv 1$.

We use $C_{x}$ to denote a generic constant that depends only on $x$, whose numeric value may change from line to line unless otherwise specified. $a\lesssim_{x} b$ and $a\gtrsim_x b$ mean $a\leq C_x b$ and $a\geq C_x b$, abbreviated as $a=\bigo_x(b), a=\Omega_x(b)$ respectively;  $a\asymp_x b$ means $a\lesssim_{x} b$ and $a\gtrsim_x b$, abbreviated as $a=\Theta_x(b)$. For two nonnegative sequences $\{a_n\}$ and $\{b_n\}$, we write $a_n\ll b_n$ (respectively~$a_n\gg b_n$) if $\lim_{n\rightarrow\infty} (a_n/b_n) = 0$ (respectively~$\lim_{n\rightarrow\infty} (a_n/b_n) = \infty$).  $\bigo$ and $\smallo$ (resp. $\mathcal{O}_{\mathbf{P}}$ and $\mathfrak{o}_{\mathbf{P}}$) denote the usual big and small O notation (resp. in probability). 

For a random variable $X$, we denote $\mathscr{L}(X)$ as its law, and we use $\Prob_X,\E_X$ (resp. $\Prob^X,\E^X$) to indicate that the probability and expectation are taken with respect to $X$ (resp. conditional on $X$).  The notation $\rightsquigarrow$ is reserved for weak convergence. 

For Euclidean space $E \subset \R^n$, we write $\mathcal{B}(E)$ as the Borel $\sigma$-algebra on $E$, $B(E)$ as the space of all bounded, measurable functions on $E$, and $C_b(E)$ as the space of all bounded, continuous functions on $E$. For a monotonic non-increasing function $F:\R \to \R$, let $F^-(u)\equiv \inf\{x\in \R: F(x)\leq u\}$ denote its generalized inverse.

\section{Precise arm-pull dynamics}\label{section:main_results}

\subsection{Assumptions}

First, we state our minimal assumptions on the noise variables $\{\xi_t\}$ in the statistical model \eqref{eqn:bandit_seq_model}.

\begin{assumption}\label{assump:noise}
	The noises $\{\xi_t\}$ are i.i.d. variables with mean 0 and variance 1.
\end{assumption}

Next, we state our assumptions on the sampling scheme $\mathsf{Z}$ used in Algorithm \ref{alg:ts}. We introduce the notation
\begin{align*}
\Phi(z) &\equiv \mathbb{P}(\mathsf{Z}\le z), 
\qquad \bar{\Phi}(z) \equiv 1-\Phi(z).
\end{align*}
\begin{assumption}\label{assump:sampling_dist}
	Suppose the following conditions hold:
	\begin{enumerate}
		\item[(B1)] For any $z \in \R$, $\Phi(z)\in(0,1)$. Moreover, $\Phi\in C^1(\R)$ and $\pnorm{\Phi'}{\infty}<\infty$.
		\item[(B2)] For any $c>1$, it holds that
		\begin{align*}
		\lim_{z\uparrow \infty} \frac{\log z}{-\log \Phi(-z)}= 0,\quad \lim_{z\uparrow \infty} \max\bigg\{\frac{\log z}{-\log \bar{\Phi}(z)},\frac{z^2\bar{\Phi}(cz)}{\bar{\Phi}(z)}\bigg\}  = 0.
		\end{align*}	
	\end{enumerate}
\end{assumption}

Condition (B1) is mild: it requires $\mathsf{Z}$ to have unbounded support and a bounded Lebesgue density. Condition (B2) imposes a left tail-decay requirement on $\Phi$, and a right tail-decay requirement and a regularly varying property on $\bar{\Phi}$. A simple calculation shows that (B2) holds if there exist $\alpha_\pm,c_\pm>0$ such that
\begin{align}\label{cond:assump_B_tail}
\lim_{z \uparrow \infty} \frac{-\log \Phi(-z)}{z^{\alpha_-}} = c_-,\quad \lim_{z \uparrow \infty} \frac{-\log \bar{\Phi}(z)}{z^{\alpha_+}} = c_+.
\end{align}
Condition \eqref{cond:assump_B_tail} covers many light-tailed distribution classes, including Gaussian, exponential, Gamma, and Weibull. Other distribution classes with even lighter or heavier tails can be accommodated by replacing the polynomial function $z^{\alpha_\pm}$ with an appropriate alternative.

\subsection{Sub-optimal arms}

The following theorem describes the asymptotically deterministic behavior of the arm-pull count $n_{a;T}$ for suboptimal arms $a \in \mathcal{A}_+$.

\begin{theorem}\label{thm:suboptimal_arm_pull}
	Suppose Assumptions \ref{assump:noise} and \ref{assump:sampling_dist} hold. Fix $\epsilon \in (0,1/2)$. Then there exists some $c_0=c_0(K,\Delta,\epsilon,\sigma,\mathscr{L}(\mathsf{Z}))>1$ such that 
	\begin{align*}
	&\inf_{T\geq c_0}\Prob\bigg( 1-\epsilon \leq \frac{n_{a;T}/\sigma^2}{[\bar{\Phi}^-(1/T)/\Delta_a]^2} \leq 1+\epsilon+ \frac{c_0}{[\bar{\Phi}^-(1/T)]^2},\, \forall a \in \mathcal{A}_+\bigg)\geq 1-\epsilon.
	\end{align*}
	Here recall $\bar{\Phi}^-$ is the generalized inverse of $\bar{\Phi}$.
\end{theorem}

The proof of Theorem \ref{thm:suboptimal_arm_pull} relies on a new `inverse process' approach: instead of working directly with the arm-pull count process $t\mapsto n_{a;t}$, we work with its inverse process $n\mapsto \tau_{a;n}$ (see Definition \ref{def:hitting_time}). We refer the reader to Section \ref{section:proof_outline_suboptimal} for a detailed explanation of the technical merit and a proof outline of this approach.

As a consequence of Theorem \ref{thm:suboptimal_arm_pull}, we have the following corollary (formally proved in Section \ref{subsection:proof_suboptimal_arm_rates}):
\begin{corollary}\label{cor:suboptimal_arm_rates}
Suppose Assumptions \ref{assump:noise} and \ref{assump:sampling_dist} hold. Then (\ref{eqn:intro_suboptimal_arm}) holds for all $a \in \mathcal{A}_+$. In particular, for Gaussian Thompson sampling, 
\begin{align}\label{ineq:suboptimal_arm_pull_GTS}
\max_{a \in \mathcal{A}_+} \biggabs{\frac{n_{a;T}/\sigma^2 }{(2\log T)/\Delta_a^2}-1}\stackrel{\Prob}{\to} 0.
\end{align}
Moreover, for any slowly growing $r_T \uparrow \infty$, there exists a sampling scheme $\mathsf{Z}$ in Algorithm \ref{alg:ts}, such that 
\begin{align}\label{ineq:suboptimal_arm_pull_slow}
\max_{a \in \mathcal{A}_+} n_{a;T}= \bigop(r_T). 
\end{align}
\end{corollary}
It is interesting to compare Corollary \ref{cor:suboptimal_arm_rates} with the seminal  Lai-Robbins lower bound \cite{lai1985asymptotically}. In our model (\ref{eqn:bandit_seq_model}) with Gaussian errors $\xi_i \stackrel{\mathrm{i.i.d.}}{\sim} \mathcal{N}(0,1)$, this lower bound states that for any fixed $K \in \N$ and suboptimal arm $a \in \mathcal{A}_+$ with suboptimality gap $\Delta_a>0$, the arm-pull count $n_{a;T}(\mathscr{A})$ for any `uniformly good' bandit algorithm $\mathscr{A}$ satisfies 
\begin{align}\label{ineq:lai_robbins}
\liminf_{T\to \infty} \frac{\E n_{a;T}(\mathscr{A})/\sigma^2}{2\log T} \geq  \frac{1}{\Delta_a^2}.
\end{align}
We have two observations:
\begin{itemize}
	\item While Gaussian Thompson sampling is known to attain the Lai-Robbins lower bound (\ref{ineq:lai_robbins}) \cite{agrawal2017near}, here our (\ref{ineq:suboptimal_arm_pull_GTS}) proves a much stronger result in that the \emph{typical behavior} of suboptimal arm pulls is exactly tracked by (\ref{ineq:lai_robbins}). 
	\item On the other hand, the Lai-Robbins lower bound (\ref{ineq:lai_robbins}) does not, in general, indicate the typical behavior of general Thompson sampling schemes. In particular, (\ref{ineq:suboptimal_arm_pull_slow}) shows the arm-pull count $n_{a;T}$ in Thompson sampling can actually grow \emph{arbitrarily} slowly for all suboptimal arms.
\end{itemize}
In a related direction, \cite{fan2024fragility} shows that many regret-optimal bandit algorithms in the sense of (\ref{ineq:lai_robbins}) must exhibit heavy tails, including the Gaussian Thompson sampling scheme. Here the discrepancy between (\ref{ineq:suboptimal_arm_pull_slow}) and (\ref{ineq:lai_robbins}) shows that heavy tails are so extreme that the typical behavior of $n_{a;T}$ and its first moment $\E n_{a;T}$ can be drastically different for generalized Thompson sampling schemes.

\begin{remark}
\cite[Theorem 1]{fan2022typical} proves \eqref{ineq:suboptimal_arm_pull_GTS} for the two-armed setting, and outlines a sketch for the multi-armed case with a unique optimal arm. Here \eqref{ineq:suboptimal_arm_pull_GTS} is proved for the most general multi-armed case with possibly multiple optimal arms, and our more general formula \eqref{eqn:intro_suboptimal_arm} applies well beyond the Gaussian sampling scheme. As will be clear below, the presence of multiple optimal arms brings about essential technical differences compared to the unique-optimal-arm case.
\end{remark}

\subsection{Optimal arms}

For $a \in \mathcal{A}_0$, let $p_{a}: \R_{>0}^{\mathcal{A}_0}\times \R^{\mathcal{A}_0}\to \R_{\geq 0}$ be defined by
\begin{align}\label{def:p_a}
p_{a}\big(r_{\mathcal{A}_0},\xi_{\mathcal{A}_0}\big)= \E_{\mathsf{Z}} \prod_{b \in \mathcal{A}_0 \setminus [a]} \Phi\,\bigg[\sqrt{r_b}\cdot \bigg(\frac{\xi_a}{r_a}-\frac{\xi_b}{r_b}+\frac{\mathsf{Z}}{\sqrt{r_a}}\bigg) \bigg].
\end{align}
When $\abs{\mathcal{A}_0}=1$, we interpreted the above display as  $p_a\equiv 1$.

\begin{definition}\label{def:SDE_ts}
The stochastic differential equation (SDE) for the generalized Gaussian Thompson sampling in Algorithm \ref{alg:ts} is given by
\begin{align}\label{def:sde_time_change_renor}
\begin{cases}
\d{u_a}(t) = \big(p_a(u_\cdot(t),w_\cdot(t))-u_a(t)\big)\,\d{t},\\
\d{w_a}(t) = -\frac{1}{2} w_a(t)\,\d{t}+\sqrt{p_a(u_\cdot(t),w_\cdot(t))}\,\d{B_a(t)},
\end{cases}
\forall a \in \mathcal{A}_0,\, t \in \R.
\end{align}
Here $\{B_a: a \in \mathcal{A}_0\}$ are independent two-sided Brownian motions with $B_a(0)=0$. 
\end{definition}
For notational simplicity, we also write 
\begin{align}\label{def:Delta_circ}
\Delta_{\mathcal{A}_0}^\circ\equiv \bigg\{(u_a)_{a \in \mathcal{A}_0}: \sum_{a \in \mathcal{A}_0} u_a = 1, \inf_{a \in \mathcal{A}_0}u_a>0\bigg\},\quad E_0\equiv \Delta_{\mathcal{A}_0}^\circ\times \R^{\mathcal{A}_0}.
\end{align}
The following proposition shows the well-posedness of the SDE (\ref{def:sde_time_change_renor}).
\begin{proposition}\label{prop:sde_strong_sol}
	Suppose (B1) in Assumption \ref{assump:sampling_dist} holds. For any initial condition $(u_{\cdot}(0),w_{\cdot}(0))\in E_0$, there exists a unique strong solution $(u_{\cdot}(t),w_{\cdot}(t))$ of the SDE (\ref{def:sde_time_change_renor}) on $[0,\infty)$ with $(u_{\cdot}(t)) \in \Delta_{\mathcal{A}_0}^\circ$ for all $t \in [0,\infty)$.
\end{proposition}
The initial time is chosen to be $t=0$ merely for convenience. A technical difficulty in proving the well-posedness of the SDE \eqref{def:sde_time_change_renor} is that the functions $\{p_a\}$ in \eqref{def:p_a} are not globally Lipschitz. We address this issue by deriving suitable apriori localization estimates for the solution $(u_\cdot,w_\cdot)$.  Details of the proof of Proposition \ref{prop:sde_strong_sol} are provided in Section \ref{subsection:proof_sde_strong_sol}.

We now recall the notions of a (Markov) semigroup, transition function, and invariant measure; see, e.g., \cite[Section 2.1]{daprato1996ergodicity}. For the state space $E_0$, we also recall the notation $\mathcal{B}(E_0)$ and $B(E_0)$ introduced in Section \ref{section:notation}.

\begin{definition}\label{def:semigroup}
	Let $\big\{X(t)\equiv (u_a(t),w_a(t))_{a \in \mathcal{A}_0}\big\}_{t\geq 0} \subset E_0$ be the unique strong solution to the SDE (\ref{def:sde_time_change_renor}) in Proposition \ref{prop:sde_strong_sol}, and let $\E_x$ denote expectation with respect to the law of $X(t)$ with initial condition $X(0)=x \in E_0$.
	\begin{enumerate}
		\item The semigroup $(P_t: B(E_0)\to B(E_0) )_{t\geq 0}$ is defined as follows: for any $x\equiv (u_a,w_a)_{a \in \mathcal{A}_0} \in E_0$ and bounded measurable test function $f \in B(E_0)$, let
		\begin{align*}
		P_tf (x) \equiv \E_x f(X(t)).
		\end{align*}
		\item The transition function $\big(P_t: E_0\times \mathcal{B}(E_0)\to [0,1]\big)_{t\geq 0}$ is defined by
		\begin{align*}
		P_t(x,A)\equiv P_t \bm{1}_A(x) = \Prob_x\big(X(t) \in A\big),\quad \forall (x,A) \in E_0\times \mathcal{B}(E_0).
		\end{align*}
		\item A measure $\mu$ on $E_0$ is called an \emph{invariant measure} associated with the semigroup $(P_t)$ if and only if $\mu = P_t^\ast \mu \equiv \int_{E_0} P_t(x, \cdot)\,\mu(\d x)$ for all $t\geq 0$.
	\end{enumerate}
\end{definition}

\begin{proposition}\label{prop:sde_invariant_measure}
	Suppose (B1) in Assumption \ref{assump:sampling_dist} holds. There exists a unique invariant probability measure $\mu_{\abs{\mathcal{A}_0}}^\ast$ on $E_0$ associated with the semigroup $(P_t)$. For $\abs{\mathcal{A}_0}=1$, $\mu_1^\ast=\delta_{\{1\}}\otimes \mathcal{N}(0,1)$.
\end{proposition}

With the uniqueness of the invariant distribution of the SDE \eqref{def:sde_time_change_renor} guaranteed by the above proposition, we may now describe the joint limiting distribution of the arm-pull counts $\{n_{a;T}\}$ and the associated noise processes.

\begin{figure}[t]
	\begin{minipage}[t]{0.495\textwidth}
		\includegraphics[width=\textwidth]{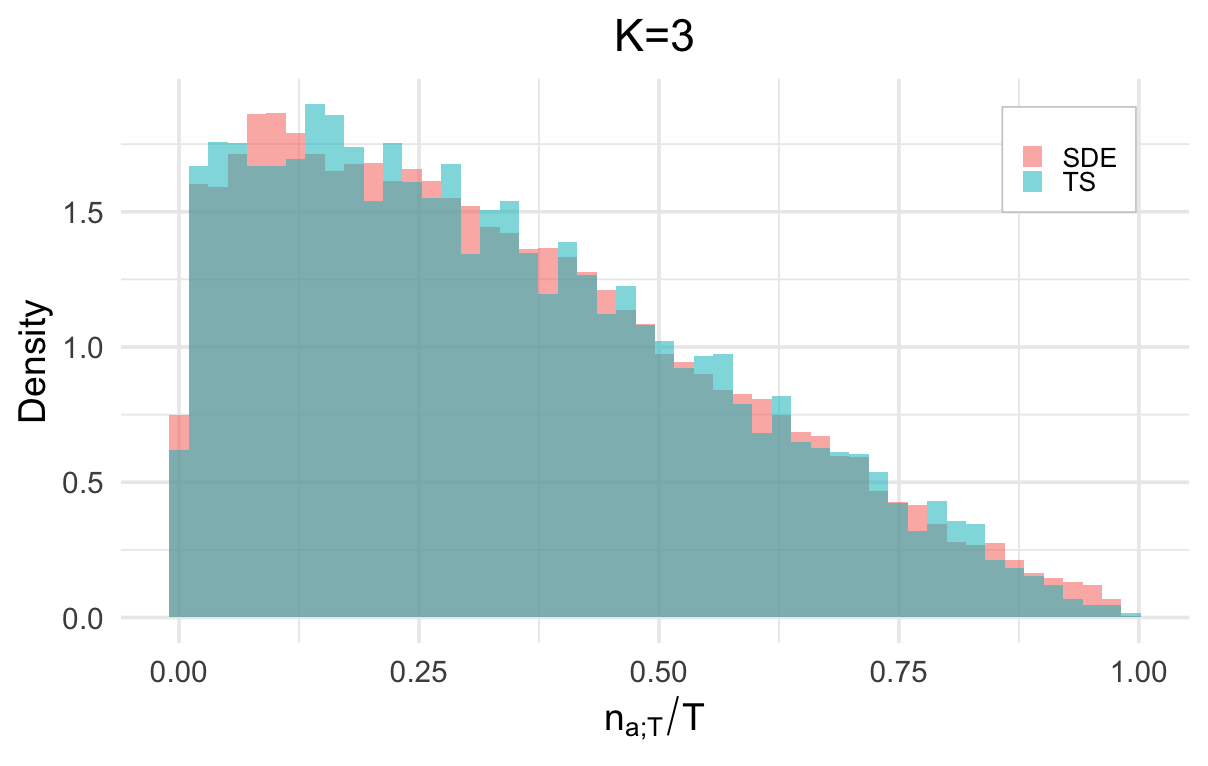}
	\end{minipage}
	\begin{minipage}[t]{0.495\textwidth}
		\includegraphics[width=\textwidth]{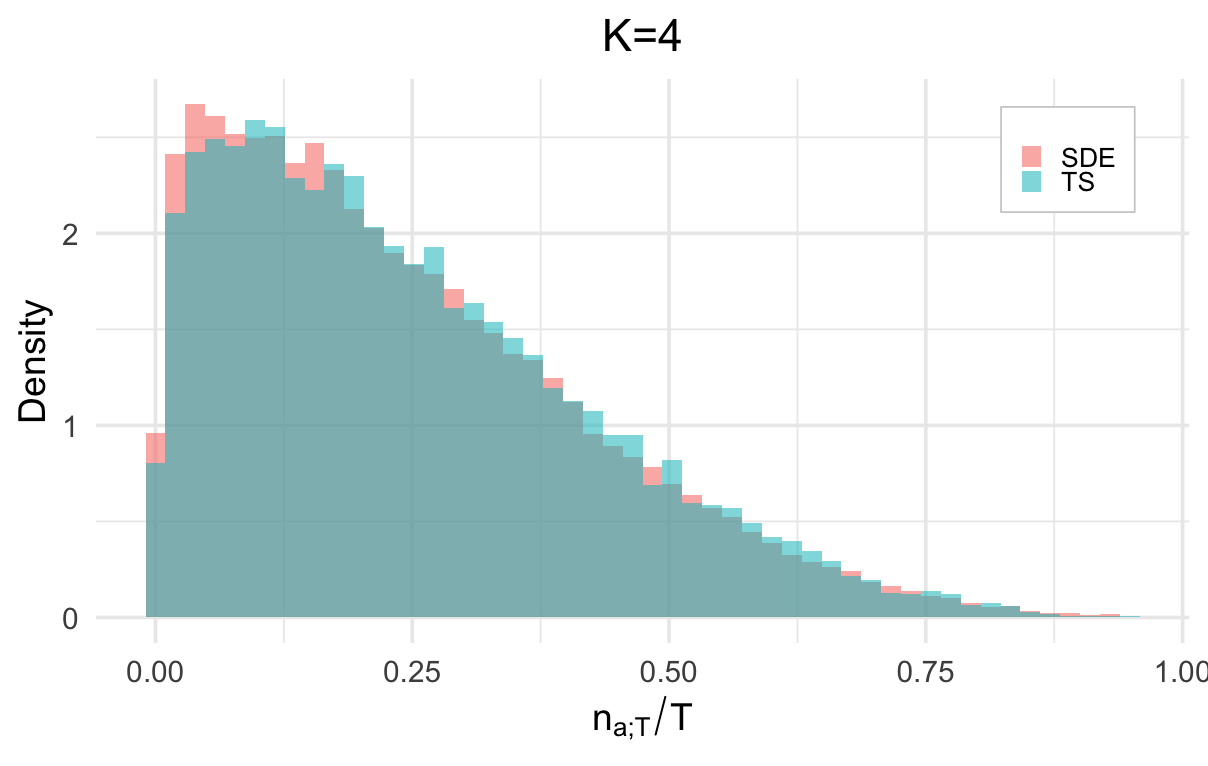}
	\end{minipage}
	\caption{Distributional comparison between $n_{a;T}/T$ and its limit distribution $\mathbb{L}_{\abs{\mathcal{A}_0};a}^{[r],\ast}$ for an optimal arm with $K=\abs{\mathcal{A}_0}=3,4$. \emph{Pink histogram}: distribution of $n_{a;T}/T$ via Gaussian Thompson sampling. \emph{Blue histogram}: distribution of $\mathbb{L}_{\abs{\mathcal{A}_0};a}^{[r],\ast}$ via SDE simulation.}
	\label{fig:arm_pull}
\end{figure}

\begin{theorem}\label{thm:optimal_arm_pull}
Suppose Assumptions \ref{assump:noise} and \ref{assump:sampling_dist} hold. Then 
\begin{align*}
\bigg(\bigg\{\frac{n_{a;T}}{T}\bigg\}_{a \in \mathcal{A}_0},\bigg\{\frac{1}{T^{1/2}}\sum_{s \in [T]}\bm{1}_{A_s=a}\xi_s\bigg\}_{a \in \mathcal{A}_0}\bigg)\rightsquigarrow \mathbb{L}_{\abs{\mathcal{A}_0}}^\ast.
\end{align*}
Here $\mathbb{L}_{\abs{\mathcal{A}_0}}^\ast=\big(\big(\mathbb{L}_{\abs{\mathcal{A}_0};a}^{[r],\ast}\big)_{a \in \mathcal{A}_0}, \big(\mathbb{L}_{\abs{\mathcal{A}_0};a}^{[\xi],\ast}\big)_{a \in \mathcal{A}_0}\big)$ is a random variable in $E_0$ with law corresponding to $\mu_{\abs{\mathcal{A}_0}}^\ast$ as specified in Proposition \ref{prop:sde_invariant_measure}.
\end{theorem}

As an illustration of the distributional approximation in the theorem above, Figure \ref{fig:arm_pull} compares the distributions of $n_{a;T}/T$ and $\mathbb{L}_{\abs{\mathcal{A}_0};a}^{[r],\ast}$ for $\abs{\mathcal{A}_0}=3,4$. The close agreement shown in Figure \ref{fig:arm_pull} persists for other values of $\abs{\mathcal{A}_0}$.

The proof of Theorem \ref{thm:optimal_arm_pull} relies on a new approach based on an `unnatural reparametrization' for the so-called arm-pull and noise processes in (\ref{def:r_xi_T}) ahead with a time change and renormalization. This reparametrization reduces the essential difficulty in formalizing the mathematical meaning of a `natural SDE' (cf. Eqn. (\ref{def:sde_singular})) with a singular initial condition, to proving the uniqueness of the invariant distribution of the SDE \eqref{def:sde_time_change_renor}. We establish this uniqueness by leveraging deep tools from analysis and probability to verify the strong Feller property and (topological) irreducibility of the semigroup $(P_t)$, including a parabolic H\"ormander condition \cite{hormander1967hypoelliptic,hairer2011malliavin} and the Stroock-Varadhan support theorem \cite{stroock1972support,benarous1994holder,millet1994simple}. We refer the reader to Section \ref{section:proof_outline_optimal} for a technical outline of our approach.

\begin{remark}
	From the SDE \eqref{def:sde_time_change_renor}, it is immediate that the law of $\mathbb{L}_{\abs{\mathcal{A}_0}}^\ast$ (i.e., $\mu_{\abs{\mathcal{A}_0}}^\ast$) in Theorem \ref{thm:optimal_arm_pull} depends only on the number of optimal arms $\abs{\mathcal{A}_0}$. Moreover,  $\big(\mathbb{L}_{\abs{\mathcal{A}_0};a}^{[r],\ast}\big)_{a \in \mathcal{A}_0}$ and $\big(\mathbb{L}_{\abs{\mathcal{A}_0};a}^{[\xi],\ast}\big)_{a \in \mathcal{A}_0}$ have identical marginal laws across $a\in\mathcal{A}_0$.
\end{remark}

\subsection{Duality between arm `stability' and noise interaction}\label{subsection:stability}

A substantial recent literature \cite{kalvit2021closer,fan2022typical,khamaru2024inference,han2024ucb,fan2024precise,fan2025statistical,halder2025stable,praharaj2025instability} has examined various notions of `stability' for a bandit algorithm $\mathscr{A}$, largely inspired by the seminal work \cite{lai1982least} in a closely related, though different, context.

To clearly distinguish potentially different stability behaviors across arms, we adopt the following arm-specific notion.

\begin{definition}\label{def:stability}
	A bandit algorithm $\mathscr{A}$ is called \emph{stable} for arm $a\in [K]$, if there exists a sequence of deterministic real numbers $\{n_{a;T}^\ast(\mathscr{A}):T \in \mathbb{N}\}$ such that the number of pulls $n_{a;T}(\mathscr{A})$ satisfies $
	{n_{a;T}(\mathscr{A})}/{n_{a;T}^\ast(\mathscr{A})}\stackrel{\Prob}{\to} 1$.
\end{definition}

For Upper Confidence Bound (UCB)-type algorithms, \cite{fan2022typical,khamaru2024inference,fan2024precise,han2024ucb,fan2025statistical,praharaj2025instability} points to a common phenomenology: once the exploration rate is moderately large so as to suppress interactions with statistical noise, UCB algorithms exhibit stability for all arms. More concretely, \cite{han2024ucb} shows that, for a variant of the canonical \texttt{UCB1} algorithm \cite{lai1987adaptive,agrawal1995sample,auer2002finite}, all arms are stable in the sense of Definition \ref{def:stability}, provided that the exploration rate satisfies $\gamma_T\gg \sqrt{\log\log T}$; see \cite[Proposition 3.7]{han2024ucb}. Under such moderately large exploration, the arm-pull counts $\{n_{a;T}\}$ for \texttt{UCB1} behave essentially as if the statistical noises $\{\xi_t\}$ in \eqref{eqn:bandit_seq_model} were absent.

The situation for the Thompson sampling scheme studied here is markedly different. Indeed, Theorems \ref{thm:suboptimal_arm_pull} and \ref{thm:optimal_arm_pull} reveal a sharp dichotomy in arm-pull behavior, and therefore:
\begin{corollary}\label{cor:stability_ts}
	An arm under the generalized Thompson sampling algorithm in Algorithm \ref{alg:ts} is stable if and only if it is either a suboptimal arm or the unique optimal arm.
\end{corollary}
At a fundamental level, the stability of suboptimal arms or the unique optimal arm in Thompson sampling is conceptually similar to the UCB setting, in that the statistical noise processes $\{\xi_t\bm{1}_{A_t=a}: t\ge 1\}$ become asymptotically negligible for such arms. In contrast, when there are multiple optimal arms, Thompson sampling is unstable because the interaction with the noise processes remains nontrivial even as $T\to\infty$.

The duality between arm stability and the effect of statistical noise also provides a unified principle for interpreting several recent stability and instability results:
\begin{enumerate}
	\item In \cite{halder2025stable}, a variance-inflated Thompson sampling scheme is proposed and shown to be stable in the two-armed setting ($K=2$) when both arms are optimal. Since variance inflation in Thompson sampling is equivalent to reducing the effective noise level in \eqref{eqn:bandit_seq_model}, the stability result in \cite{halder2025stable} can be interpreted as arising from asymptotically negligible noise interactions between optimal arms. As noted in \cite{halder2025stable}, the price for this stability is an inflated number of pulls of suboptimal arms, which deviates strictly from the Lai-Robbins lower bound \eqref{ineq:lai_robbins}.
	\item In \cite{praharaj2025instability}, a broad class of UCB-type algorithms is shown to be unstable when the exploration rate is of constant order. As observed in \cite{han2024ucb}, for \texttt{UCB1} to behave like its noiseless counterpart and remain stable, the exploration rate must stay above $\sqrt{\log\log T}$. The instability in \cite{praharaj2025instability} can therefore be understood as a consequence of nontrivial noise interactions once the exploration rate becomes too small.
\end{enumerate}

\begin{remark}
	The instability of Bernoulli Thompson sampling is observed in \cite{kalvit2021closer} in the two-armed setting when both arms are optimal, via an exact distributional characterization of the arm pulls for finite horizons $T$. It remains an open question to formulate a natural analogue of `statistical noise' that can fully characterize stability and instability for bandit algorithms with Bernoulli observations.
\end{remark}

\section{Application to adaptive inference}\label{section:adaptive_inf}

\subsection{Limit distribution of normalized empirical means}
\begin{definition}
For any $K \in \mathbb{N}$, let 
\begin{align*}
\mathscr{N}_{K}\equiv \hbox{the law of }
\begin{cases}
\mathcal{N}(0,1), & K=1;\\
\mathbb{L}_{\abs{[K]};1}^{[\xi],\ast}/\{\mathbb{L}_{\abs{[K]};1}^{[r],\ast}\}^{1/2}, & K\geq 2.
\end{cases}
\end{align*}
\end{definition}

Since the distributions $\big\{\big(\mathbb{L}_{\abs{[K]};a}^{[r],\ast},\mathbb{L}_{\abs{[K]};a}^{[\xi],\ast}\big)\big\}_{a \in [K]}$ are identical across $a\in [K]$, we may define $\mathscr{N}_{K}$ as the law of $\mathbb{L}_{\abs{[K]};a}^{[\xi],\ast}/\{\mathbb{L}_{\abs{[K]};a}^{[r],\ast}\}^{1/2}$ for any $a\in [K]$.

The following theorem describes the limiting distribution of the normalized empirical means for all arms; its proof is given in Section \ref{subsection:proof_arm_mean_dist}.

\begin{theorem}\label{thm:arm_mean_dist}
Suppose Assumptions \ref{assump:noise} and \ref{assump:sampling_dist} hold. Then
\begin{align*}
\big(n_{a;T}/\sigma^2\big)^{1/2}\cdot \big(\hat{\mu}_{a;T}-\mu_a\big) \rightsquigarrow
\begin{cases}
\mathcal{N}(0,1), & \hbox{ if } a \in \mathcal{A}_+;\\
\mathscr{N}_{ \abs{\mathcal{A}_0} }, & \hbox{ if }a \in \mathcal{A}_0.
\end{cases}
\end{align*}
\end{theorem}

\begin{figure}[t]
	\begin{minipage}[t]{0.495\textwidth}
		\includegraphics[width=\textwidth]{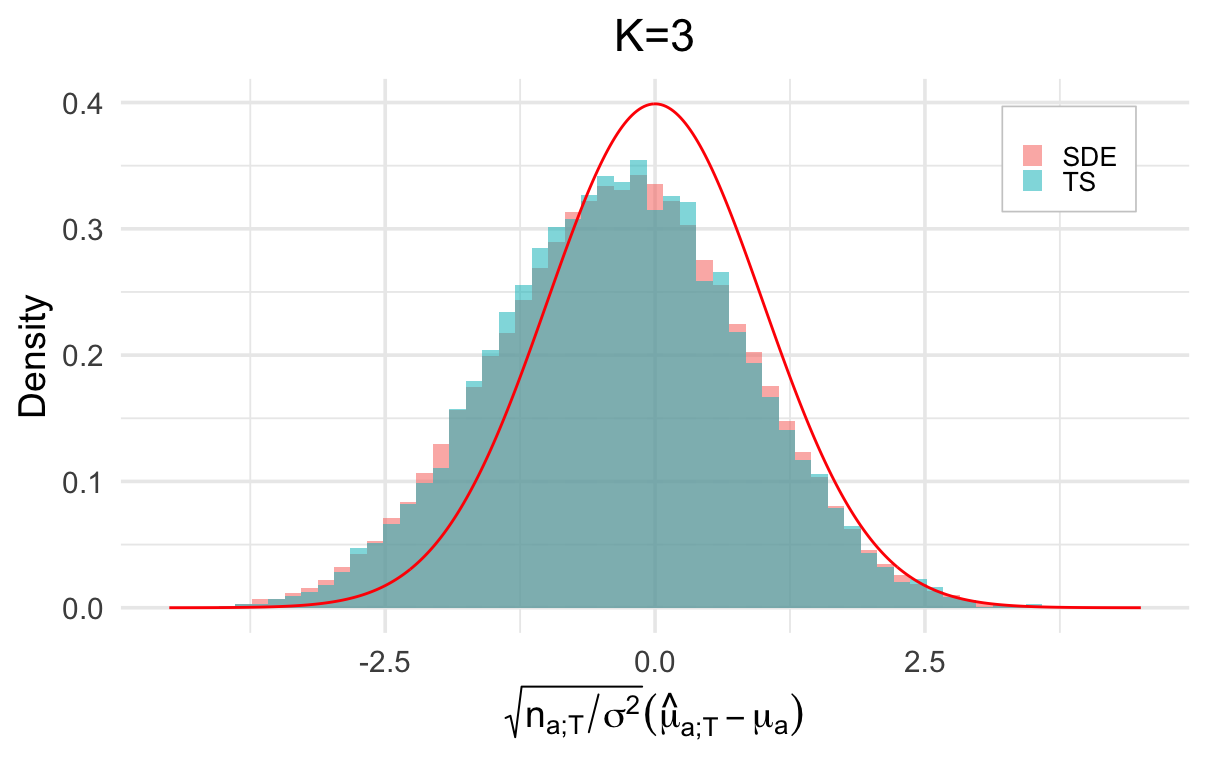}
	\end{minipage}
	\begin{minipage}[t]{0.495\textwidth}
		\includegraphics[width=\textwidth]{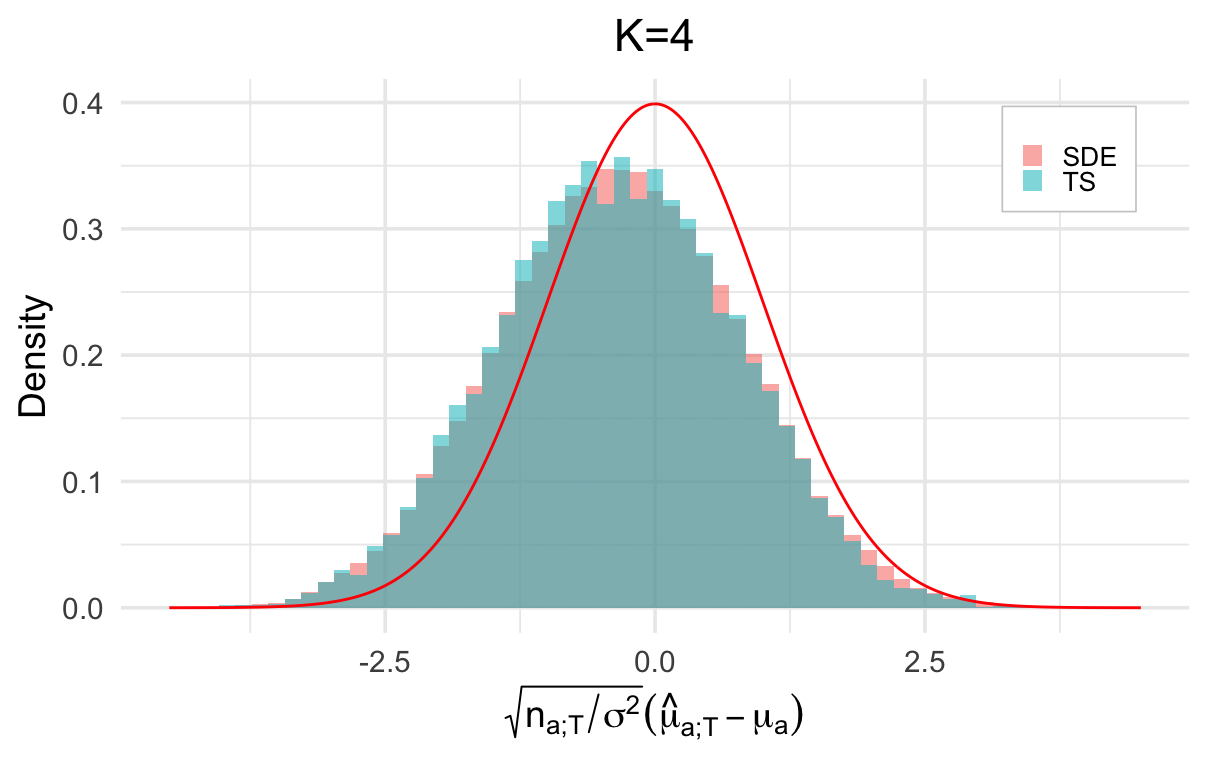}
	\end{minipage}
	\caption{Distributional comparison between $\mathscr{N}_K$ and $\mathcal{N}(0,1)$ for $K=3,4$ in Gaussian Thompson sampling. \emph{Pink histogram}: distribution of normalized empirical means. \emph{Blue histogram}: distribution of $\mathscr{N}_K$ from SDE. \emph{Red curve}: distribution of $\mathcal{N}(0,1)$.}
	\label{fig:comp_N_normal}
\end{figure}

While we do not have detailed analytic characterizations of the laws $\mathscr{N}_K$, numerical simulations suggest that they are genuinely different from $\mathcal{N}(0,1)$. Figure \ref{fig:comp_N_normal} compares the distributions of (i) the normalized empirical mean $\big(n_{a;T}/\sigma^2\big)^{1/2}\big(\hat{\mu}_{a;T}-\mu_a\big)$, (ii) the limiting distribution $\mathscr{N}_K$, and (iii) the standard normal $\mathcal{N}(0,1)$, for $K=3,4$. As seen in Figure \ref{fig:comp_N_normal}, the distributions in (i) and (ii) are very close, as predicted by Theorem \ref{thm:arm_mean_dist}, while both deviate substantially from $\mathcal{N}(0,1)$. Other values of $K\ge 2$ exhibit a similar qualitative behavior.

Theorem \ref{thm:arm_mean_dist} is closely tied to the notion of stability in Definition \ref{def:stability}. In particular, it shows that the normalized arm mean is asymptotically normal if and only if the arm is stable, and, in view of the discussion in Section \ref{subsection:stability}, if and only if the effect of statistical noise is asymptotically negligible.

From a broader perspective, the validity of asymptotic-normality-based inference has been established for stable bandit algorithms \cite{khamaru2024inference,han2024ucb,fan2024precise,fan2025statistical,halder2025stable}, while the failure of such inference has been documented in a number of unstable settings \cite{zhang2020inference,deshpande2023online,khamaru2025near}. Here Theorem \ref{thm:arm_mean_dist} provides the first example, in the context of canonical Thompson sampling, in which both the validity and invalidity of normality-based inference coexist across arms.

\subsection{Inference for the arm mean}

The limit distribution theory in Theorem \ref{thm:arm_mean_dist} can be naturally inverted to construct confidence intervals for both suboptimal and optimal arms using the empirical mean $\hat{\mu}_{a;T}$ and the arm-pull count $n_{a;T}$.

Specifically, let $z_\alpha$ and $z_\alpha(\mathscr{N}_K)$ denote the $\alpha$-quantiles of $\mathcal{N}(0,1)$ and $\mathscr{N}_K$, respectively. Consider the following $(1-\alpha)$ confidence interval (CI):
\begin{align}\label{def:CI_mean}
\mathsf{CI}_a(\alpha)\equiv 
\begin{cases}
\big[\hat{\mu}_{a;T}+ \frac{z_{\alpha/2}}{\{n_{a;T}/\sigma^2\}^{1/2}},\ \hat{\mu}_{a;T}+ \frac{z_{1-\alpha/2}}{\{n_{a;T}/\sigma^2\}^{1/2}}\big], & \text{if } a \in \mathcal{A}_+,\\[0.2em]
\big[\hat{\mu}_{a;T}+ \frac{z_{\alpha/2}(\mathscr{N}_{\abs{\mathcal{A}_0}})}{\{n_{a;T}/\sigma^2\}^{1/2}},\ \hat{\mu}_{a;T}+ \frac{z_{1-\alpha/2}(\mathscr{N}_{\abs{\mathcal{A}_0}})}{\{n_{a;T}/\sigma^2\}^{1/2}}\big], & \text{if } a \in \mathcal{A}_0.
\end{cases}
\end{align}
By Theorem \ref{thm:arm_mean_dist}, together with the observation that $\mathbb{L}_{\abs{[K]};1}^{[\xi],\ast}$ is atomless (from the second equation of the SDE \eqref{def:sde_time_change_renor}) and hence so is $\mathscr{N}_K$, the above CIs have the correct asymptotic coverage:
\begin{corollary}
	For any $\alpha \in (0,1)$, we have
	\begin{align*}
	\lim_{T \to \infty}\max_{a \in [K]}\,\abs{\Prob\big(\mu_a \in \mathsf{CI}_a(\alpha)\big)-(1-\alpha)}=0.
	\end{align*}
\end{corollary}

To construct the CIs in \eqref{def:CI_mean}, it is therefore crucial to know the quantiles $z_\alpha(\mathscr{N}_K)$. Although the laws $\{\mathscr{N}_K\}$ are generally not analytically known, their quantiles can be obtained easily by simulation. In Table \ref{table:quantile_N_K} below, we report a selected collection of these quantiles $\{z_\alpha(\mathscr{N}_K)\}$ for Gaussian Thompson sampling.
\begin{table}[ht]
	\centering
	\begin{tabular}{c|ccccccccc}
		\hline
		\hline
		$\bm{\alpha \%}$ & \textbf{2.5} & \textbf{5} & \textbf{10} & \textbf{25} & \textbf{50} & \textbf{75} & \textbf{90} & \textbf{95} & \textbf{97.5} \\
		\hline
		$\mathscr{N}_2$ & -2.57 & -2.20 & -1.76 & -1.02 & -0.22  & 0.53  & 1.19  & 1.57 & 1.90  \\
		$\mathscr{N}_3$ & -2.52 & -2.19 & -1.79 & -1.09 & -0.30 & 0.46 & 1.12 & 1.50 & 1.84 \\
		$\mathscr{N}_4$ & -2.52 & -2.18 & -1.78 & -1.10 & -0.34  & 0.42  & 1.08  &1.48  &1.82 \\
		$\mathscr{N}_5$ & -2.53 & -2.19 & -1.80 & -1.12 & -0.36  & 0.39  & 1.05  &1.45  &1.79 \\
		$\mathscr{N}_6$ & -2.51 & -2.17 & -1.78 & -1.11 & -0.37  & 0.38  & 1.05  &1.44  &1.78 \\
		\hline
		\hline
	\end{tabular}
    \caption{Simulated quantiles $z_\alpha(\mathscr{N}_K)$. }
    \label{table:quantile_N_K}
\end{table}

A notable feature of the simulated quantiles of $\mathscr{N}_K$ is that these distributions appear genuinely asymmetric. Consequently, both the lower and upper quantiles reported in Table \ref{table:quantile_N_K} should be used when constructing confidence intervals.

We also note that when the noise level $\sigma$ is unknown in \eqref{def:CI_mean}, we may replace it with any consistent estimator $\hat{\sigma}$. For instance, one may use the following averaged variance estimator over suboptimal arms:
\begin{align}\label{def:var_est}
\hat{\sigma}^2\equiv \frac{1}{\abs{\mathcal{A}_+}}\sum_{a \in \mathcal{A}_+} \bigg(\frac{1}{n_{a;T}}\sum_{t \in [T]} (R_t-\hat{\mu}_{a;t})^2\bm{1}_{A_t = a}\bigg).
\end{align}
Using a martingale argument in combination with Theorem \ref{thm:suboptimal_arm_pull}, we may prove the consistency of $\hat{\sigma}^2$; details are given in Section \ref{subsection:proof_var_consist}.
\begin{proposition}\label{prop:var_consist}
	Suppose Assumptions \ref{assump:noise} and \ref{assump:sampling_dist} hold, and
	\begin{align}\label{cond:var_consist}
	\limsup_{T \to \infty} \max_{a \in \mathcal{A}_+} \E n_{a;T}/n_{a;T}^\ast<\infty,\quad \text{where } n_{a;T}^\ast \equiv \sigma^2 [\bar{\Phi}^-(1/T)/\Delta_a]^2.
	\end{align}
	Then $\hat{\sigma}^2 \to \sigma^2$ in probability.
\end{proposition}

Condition \eqref{cond:var_consist} is satisfied by Gaussian Thompson sampling; see, e.g., \cite{agrawal2017near}. It is worth noting, however, that \eqref{cond:var_consist} is not implied by Theorem \ref{thm:suboptimal_arm_pull}. Indeed, while Theorem \ref{thm:suboptimal_arm_pull} asserts that $\max_{a \in \mathcal{A}_+}\abs{n_{a;T}/n_{a;T}^\ast-1} \stackrel{\Prob}{\to} 0$, the expectation $\E n_{a;T}$ can be of substantially larger order than $n_{a;T}^\ast$; see, for example, the discussion following Corollary \ref{cor:suboptimal_arm_rates} (in particular, after \eqref{ineq:lai_robbins}).

\subsection{Some illustrative simulations}

We now present illustrative simulations for the proposed confidence intervals in \eqref{def:CI_mean}. Specifically, we compare the CIs in \eqref{def:CI_mean} with CIs constructed by (incorrectly) using the critical values of $\mathcal{N}(0,1)$ for all arms. To highlight the main points, we assume that the noise level is known and set $\sigma=1$.

We examine the performance of these two types of CIs in the following settings:
\begin{itemize}
	\item (\emph{Setting 1}). $K=4$ with $\Delta_1=\Delta_2=0$ and $\Delta_3=0.5$, $\Delta_4=1$.
	\item (\emph{Setting 2}). $K=5$ with $\Delta_1=\Delta_2=\Delta_3=0$ and $\Delta_4=0.5$, $\Delta_5=1$.
\end{itemize}
Note that in Setting~1 (resp.\ Setting~2), there are two (resp.\ three) optimal arms. Accordingly, the CIs in \eqref{def:CI_mean} for these optimal arms use the critical values of $\mathscr{N}_2$ (resp.\ $\mathscr{N}_3$) from Table \ref{table:quantile_N_K}.

\begin{figure}[t]
	\begin{minipage}[t]{0.495\textwidth}
		\includegraphics[width=\textwidth]{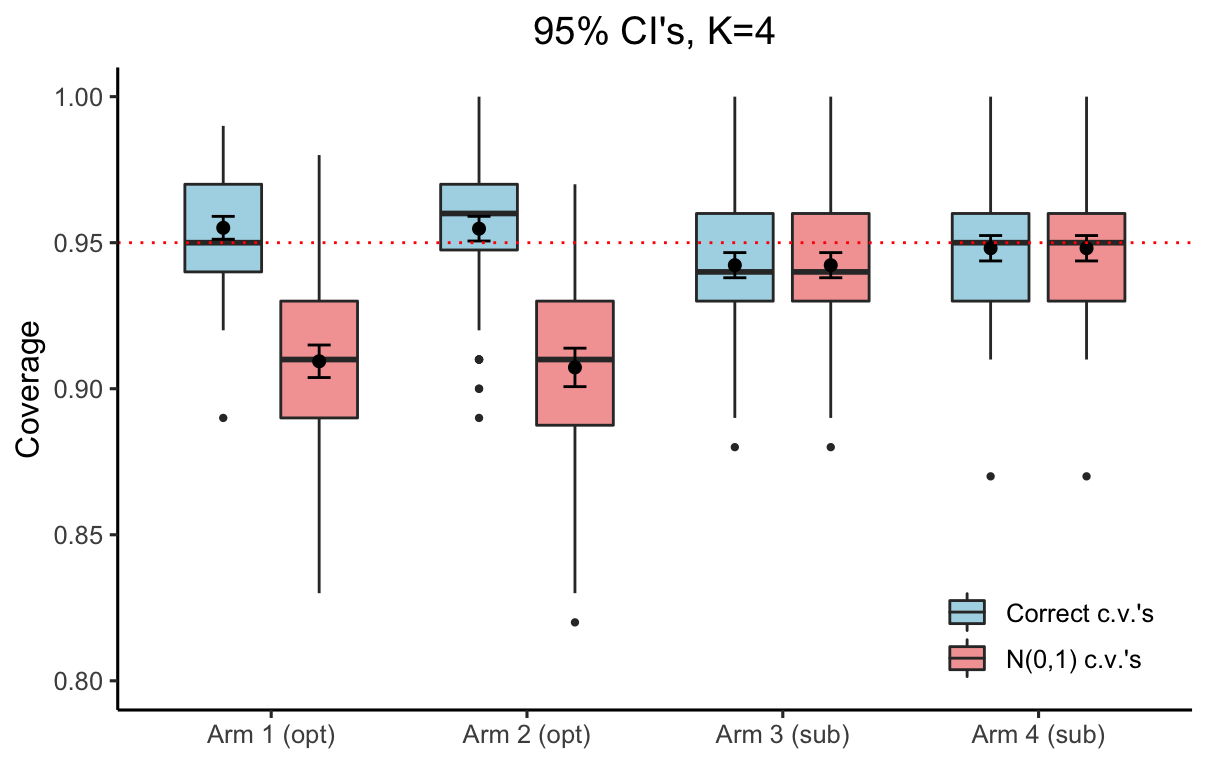}
	\end{minipage}
	\begin{minipage}[t]{0.495\textwidth}
		\includegraphics[width=\textwidth]{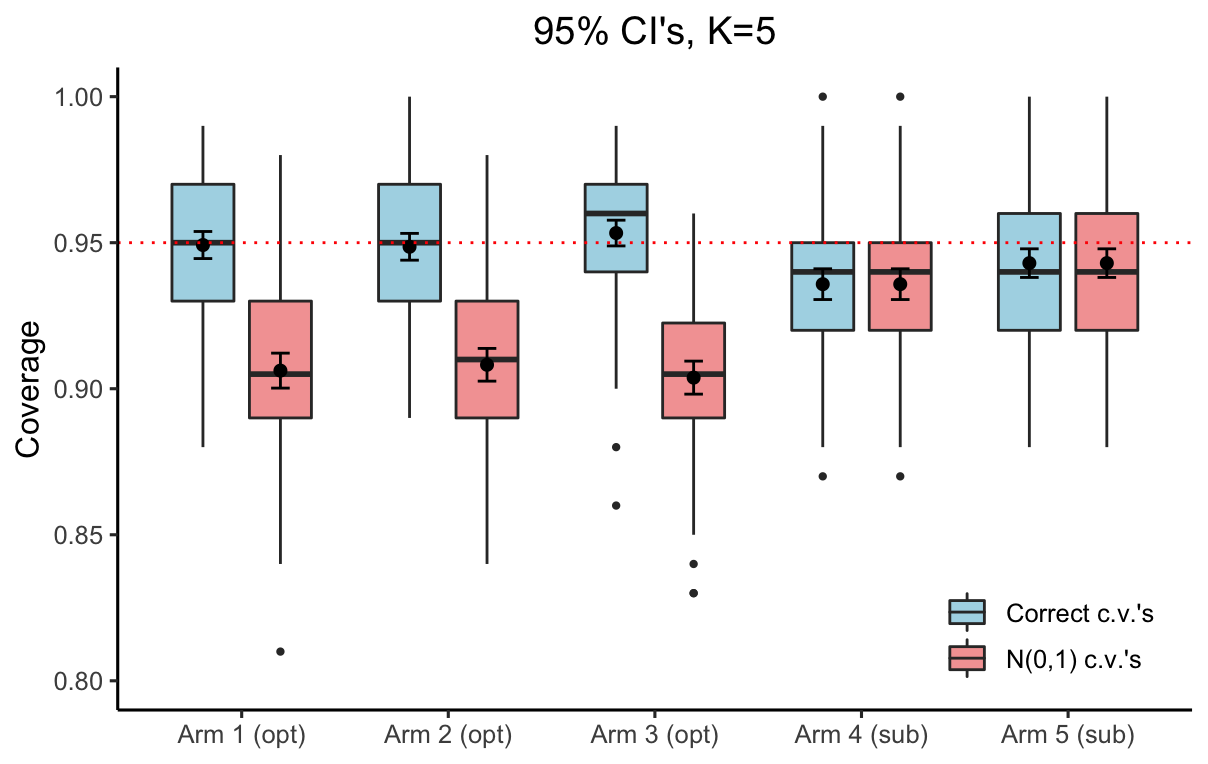}
	\end{minipage}
	\caption{Coverage of 95\% confidence intervals. \emph{Blue}: CIs using critical values from $\mathscr{N}_K$. \emph{Red}: CIs using critical values from $\mathcal{N}(0,1)$. \emph{Left panel}: $K=4$ with two optimal arms. \emph{Right panel} $K=5$ with three optimal arms.}
	\label{fig:CI}
\end{figure}

Figure \ref{fig:CI} reports the numerical results for the two types of CIs in the two settings above, with $T=2\times 10^4$. Each experiment consists of $100$ Monte Carlo replications, and the reported coverage is computed as the average of $100$ indicators of whether the true mean is contained in the CI. We observe that the CIs in \eqref{def:CI_mean}, which use the correct critical values, achieve valid coverage across all arms, whereas the CIs constructed using the (incorrect) $\mathcal{N}(0,1)$ critical values exhibit substantial under-coverage for optimal arms. These qualitative features persist across a range of settings and different numbers of optimal arms.

\section{Proof outline for Theorem \ref{thm:suboptimal_arm_pull}}\label{section:proof_outline_suboptimal}

\subsection{Some further notation}

We rewrite the reward sequence as  
\begin{align}\label{def:R_a_i}
R_{a,i} \equiv \mu_a + \sigma\cdot \xi_{a;i},\quad \forall a \in [K], \, i \in \N,
\end{align}
with $\{\xi_{a;i}\}_{a \in [K], i \in \N}\stackrel{\mathrm{i.i.d.}}{\sim} \xi_1$, so that  \eqref{eqn:bandit_seq_model} can be identified as $R_t = R_{A_t,n_{A_t;t}}$. 

With the notation in (\ref{def:R_a_i}), let us define the event
\begin{align}\label{def:event_E_xi}
E_{\xi}(x)\equiv  \bigg\{\max_{a \in [K]}\max_{t\geq 1 }\biggabs{\frac{1}{ (t+1)^{1/2}\log  \big((t+1)\vee e\big) }\sum_{i \in [t]}\xi_{a;i} }<  x\bigg\}.
\end{align}
For notational simplicity, we set $\sigma=1$ in all proofs below. The general case follows by replacing the gap $\Delta_a$ with $\Delta_a/\sigma$. 

\subsection{The inverse process approach}

\begin{definition}\label{def:hitting_time}
	Fix $a\in [K]$. For $n=0,1,2,\ldots$, we define the hitting time $
	\tau_{a;n}\equiv \inf\big\{t\geq 0:\ n_{a;t}\geq n\big\}$ 
	with the convention that $\tau_{a;0}\equiv 0$.
\end{definition}
In words, $\tau_{a;n}$ is the time at which arm $a$ is pulled for the $n$-th time. Clearly, $\tau_{a;n}$ is a stopping time with respect to $\{\mathscr{F}_t^Z\}$ (for any realization of $\{\xi_{a;i}\}$), where $\big\{\mathscr{F}_{t}^Z=\sigma(\{Z_{a;s}\}_{a \in [K],\, s \in [t]})\big\}_{t\geq 0}$ denotes the natural filtration generated by the random variables $\{Z_{a;t}\}$. Moreover, the map $n\mapsto \tau_{a;n}$ can be viewed as the inverse of $t\mapsto n_{a;t}$ in the sense that
\begin{align}\label{eqn:n_tau_equiv}
n_{a;t}\geq n\quad \Leftrightarrow\quad \tau_{a;n}\leq t.
\end{align}
The key technical advantage of working with the inverse process $n\mapsto \tau_{a;n}$ is that its asymptotic behavior can be approximately characterized as the integrator in the following approximate Stieltjes integral:
\begin{align}\label{eqn:tau_integral}
n \approx \int_0^n \bar{\Phi}\big(u^{1/2}\Delta_a\big)\,\d{\tau_a(u)},\qquad \text{as } n\to \infty.
\end{align}
Here $u\mapsto \tau_a(u)$ is the canonical right-continuous, nondecreasing step function associated with $n\mapsto \tau_{a;n}$, with $\tau_a(n)\equiv \tau_{a;n}$. 

Below we present a formal, quantitative version of \eqref{eqn:tau_integral}.

\begin{proposition}\label{prop:tau_n_char_eqn}
	Suppose Assumption \ref{assump:sampling_dist} holds. Fix $\epsilon \in (0,1/2)$. Then there exists some $c_0=c_0(K,\Delta,\epsilon,\mathscr{L}(\mathsf{Z}))>1$ such that $\Prob\big(E_\xi(c_0)\big)\geq 1-\epsilon$, and for all $\{\xi_{a;i}\} \in E_\xi(c_0)$, we have
	\begin{align*}
	&\inf_{0\leq m< n} \Prob^\xi\bigg(\max_{a \in \mathcal{A}_+}\bigg[n-m- \sum_{k \in (m:n]} (\tau_{a;k}-\tau_{a;k-1})\bar{\Phi}\big((k+\bm{1}_-)^{1/2}\Delta_{a;\epsilon}^\mp\big)\bigg]_\pm \\
	&\qquad\qquad\qquad \leq  c_0\cdot \big[\alpha_{n-m}(c_0) \bm{1}_-+\sqrt{n-m}\cdot (1\vee\abs{Y_{n,m}})\big]\bigg)\geq 1-\epsilon.
	\end{align*}
	Here  $\Delta_{a;\epsilon}^\pm\equiv (1\pm\epsilon)\Delta_a$,  $\alpha_{n}(c)\equiv n \bar{\Phi}\big(\bar{\Phi}_\ast^{-}(2n^{-1/2})/c \big)+\sqrt{n}$, and $Y_{n,m}$ is a random variable such that $\sup_{\{\xi_{a;i}\}} \E^\xi Y_{n,m}^2\leq 1$. The indicator $\bm{1}_-$ means that the term is present only for controlling the negative part. 
\end{proposition}
The key to the proof of Proposition \ref{prop:tau_n_char_eqn} is to provide two-sided estimates for the conditional probabilities
\begin{align}\label{def:cond_prob}
\mathfrak{p}_{a;t}^\xi\equiv \Prob^\xi\big(A_t= a |\mathscr{F}_{t-1}^Z\big),\quad \forall (a,t) \in [K]\times \N.
\end{align}

\begin{proposition}\label{prop:cond_prop_estimate}
	Suppose Assumption \ref{assump:sampling_dist} holds. Fix $x\geq 1$, errors $\{\xi_{a;i}\}$ such that $\{\xi_{a;i}\} \in E_{\xi}(x)$, and a suboptimal arm $a \in \mathcal{A}_+$. Then there exists some constant $c_0=c_0(K,\Delta,\mathscr{L}(\mathsf{Z}))>1$ and an event $\mathcal{E}_{x}$ with $\Prob^\xi(\mathcal{E}_{x}^c)\leq c_0\exp\big(-1/\{c_0\bar{\Phi}_\ast(c_0x^4)\}\big)$, such that for any $t\geq 2/[\bar{\Phi}_\ast(c_0x^4)]^2$, 
	\begin{align*}
	&\bar{\Phi}\big(\bar{n}_{a;t-1}^{1/2}[\Delta_a+c_0/x]\big)\cdot \big[1-\bar{\Phi}\big(\bar{\Phi}_\ast^{-}(2t^{-1/2})/(c_0 x^2)\big)\big]^K\\
	&\leq \mathfrak{p}_{a;t}^\xi\leq  \bar{\Phi}\big(\bar{n}_{a;t-1}^{1/2} [\Delta_a-c_0/x] \big)+2\Phi\big(-t^{1/2}/(c_0 x\log t)\big).
	\end{align*}
\end{proposition}

Roughly speaking, Proposition \ref{prop:cond_prop_estimate} makes rigorous the heuristic that
\begin{align}\label{ineq:cond_prop_estimate_informal}
\{A_t=a\}\approx \{\bar{\theta}_{a;t}>\mu_\ast\}.
\end{align}
Of course, \eqref{ineq:cond_prop_estimate_informal} is not literally correct, since a suboptimal arm $a\in\mathcal{A}_+$ must compete with both the other suboptimal arms and the optimal arms. Our proof of Proposition \ref{prop:cond_prop_estimate} shows that this competition becomes asymptotically negligible, so that \eqref{ineq:cond_prop_estimate_informal} is essentially correct as $t\to\infty$.

Details of the proof of Proposition \ref{prop:cond_prop_estimate} are given in Section \ref{subsection:proof_cond_prop_estimate}. We then prove Proposition \ref{prop:tau_n_char_eqn} in Section \ref{subsection:proof_tau_n_char_eqn}.

\subsection{Proof of Theorem \ref{thm:suboptimal_arm_pull}}

To use Proposition \ref{prop:tau_n_char_eqn} to prove Theorem \ref{thm:suboptimal_arm_pull}, we need to convert the estimate in Proposition \ref{prop:tau_n_char_eqn} into a two-sided bound for the inverse process $n\mapsto \tau_{a;n}$.

\begin{proposition}\label{prop:tau_n_est}
	Suppose Assumption \ref{assump:sampling_dist} holds. Fix $\epsilon \in (0,1/2)$. Then there exists some $c_0=c_0(K,\Delta,\epsilon,\mathscr{L}(\mathsf{Z}))>1$ such that $\Prob\big(E_\xi(c_0)\big)\geq 1-\epsilon$, and for all $\{\xi_{a;i}\} \in E_\xi(c_0)$ and $n\geq c_0$,
	\begin{align*}
	\Prob^\xi\Big(1/ \bar{\Phi}\big((n-c_0)_+^{1/2}\Delta_{a;\epsilon}^-\big) \leq \tau_{a;n}\leq 1/ \bar{\Phi}\big(n^{1/2}\Delta_{a;\epsilon}^+\big),\, \forall a \in \mathcal{A}_+\Big)\geq 1-\epsilon.
	\end{align*}
\end{proposition}

We note that if \eqref{eqn:tau_integral} held exactly (as an identity) for every $n$, then the inversion would be straightforward by applying \eqref{eqn:tau_integral} at two consecutive integers. The technical complication in Proposition \ref{prop:tau_n_est} arises precisely from the probabilistic sense in which \eqref{eqn:tau_integral} holds: it requires a \emph{growing} size of the interval $(m:n]$, rather than simply taking $m=n-1$. Details of the proof of Proposition \ref{prop:tau_n_est}, including how to handle this technical issue, are given in Section \ref{subsection:proof_tau_n_est}.

\begin{proof}[Proof of Theorem \ref{thm:suboptimal_arm_pull}]
	Let $c_0>0$ be the constant in Proposition \ref{prop:tau_n_est}. By the same proposition, for all $\{\xi_{a;i}\} \in E_\xi(c_0)$ and $n\geq c_0$, on an event $\mathcal{E}_{n;\epsilon}\equiv \mathcal{E}_{n;\epsilon}(\{\xi_{a;i}\})$ with $\Prob^\xi(\mathcal{E}_{n;\epsilon}^c)\leq \epsilon$, it holds that 
	\begin{align*}
	1/ \bar{\Phi}\big((n-c_0)_+^{1/2}\Delta_{a;\epsilon}^-\big) \leq \tau_{a;n}\leq 1/ \bar{\Phi}\big(n^{1/2}\Delta_{a;\epsilon}^+\big),\quad \forall a \in \mathcal{A}_+.
	\end{align*}
	For $t\geq 1$, let $n_{a;t}^\pm$ be any number satisfying $
	1/ \bar{\Phi}\big((n_{a;t}^+ -c_0)_+^{1/2}\Delta_{a;\epsilon}^-\big)  = t = 1/ \bar{\Phi}\big(n_{a;t}^{-,1/2}\Delta_{a;\epsilon}^+\big)$. 
	In particular, we may choose
	\begin{align*}
	n_{a;t}^+\equiv c_0+ \big((\Delta_{a;\epsilon}^-)^{-1} \cdot \bar{\Phi}^-(1/t) \big)^2, \quad n_{a;t}^-\equiv \big( (\Delta_{a;\epsilon}^+)^{-1} \cdot \bar{\Phi}^-(1/t) \big)^2.
	\end{align*}
	By (\ref{eqn:n_tau_equiv}), we have $n_{a;t}^-\leq n_{a;t}\leq n_{a;t}^+$. The claim now follows from Lemma \ref{lem:max_err}.
\end{proof}

\begin{remark}
For continuous, non-increasing $\bar{\Phi}$ satisfying the regularly varying condition in Assumption~B2, we define $\bar{\Phi}^+(u)\equiv \sup\{x\ge 0:\ \bar{\Phi}(x)=u\}$. It is then easy to show that $\bar{\Phi}^+(1/T)/\bar{\Phi}^-(1/T)\to 1$ as $T\to\infty$. Therefore, in the proof above we may choose the tighter lower bound $
n_{a;t}^-\equiv \big( (\Delta_{a;\epsilon}^+)^{-1}\cdot \bar{\Phi}^+(1/t)\big)^2$, which makes no effective difference for large values of $t$.
\end{remark}

\section{Proof outline for Theorem \ref{thm:optimal_arm_pull}}\label{section:proof_outline_optimal}

\subsection{The natural parametrization: What breaks down?}

As mentioned above, the proof of Theorem \ref{thm:optimal_arm_pull} relies on an `unnatural reparametrization' of the arm-pull and noise processes
\begin{align}\label{def:r_xi_T}
\begin{cases}
r_{a;T}(z)\equiv T^{-1}\big(\sum_{s=1}^{ \floor{z T}} \bm{1}_{A_s=a}+1\big),\\
\xi_{a;T}(z)\equiv T^{-1/2}\sum_{s=1}^{ \floor{z T}} \bm{1}_{A_s=a} \xi_s,
\end{cases}
\forall a\in\mathcal{A}_0,\, z \in [0,1].
\end{align}
To highlight the intrinsic difficulties of the `natural parametrization' \eqref{def:r_xi_T}, we first state two lemmas that provide approximate characterizations of the arm-pull process $\{r_{a;T}(\cdot)\}$ and the noise process $\{\xi_{a;T}(\cdot)\}$ via a self-consistent equation and a martingale property, respectively.

\begin{lemma}\label{lem:r_lim}
	Suppose Assumption \ref{assump:sampling_dist} holds. Fix $\epsilon \in (0,1/2]$ and $x\geq 1$. There exist some constants $c_1=c_1(K,\Delta,\mathscr{L}(\mathsf{Z}))>0$ and $c_2=c_2(\epsilon,x,K,\Delta,\mathscr{L}(\mathsf{Z}))>0$ such that for any $\{\xi_{a;i}\} \in E_\xi(x)$, if $T\geq c_2$, 
	\begin{align*}
	&\Prob^\xi\bigg(\sup_{z \in [\epsilon,1]}\biggabs{r_{a;T}(z)-r_{a;T}(\epsilon)- \frac{1}{T}\sum_{s=\floor{\epsilon T}+1}^{ \floor{z T}} p_{a}\big(r_{\mathcal{A}_0;T}(z_{s-1}^{(T)}), \xi_{\mathcal{A}_0;T}(z_{s-1}^{(T)})\big) } \\
	&\qquad \geq 2\Prob\big(\abs{\mathsf{Z}}>\{\bar{\Phi}_\ast^{-}(1\wedge c_2T^{-1/2})\}^{1/2}/c_2 \big)+\epsilon_T\bigg)\leq c_1\cdot  \big[e^{-1/\{c_1\bar{\Phi}_\ast(c_1x)\}}+\epsilon_T^{100}\big].
	\end{align*}
	Here $z_{t}^{(T)}\equiv t/T$, $\epsilon_T\equiv \sqrt{\log T/T}$, $\bar{\Phi}_\ast$ is defined in (\ref{def:Phi_Psi_ast}), and recall that $p_a$ is defined in (\ref{def:p_a}). The above estimate also holds when $\Delta_a\equiv 0$ for all $a \in [K]$.
\end{lemma}

Let $\big\{\mathscr{F}_{t}\equiv\sigma(\{\xi_{s_1},Z_{a;s_2}\}_{a \in [K], s_1 \in [t], s_2 \in [t+1]})\big\}_{t\geq 0}$ denote a nested filtration generated by the random variables $\{\xi_{[t]},Z_{[K];[t+1]}\}$.

\begin{lemma}\label{lem:xi_mtg}
	Suppose Assumption \ref{assump:noise} holds. With $\{\xi_{a;T}\}$ defined in (\ref{def:r_xi_T}) and $\bar{r}_{a;T}(z)\equiv T^{-1}\sum_{s=1}^{ \floor{z T}} \bm{1}_{A_s=a}$, the process $\{
	\xi_{a;T}^2(z)-\bar{r}_{a;T}(z): z \in [0,1]\}$ is a martingale with respect to the filtration $\{\mathscr{F}_{\floor{zT}}: z \in [0,1]\}$.
\end{lemma}
The proofs of the above two lemmas are given in Sections \ref{subsection:proof_r_lim} and \ref{subsection:proof_xi_mtg}, respectively.

Formally taking the limit $T\to\infty$ in Lemmas \ref{lem:r_lim} and \ref{lem:xi_mtg}, while ignoring technical details, suggests considering the following SDE associated with the natural parametrization \eqref{def:r_xi_T}:
\begin{align}\label{def:sde_singular}
\begin{cases}
\d r_a(z)= p_{a}\big(r_{\cdot}(z), \xi_{\cdot}(z)\big)\,\d{z},\\
\d \xi_a(z) = \sqrt{p_{a}\big(r_{\cdot}(z), \xi_{\cdot}(z)\big)}\,\d{B_a(z)},
\end{cases}
 \forall a \in \mathcal{A}_0,\ z \in [0,1].
\end{align}
This strategy has been adopted in \cite{kuang2024weak,fan2025diffusion} in settings more general than Algorithm \ref{alg:ts}, under the crucial assumption that $\{p_a\}$ are globally Lipschitz. Under this Lipschitz condition, the SDE (\ref{def:sde_singular}) admits a unique strong solution by classical theory; see, e.g., \cite{karatzas1991brownian,revuz1999continuous,oksendal2003stochastic}.

Unfortunately, this simple and natural approach breaks down in our setting because the initial condition $r_\cdot(0)=0$ induces an essential singularity in \eqref{def:sde_singular} at time $0$. Indeed, the functions $p_a(r_\cdot,\xi_\cdot)$ are not well defined when $r_\cdot(0)=0$. Equivalently, under the natural parametrization \eqref{def:r_xi_T}, the limiting SDE \eqref{def:sde_singular} involves a singular initial condition that is incompatible with its dynamics.

\subsection{Time change, renormalization, and compact convergence}

To circumvent the singularity arising from the natural parametrization \eqref{def:r_xi_T}, we instead consider a time-changed and renormalized version of $\big(r_{a;T}(z),\xi_{a;T}(z)\big)$. For $z\in(0,1)$, with $t=\log z\in(-\infty,0]$, define
\begin{align}\label{def:u_w}
\begin{cases}
u_{a;T}(t)\equiv  e^{-t}\cdot r_{a;T}(e^t),\\ 
w_{a;T}(t)\equiv e^{-t/2}\cdot \xi_{a;T}(e^t),
\end{cases}
\forall a \in \mathcal{A}_0,\, t \in (-\infty,0].
\end{align}
The key technical advantage of working with this somewhat unnatural reparametrization $(u_{\cdot;T},w_{\cdot;T})$ in \eqref{def:u_w} is that it converges on any compact set to a process whose marginal laws are stationary in time and, in fact, form an invariant distribution associated with the semigroup $(P_t)_{t\ge 0}$ of the SDE \eqref{def:sde_time_change_renor}.

\begin{proposition}\label{prop:compact_conv_sde}
	Suppose Assumptions \ref{assump:noise} and \ref{assump:sampling_dist} hold. For any $T_0>0$, the sequence $\{(u_{a;T},w_{a;T}):a \in \mathcal{A}_0\}_{T\geq 1}$ is tight in $(\ell^\infty[-T_0,0])^{\mathcal{A}_0\times \mathcal{A}_0}$, and the sequential limit $\{(u_{a},w_{a}):a \in \mathcal{A}_0\}$ satisfies the SDE (\ref{def:sde_time_change_renor}) on $[-T_0,0]$ with the same marginal law that constitutes an invariant probability measure of the semigroup $(P_t)_{t\geq 0}$. 
\end{proposition}
The proof of Proposition \ref{prop:compact_conv_sde} can be found in Section \ref{subsection:proof_compact_conv_sde}.

\subsection{Uniqueness of the invariant measure: Proof outline of Proposition \ref{prop:sde_invariant_measure}}

Given the compact convergence in Proposition \ref{prop:compact_conv_sde}, the main remaining task is to establish the uniqueness of the invariant distribution as stated in Proposition \ref{prop:sde_invariant_measure}.

To this end, we first define the notions of \emph{strong Feller} and \emph{irreducibility} properties associated with the semigroup $(P_t)$.
\begin{definition}\label{def:strong_feller_irreducible}
	\begin{enumerate}
		\item[(\texttt{P}1)] The semigroup $(P_t)$ is called \emph{strongly Feller} if $P_t(B(E_0))\subset C_b(E_0)$ for all $t\geq 0$.
		\item[(\texttt{P}2)] The semigroup $(P_t)$ is called \emph{irreducible} if, for any $x\in E_0$ and any nonempty open set $O\subset E_0$, there exists some $t\geq 0$ such that $P_t(x,O)>0$.
	\end{enumerate}
\end{definition}

The main tool for establishing uniqueness of the invariant distribution for the SDE \eqref{def:sde_time_change_renor} is the following version of Doob's theorem, tailored to our setting.

\begin{theorem}\label{thm:Doob}
	If the semigroup $(P_t)$ is strongly Feller \textup{(\texttt{P}1)} and irreducible \textup{(\texttt{P}2)}, then it has at most one invariant measure.
\end{theorem}

For the reader's convenience, we provide a complete proof of Theorem \ref{thm:Doob} in Section \ref{subsection:proof_Doob} by verifying the uniqueness criteria in \cite{hairer2008ergodic}.

\subsubsection{SDE in Stratonovich form}

To establish properties \textup{(\texttt{P}1)}-\textup{(\texttt{P}2)} required by Theorem \ref{thm:Doob}, we verify a suitable parabolic H\"ormander condition \cite{hormander1967hypoelliptic} and a topological irreducibility property via the Stroock-Varadhan support theorem \cite{stroock1972support} for the SDE \eqref{def:sde_time_change_renor}. 

Since both arguments require rewriting the SDE \eqref{def:sde_time_change_renor} in Stratonovich form, we consider the following standard Stratonovich SDE for $X(\cdot)\in \R^n$:
\begin{align}\label{def:hormander_sde}
\d X(t) = V_0(X(t))\,\d t+\sum_{i \in [m]} V_i(X(t))\circ \d B_i(t).
\end{align}
Here $V_i$'s are vector fields on $\R^n$ (identified as elements of $C^\infty(\R^n;\R^n)$), and $B_i$'s are independent one-dimensional Brownian motions.

\begin{lemma}\label{lem:Ito_Stratonovich}
	The Stratonovich form of the SDE in (\ref{def:sde_time_change_renor}) is given by
	\begin{align}\label{def:sde_time_change_renor_Stratonovich}
	\begin{cases}
	\d{u_a}(t) = \big(p_a(u_\cdot(t),w_\cdot(t))-u_a(t)\big)\,\d{t},\\
	\d{w_a}(t) = \big(-\frac{1}{2} w_a(t)-\frac{1}{4}\partial_{w_a} p_a(u_\cdot(t),w_\cdot(t)) \big)\,\d{t}\\
	\qquad \qquad\qquad +\sqrt{p_a(u_\cdot(t),w_\cdot(t))}\circ \d{B_a(t)},
	\end{cases}
	\forall a \in \mathcal{A}_0,\, t \in \R.
	\end{align}
\end{lemma}

The proof of the above lemma can be found in Section \ref{subsection:proof_Ito_Stratonovich}.

\subsubsection{Strong Feller via parabolic H\"ormander condition}
This subsection requires some working knowledge of Lie brackets. For readers unfamiliar with Lie brackets, we summarize some basic notions in Appendix \ref{section:Lie_bracket}.

The following definition of the parabolic H\"ormander condition is taken from \cite[Definition 1.2]{hairer2011malliavin}.

\begin{definition}\label{def:hormander_cond}
	Consider the SDE for $X(\cdot) \in \R^n$ in Stratonovich form (\ref{def:hormander_sde}). Let a nested collection of vector fields $\mathscr{V}_0\subset\mathscr{V}_1\subset \cdots$ be defined recursively as follows. Let $\mathscr{V}_0\equiv \{V_1,\ldots,V_m\}$, and for $k\geq 1$, let
	\begin{align*}
	\mathscr{V}_k \equiv \mathscr{V}_{k-1} \cup \big\{[U,V_j]: U \in \mathscr{V}_{k-1}, j=0,1,\ldots,m\big\}.
	\end{align*} 
	For any $x \in \R^n$, define the vector space $\mathscr{V}_k(x)\equiv \mathrm{span}\{V(x): V \in \mathscr{V}_k\}$. We say (\ref{def:hormander_sde}) satisfies the \emph{parabolic H\"ormander condition}, if
	\begin{align}\label{cond:hormander}
	\cup_{k\geq 1} \mathscr{V}_k(x) = \R^n,\quad \hbox{for all $x \in \R^n$}.
	\end{align}
\end{definition}

H\"ormander's theorem \cite{hormander1967hypoelliptic} was originally formulated in terms of second-order differential operators in partial differential equations; see also \cite{bakry2014analysis}. A probabilistic proof in the context of stochastic differential equations became possible with the development of Malliavin calculus \cite{malliavin1978stochastic}, and was subsequently further simplified in \cite{kusuoka1984applications,kusuoka1985applications,kusuoka1987applications}. Below we state a version of H\"ormander's theorem, taken from \cite[Theorem 1.3]{hairer2011malliavin}, which guarantees the strong Feller property for the semigroup associated with the SDE \eqref{def:hormander_sde}.

\begin{theorem}\label{thm:hormander}
	Consider the Stratonovich SDE (\ref{def:hormander_sde}) and assume that all vector fields $\{V_0,V_1,\ldots,V_m\}$ have bounded derivatives of all orders. If it satisfies the parabolic H\"ormander condition (\ref{cond:hormander}), then its solutions admit a smooth density with respect to Lebesgue measure, and the corresponding Markov semigroup maps bounded functions into smooth functions at any time.
\end{theorem}

The above version of H\"ormander's theorem will be the basis for proving the following proposition.
\begin{proposition}\label{prop:sde_strong_feller}
	Suppose (B1) in Assumption \ref{assump:sampling_dist} holds and $\abs{\mathcal{A}_0}\geq 2$. Then the semigroup $(P_t)$ in Definition \ref{def:semigroup} is strongly Feller, i.e., for any $t\geq 0$ and $f \in B(E_0)$, the map $x\mapsto P_{t} f(x)$ is continuous on $E_0$.
\end{proposition}

The proof of the above proposition is technically involved. At a high level, we show that the parabolic H\"ormander condition in \eqref{cond:hormander} can be verified for the Stratonovich form of the SDE \eqref{def:sde_time_change_renor_Stratonovich} by restricting attention to the vector fields $\mathscr{V}_0$ and $\mathscr{V}_1$, without the need to consider higher-order iterated Lie brackets $\{\mathscr{V}_k\}_{k\ge 2}$. A key insight in our calculations is that an appropriate version of the matrix $\big(\partial_{w_a}p_b(u_\cdot,w_\cdot)\big)_{a,b\in\mathcal{A}_0}$ is rank-deficient, with rank exactly $\abs{\mathcal{A}_0}-1$. The proof is further complicated by the fact that $p_\cdot(u_\cdot,w_\cdot)$ becomes singular for small values of $u_\cdot$, so localization techniques are required.

Full details of the proof of Proposition \ref{prop:sde_strong_feller} are provided in Section \ref{subsection:proof_sde_strong_feller}.

\subsubsection{Irreducibility via support theorem}
The basic tool for proving irreducibility of the semigroup $(P_t)$ is a version of the Stroock-Varadhan support theorem \cite{stroock1972support}. To state the result, for $\alpha>0$ and $T>0$, we recall the space $C^{\alpha}([0,T]\to \R^n)$ of $\alpha$-H\"older continuous functions with finite $\alpha$-H\"older norm
\begin{align*}
\pnorm{f}{\alpha}\equiv \sup_{t \in [0,T]} \pnorm{f(t)}{}+\sup_{t\neq s \in [0,T]} \frac{\pnorm{f(t)-f(s)}{}}{\abs{t-s}^\alpha}<\infty.
\end{align*}
We usually omit the dependence on $T$ in the notation.

Given a probability measure $\mu$ on $(C^\alpha,\pnorm{\cdot}{\alpha})$, its support is defined by
\begin{align*}
\mathrm{supp}_\alpha(\mu)\equiv \bigcap\Big\{F\subset C^\alpha:\ F \text{ is closed under }\pnorm{\cdot}{\alpha}\ \text{and}\ \mu(F)=1\Big\}.
\end{align*}
Equivalently, $
\mathrm{supp}_\alpha(\mu)=\big\{f \in C^\alpha:\ \mu\big(B_\alpha(f,r)\big)>0\ \text{for all } r>0\big\}$, 
where the norm ball is $B_\alpha(f,r)\equiv \{g \in C^\alpha:\ \pnorm{g-f}{\alpha}<r\}$.

For smooth vector fields $\{V_0,V_1,\ldots,V_m\}$ with bounded derivatives of all orders, the solution $X(\cdot)$ to the Stratonovich SDE \eqref{def:hormander_sde} can be viewed as a random element of $C^\alpha([0,T]\to \R^n)$ for any $T$, and hence $\Prob_x\circ X^{-1}$ defines a probability measure on $(C^\alpha,\pnorm{\cdot}{\alpha})$.

The following support theorem, which characterizes $\mathrm{supp}_\alpha(\Prob_x\circ X^{-1})$, is taken from \cite[Theorem 4]{benarous1994holder}; see also \cite[Theorem 3.5]{millet1994simple}.

\begin{theorem}\label{thm:support}
	Consider the Stratonovich SDE (\ref{def:hormander_sde}) and assume that all vector fields $\{V_0,V_1,\ldots,V_m\}$ have bounded derivatives of all orders. Fix $x \in \R^n$ and $T>0$. Let $\mathscr{S}_x$ be the map which associates to $h \in L^2\equiv L^2([0,T]\to \R^n)$ the solution $Y(\cdot)\in \R^n$ to the following ordinary differential equation (ODE):
	\begin{align}\label{def:skeleton_ODE}
	\d Y(t)=V_0(Y(t))\,\d{t}+\sum_{i \in [m]} V_i(Y(t)) h_i(t)\,\d{t},\quad Y(0)=x.
	\end{align}
	Then for any $\alpha \in [0,1/2)$,
	\begin{align*}
	\mathrm{supp}_\alpha(\Prob_x\circ X^{-1}) =\hbox{closure of }\mathscr{S}_x(L_2)\hbox{ in } (C^\alpha,\pnorm{\cdot}{\alpha}). 
	\end{align*}
\end{theorem}

We note that \cite[Theorem 4]{benarous1994holder} does not specify precise regularity conditions on $\{V_i\}$. Here we adopt a condition that is stronger than necessary (and stronger than that in \cite[Theorem 3.5]{millet1994simple}) to keep the presentation simple and consistent with Theorem \ref{thm:hormander}.

The version of the support theorem stated in Theorem \ref{thm:support} will be used to prove the following proposition.

\begin{proposition}\label{prop:sde_irreducible}
	Suppose (B1) in Assumption \ref{assump:sampling_dist} holds and $\abs{\mathcal{A}_0}\geq 2$. Then the semigroup $(P_t)$ in Definition \ref{def:semigroup} is irreducible, i.e., with $X(t)$ denoting the unique strong solution to the SDE (\ref{def:sde_time_change_renor}), for any $x \in E_0$ and any nonempty open set $O\subset E_0$, there exists some $t>0$ such that $\Prob_x(X(t) \in O)>0$.
\end{proposition}

The proof of the above proposition is based on an explicit construction of a square integrable function $h$ such that the ODE solution $\mathscr{S}_x(h)$ to \eqref{def:skeleton_ODE} approximately reaches a prescribed interior point of $O$ at some time $t$, thereby enabling an application of the support theorem in Theorem \ref{thm:support}. Given an initial point $x=(\bar{u}_x,\bar{w}_x)\in E_0$ and a target point $x_0=(\bar{u}_0,\bar{w}_0)\in O\subset E_0$, we construct $h(\cdot)$ so that the corresponding ODE solution $\mathscr{S}_x(h)(\cdot)=(\bar{u}(\cdot),\bar{w}(\cdot))$ proceeds in two stages:
\begin{itemize}
	\item (\emph{Stage I}). We first show that there exist a time $t_0$ and a choice of $h$ on $[0,t_0]$ such that $\bar{u}(\cdot)$ moves from its initial value $\bar{u}_x$ to a neighborhood of $\bar{u}_0$.
	\item (\emph{Stage II}). We then show that $\bar{w}(t_0)$ can be moved to the target value $\bar{w}_0$ by concatenating $h$ with an additional control over a very \emph{short} time interval, so that $\bar{u}(\cdot)$ remains close to the target value $\bar{u}_0$.
\end{itemize}
Details of the construction of $h$ and the proof of Proposition \ref{prop:sde_irreducible} are given in Section \ref{subsection:proof_sde_irreducible}.

\subsubsection{Proof of Proposition \ref{prop:sde_invariant_measure}}

For $\abs{\mathcal{A}_0}\geq 2$, the existence of an invariant probability measure is part of the claim in Proposition \ref{prop:compact_conv_sde}. The uniqueness claim in Proposition \ref{prop:sde_invariant_measure} then follows from Theorem \ref{thm:Doob} with the help of Propositions \ref{prop:sde_strong_feller} and \ref{prop:sde_irreducible}. 

For $\mathcal{A}_0=1$, as $u(\cdot)$ remains at $1$, we only need to consider the invariant measure corresponding to the SDE $\d{w(t)} = -\frac{1}{2}w(t)+\d{B(t)}$, whose law is well known to be uniquely determined as $\mathcal{N}(0,1)$; see, e.g., Lemma \ref{lem:sde_1d_inv}.\qed

\subsection{Proof of Theorem \ref{thm:optimal_arm_pull}}

The claim follows by combining Proposition \ref{prop:compact_conv_sde} with $t=0$ therein and Proposition \ref{prop:sde_invariant_measure}. \qed

\section{Proofs for Section \ref{section:proof_outline_suboptimal}}\label{section:proof_suboptimal}

We need some further notation and some basic facts related to Algorithm \ref{alg:ts}:
\begin{itemize}
	\item With $\mathsf{Z}'$ denoting an independent copy of the sampling variable $\mathsf{Z}$, let
	\begin{align}\label{def:Phi_Psi_ast}
	\begin{alignedat}{2}
	\Phi_\ast(z) &\equiv \sup_{(a,a')\in\partial B_2(1)} \Prob(a\mathsf{Z}+a'\mathsf{Z}'\le z),
	&\qquad \bar{\Phi}_\ast(z) &\equiv 1-\Phi_\ast(z);\\
	\Psi_\ast(z) &\equiv \inf_{(a,a')\in\partial B_2(1)} \Prob(a\mathsf{Z}+a'\mathsf{Z}'\le z),
	&\qquad \bar{\Psi}_\ast(z) &\equiv 1-\Psi_\ast(z).
	\end{alignedat}
	\end{align}
	\item $\bar{n}_{a;t},\bar{\mu}_{a;t}$ can be rewritten alternatively as 
	\begin{align*}
	\bar{n}_{a;t}= n_{a;t}+1,\quad \bar{\mu}_{a;t}= \frac{1}{\bar{n}_{a;t}}\sum_{s \in [t]} \bm{1}_{A_s=a} R_s = \frac{n_{a;t}}{\bar{n}_{a;t} }\cdot \mu_a+ \frac{1}{\bar{n}_{a;t}}\sum_{s \in [t]} \bm{1}_{A_s=a}\xi_s.
	\end{align*}
	\item The arm selection can be written as $A_t \in \argmax_{a \in [K]} \bar{\theta}_{a;t}$ with
	\begin{align}\label{def:theta}
	\bar{\theta}_{a;t}\equiv \bar{\mu}_{a;t-1}+ \frac{\sigma }{  \sqrt{\bar{n}_{a;t-1}}} Z_{a;t},\quad \sigma\equiv 1.
	\end{align}
\end{itemize}
Some other notation will be needed:
\begin{itemize}
\item Recall the filtration $\big\{\mathscr{F}_{t}^Z=\sigma(\{Z_{a;s}\}_{a \in [K],\, s \in [t]})\big\}_{t\geq 0}$.
\item Let $\Delta_{\min}\equiv \min_{a \in \mathcal{A}_+}\Delta_a$ and $\Delta_{\max} \equiv \max_{a \in \mathcal{A}_+}\Delta_a$.
\end{itemize}

\subsection{Technical lemmas}

\begin{lemma}\label{lem:tau_property}
	Suppose that $\Phi_\ast(z)<1$ holds for any $z \in \R$. Then for any $a \in [K]$, $n\geq 0$, and any $\{\xi_{a;i}\}\in E_\xi(\infty)$, we have $
	\Prob^\xi\big(\tau_{a;n}<\infty\big)= \Prob^\xi\big(n_{a;\tau_{a;n}} = n\big)=1$. 
\end{lemma}
\begin{proof}
	It suffices to prove that $\tau_{a;n}<\infty$ holds $\Prob^\xi$-almost surely. To this end, note that for any fixed integer $t\geq n$ and $x>1$, on the event $E_\xi(x)$,
	\begin{align*}
	&\Prob^\xi\big(\tau_{a;n}> t\big)= \Prob^\xi\big(n_{a;t}<n\big)\\
	&\leq \sum_{\mathcal{T}\subset [t]: \abs{[t]\setminus \mathcal{T}}< n  } \Prob^\xi\Big(\bar{\theta}_{a;s}<\max_{b\neq a} \bar{\theta}_{b;s}\, \hbox{ for all }s \in \mathcal{T}\Big)\\
	&\leq \sum_{b\neq a}\sum_{\mathcal{T}\subset [t]: \abs{[t]\setminus \mathcal{T}}< n  }  \Prob^\xi\bigg(Z_{a;s}\leq  n^{1/2}(2x+\Delta_{\max})+\bigg(\frac{\bar{n}_{a;s-1}}{\bar{n}_{b;s-1}}\bigg)^{1/2}Z_{b;s}\,\, \hbox{for all $s\in \mathcal{T}$}\bigg)\\
	&\leq K\cdot \binom{t}{n-1}\cdot \big[\Phi_\ast( n^{1/2}(2x+\Delta_{\max}) )\big]^{t-n}.
	\end{align*}
	Consequently, for $t\geq n$ and $x> 1$, we have for some $c_1=c_1(n,K,\Delta)>0$, $
	\Prob^\xi\big(\tau_{a;n}> t\big)\leq c_1 \cdot t^n\cdot \Phi_\ast^t(c_1 x)$. 
	Sending $t\to \infty$ proves that for any $\{\xi_{a;i}\} \in E_\xi(x)$, $\lim_{t\to \infty} \Prob^\xi\big(\tau_{a;n}> t\big) = 0$. Therefore, we have
	\begin{align*}
	\Prob^\xi(\tau_{a;n}<\infty)=\Prob^\xi(\cup_{t\geq 1} \{\tau_{a;n}\leq t\})=\uparrow\lim_{t \to \infty} \Prob^\xi(\tau_{a;n}\leq t)=1.
	\end{align*}
	As the above display does not depend on $x$, it holds for any $\{\xi_{a;i}\}$ with a finite $x$.
\end{proof}

\begin{lemma}\label{lem:var_bound_n_aT}
	Let $\tau_\pm$ be two finite stopping times with respect to the filtration $\{\mathscr{F}_t^Z\}$ (for every $\{\xi_{a;i}\}$) such that $\tau_-\leq \tau_+$. Then with $\{\mathfrak{p}_{a;t}^\xi\}$ defined in (\ref{def:cond_prob}),
	\begin{align}\label{ineq:var_bound_n_aT}
	\E^\xi\bigg(n_{a;\tau_+}-n_{a;\tau_-}-\sum_{t \in (\tau_-:\tau_+]}\mathfrak{p}_{a;t}^\xi \bigg)^2 \leq \E^\xi (n_{a;\tau_+}-n_{a;\tau_-}).
	\end{align}
\end{lemma}
\begin{proof}
	With $\mathbb{Z}_{\leq}^2\equiv \{(x,y)\in \mathbb{Z}_{\geq 0}^2: x\leq y\}$ and using $n_{a;t}=\sum_{s \in [t]}\bm{1}_{A_s=a}$, 
	\begin{align}\label{ineq:var_bound_n_aT_1}
	& \hbox{LHS of (\ref{ineq:var_bound_n_aT})}= \sum_{(T_-,T_+)\in \mathbb{Z}_{\leq}^2} \E^\xi\bigg(n_{a;T_+}-n_{a;T_-}-\sum_{t \in (T_-:T_+]}\mathfrak{p}_{a;t}^\xi \bigg)^2 \bm{1}_{\tau_\pm=T_\pm} \nonumber\\
	& = \sum_{(T_-,T_+)\in \mathbb{Z}_{\leq}^2}\sum_{T_-< s,t \leq T_+} \E^\xi\big(\bm{1}_{A_t=a}- \mathfrak{p}_{a;t}^\xi\big)\cdot\big(\bm{1}_{A_s=a}- \mathfrak{p}_{a;s}^\xi\big)\cdot \bm{1}_{\tau_\pm=T_\pm}\nonumber\\
	& = \sum_{ s,t \geq 1} \E^\xi \big(\bm{1}_{A_{t}=a}- \mathfrak{p}_{a;t}^\xi\big)\cdot \big(\bm{1}_{A_{s}=a}- \mathfrak{p}_{a;s}^\xi\big)\cdot \sum_{\substack{T_+\geq t\vee s, T_-<t\wedge s} } \bm{1}_{\tau_\pm=T_\pm}\nonumber\\
	& = \sum_{ s,t \geq 1} \E^\xi \big(\bm{1}_{A_{t}=a}- \mathfrak{p}_{a;t}^\xi\big)\cdot \big(\bm{1}_{A_{s}=a}- \mathfrak{p}_{a;s}^\xi\big)\cdot \bm{1}_{\tau_+\geq t\vee s} \bm{1}_{\tau_-<t\wedge s}.
	\end{align}
	On the other hand, for $s<t$, all the events $\{A_s=a\} \in \mathscr{F}_s^Z\subset \mathscr{F}_{t-1}^Z$, $\{\tau_+\geq t\}=\{\tau_+\leq t-1\}^c\in \mathscr{F}_{t-1}^Z$ and $\{\tau_-<s\} \in \mathscr{F}_{s-1}^Z\subset \mathscr{F}_{t-1}^Z$, as well as $\mathfrak{p}_{a;s}^\xi  \in \mathscr{F}_{s-1}^Z\subset \mathscr{F}_{t-1}^Z$ by definition. This means we may compute
	\begin{align}\label{ineq:var_bound_n_aT_2}
	&\E^\xi \big(\bm{1}_{A_{t}=a}- \mathfrak{p}_{a;t}^\xi\big)\cdot \big(\bm{1}_{A_{s}=a}- \mathfrak{p}_{a;s}^\xi\big)\cdot \bm{1}_{\tau_+\geq t\vee s} \bm{1}_{\tau_-< t\wedge s}\nonumber\\
	& = \E^\xi \Big\{\E^\xi \big[\big(\bm{1}_{A_{t}=a}- \mathfrak{p}_{a;t}^\xi\big)|\mathscr{F}_{t-1}\big]\cdot \big(\bm{1}_{A_{s}=a}- \mathfrak{p}_{a;s}^\xi\big)\cdot \bm{1}_{\tau_+\geq t} \bm{1}_{\tau_-< s}\Big\}=0.
	\end{align}
	A similar argument holds for $s>t$. Combining (\ref{ineq:var_bound_n_aT_1}) and (\ref{ineq:var_bound_n_aT_2}), we now have 
	\begin{align*}
	\hbox{RHS of }(\ref{ineq:var_bound_n_aT_1})  & = \sum_{t \geq 1} \E^\xi \big(\bm{1}_{A_{t}=a}- \mathfrak{p}_{a;t}^\xi\big)^2\cdot \bm{1}_{\tau_+\geq t} \bm{1}_{\tau_-<t}\\
	&\leq  \sum_{t \geq 1} \E^\xi \bm{1}_{A_{t}=a}\cdot \bm{1}_{t \in (\tau_-:\tau_+] }=\E^\xi (n_{a;\tau_+}-n_{a;\tau_-}).
	\end{align*}
	The claim follows. 
\end{proof}

\subsection{An apriori estimate}

\begin{lemma}\label{lem:apriori_arm_pull_lower_bound}
     Suppose Assumption \ref{assump:sampling_dist} holds. Fix $x\geq 1$. There exists a constant $c_1=c_1( K,\Delta,\mathscr{L}(\mathsf{Z}))>1$ such that for any $\{\xi_{a;i}\}\in E_\xi(x)$ and $1/[\bar{\Phi}_\ast(c_1 x^3)]^2\leq t_0\leq t/2$, 
	\begin{align*}
	&\Prob^\xi\Big(\min_{a \in [K]} n_{a;t}\geq \big[\bar{\Phi}_\ast^{-}(t^{-1/2})/(c_1 x)\big]^2 \hbox{ and }n_{\mathcal{A}_0;t}\geq t/(\log t)^2\\
	&\qquad\qquad \hbox{ for all }t \geq c_1 t_0\Big)\geq 1- c_1 \exp(-t_0^{1/2}/c_1).
	\end{align*} 
	Here $n_{\mathcal{A}_0;t}\equiv \sum_{o \in \mathcal{A}_0} n_{o;t}$, and the above estimate also holds when $\Delta_\cdot \equiv 0$.
\end{lemma}
\begin{proof}	
	\noindent (1). Let $w_t>0$ be determined later. Note that for $(a,t) \in [K]\times \N$, 
	\begin{align*}
	&\Prob^\xi\big(n_{a;t}\leq  w_t\big) \leq  \sum_{ \substack{\mathcal{T}\subset [t],\, \abs{[t]\setminus\mathcal{T}}\leq w_t} }\Prob^\xi\Big(\bar{\theta}_{a;s}<\max_{b \neq a} \bar{\theta}_{b;s},\,\forall s\in \mathcal{T}\Big)\\
	&\leq \sum_{b \neq a} \sum_{ \substack{\mathcal{T}\subset [t],\, \abs{[t]\setminus\mathcal{T}}\leq w_t} } \Prob^\xi\bigg(Z_{a;s}\leq  (w_t+1)^{1/2}(2x+\Delta_{\max})+\bigg(\frac{\bar{n}_{a;s-1}}{\bar{n}_{b;s-1}}\bigg)^{1/2}Z_{b;s},\,\forall s\in \mathcal{T}\bigg)\\
	&\leq K\cdot (e t/w_t)^{w_t}\cdot \big[\Phi_\ast\big((w_t+1)^{1/2}(2x+\Delta_{\max})\big)\big]^{t-w_t}.
	\end{align*}
	Then with $w_t\equiv [\bar{\Phi}_\ast^{-}(t^{-1/2})/(c_\ast x)]^2$ for a large enough $c_\ast>0$, for $t\geq t_0\equiv 1/[\bar{\Phi}_\ast(c_\ast x)]^2$, we have $w_t\geq 1$, and therefore
	\begin{align*}
	&\Prob^\xi\big(n_{a;t}\leq w_t\big)\leq \exp\big[c_1+w_t\log (et)+(t-w_t) \log \big(1-\bar{\Phi}_\ast(c_1 w_t^{1/2} x)\big)\big]\\
	&\leq c_1 \exp\big[w_t \log(et)-(t-w_t)\cdot \bar{\Phi}_\ast(c_1 w_t^{1/2} x)\big]\leq c_1 \exp\big(-t^{1/2}/c_1\big). 
	\end{align*}
	Here the last inequality follows by noting that $w_t\leq c_1 t^{0.01}$ by the equivalent form of the first condition in (B2) with $\lim_{z\uparrow \infty} \log \bar{\Phi}^-(1/z)/\log z= 0$, and the fact that $\bar{\Phi}_\ast^-\leq \bar{\Phi}^-$ [which can be seen by noting $\bar{\Phi}_\ast\leq \bar{\Phi}$ and the definition of generalized inverse]. The claim now follows by a union bound  across $t \geq t_0$.
	
	\noindent (2). Let $E_{1;t_0}$ be the event on which $\min_{a \in [K]} n_{a;t}\geq \big[\bar{\Phi}_\ast^{-}(t^{-1/2})/(c_1 x)\big]^2$ holds for all $t \geq t_0$, where we further require  $x^3/\bar{\Phi}_\ast^{-}(t_0^{-1/2})\leq 1/c_\ast$ for a large enough constant $c_\ast>0$. Then with the notation in (\ref{def:Phi_Psi_ast}) and $v_t$ to be determined later, 
	\begin{align*}
	&\Prob^\xi\big(n_{\mathcal{A}_0;t}\leq v_t; E_{1;t_0}\big) \leq  \sum_{ \substack{\mathcal{T}\subset [t],\, \abs{[t]\setminus\mathcal{T}}\leq v_t} }\Prob^\xi\Big(\max_{o \in \mathcal{A}_0}\bar{\theta}_{o;s}<\max_{b \in \mathcal{A}_+} \bar{\theta}_{b;s},\,\forall s\in \mathcal{T};E_{1;t_0}\Big)\nonumber\\
	& \leq \sum_{\substack{b \in \mathcal{A}_+\\\mathcal{T}\subset [t],\, \abs{[t]\setminus\mathcal{T}}\leq v_t}}  \inf_{o \in \mathcal{A}_0} \Prob^\xi \bigg( \frac{Z_{b;s}}{\sqrt{\bar{n}_{b;s-1}} }-\frac{Z_{o;s}}{\sqrt{\bar{n}_{o;s-1}} } > \Delta_{\min}-  \frac{c_1x^{3/2}}{\{\bar{\Phi}_\ast^{-}(t_0^{-1/2})\}^{1/2}},\,\forall s\in \mathcal{T}, s\geq t_0\bigg) \nonumber\\
	&\leq K \cdot \sum_{ \substack{\mathcal{T}\subset [t],\, \abs{[t]\setminus\mathcal{T}}\leq v_t} } \prod_{s \in \mathcal{T},s\geq t_0} \bar{\Psi}_\ast\, \big[ \bar{\Phi}_\ast^{-}(s^{-1/2})/c_1 x\big].
	\end{align*}
	With $d_{x,t_0}(c)\equiv -\log\bar{\Psi}_\ast\, \big[ \bar{\Phi}_\ast^{-}(t_0^{-1/2})/c x\big]$ and $v_t\equiv t/(\log t)^2$, using Lemma \ref{lem:Phi_Psi_1}, the right hand side of the above display is further bounded by 
	\begin{align*}
	&c_1 \sum_{ \substack{\mathcal{T}\subset [t],\, \abs{[t]\setminus\mathcal{T}}\leq v_t} }  \exp\big[-(\abs{\mathcal{T}}-t_0)_+ d_{x,t_0}(c_1)\big]\\
	&\leq c_1 \exp\big[v_t \log(et)-(t-v_t-t_0)_+ d_{x,t_0}(c_1)\big]\leq c_1 \exp(-t/c_1).
	\end{align*}
	The claim follows by a union bound across $t \geq c_1 t_0$.
\end{proof}

\subsection{Proof of Proposition \ref{prop:cond_prop_estimate}}\label{subsection:proof_cond_prop_estimate}

	Recall $\bar{\theta}_{a;t}$ defined in (\ref{def:theta}). Let  $E_{a;t}\equiv \big\{\bar{\theta}_{a;t}>\max_{o \in \mathcal{A}_0}\bar{\theta}_{o;t}\big\}$. Then we have
	\begin{align}\label{ineq:cond_prob_gauss_0}
	\{A_t=a\}=E_{a;t}\cap  \Big\{\bar{\theta}_{a;t}> \max_{b \in \mathcal{A}_+\setminus \{a\}} \bar{\theta}_{b;t}\Big\}.
	\end{align}
	Let $\mathcal{E}_{x}$ be the event specified in Lemma \ref{lem:apriori_arm_pull_lower_bound} with $t_0(x)=1/[\bar{\Phi}_\ast(c_1x^4)]^2$. In other words, on $\mathcal{E}_{x}$ with $\Prob^\xi(\mathcal{E}_{x}^c)\leq c_1 \exp\big(-1/\{c_1\bar{\Phi}_\ast(c_1x^4)\}\big)$,  we have $\min_{a \in [K]}n_{a;t}\geq \big[\bar{\Phi}_\ast^{-}(t^{-1/2})/(c_1 x)\big]^2\geq x^6/c_1$ and $n_{\mathcal{A}_0;t}\geq t/(\log t)^2\geq 1/\{c_1\bar{\Phi}_\ast(c_1x^4)\}$ for all $t\geq t_0(x)$. For notational simplicity, we let $v_t\equiv t/(\log t)^2$, and define $\bar{\xi}_{a;t}\equiv  \bar{n}_{a;t}^{-1}\sum_{s \in [t]} \bm{1}_{A_s=a} \xi_s$. On the event $\mathcal{E}_{x}$, we have $\max_{a \in [K]} \abs{ \bar{\xi}_{a;t}}\leq c_1/x$.

	\noindent (\emph{Upper bound}). First, by dropping the second event in (\ref{ineq:cond_prob_gauss_0}), we have
	\begin{align}\label{ineq:cond_prob_gauss_ub1}
	\mathfrak{p}_{a;t}^\xi\leq \Prob^\xi \Big(\bar{\theta}_{a;t}>\max_{o \in \mathcal{A}_0}\bar{\theta}_{o;t} \big|\mathscr{F}_{t-1}^Z\Big).
	\end{align}
	On the event $\mathcal{E}_{x}$, with $o_{t-1}\in \argmax_{o \in \mathcal{A}_0} \bar{n}_{o;t-1}$,
	\begin{align*}
	&\hbox{RHS of (\ref{ineq:cond_prob_gauss_ub1})}\\
	&\leq\Prob_Z\Big(Z_{a;t}\geq \bar{n}_{a;t-1}^{1/2}\big[\Delta_a -  \bar{\xi}_{a;t-1}+  (\bar{\xi}_{o_{t-1};t-1}+\bar{n}_{o_{t-1};t-1}^{-1/2}Z_{o_{t-1};t})\big]\Big)\\
	&\leq \Prob_Z \Big(Z_{a;t}\geq \bar{n}_{a;t-1}^{1/2}\big[\Delta_a-c_1/x-c_1(Z_{o_{t-1};t})_-/v_{t-1}^{1/2}\big]\Big)\\
	&\leq \Prob_Z \big(Z_{a;t}\geq \bar{n}_{a;t-1}^{1/2} [\Delta_a-c_1'/x] \big)+\Prob_Z\big((Z_{o_{t-1};t})_->c_1 v_{t-1}^{1/2}/x \big).
	\end{align*}
	Combined with (\ref{ineq:cond_prob_gauss_ub1}), we have
	\begin{align}\label{ineq:cond_prob_gauss_ub}
	\mathfrak{p}_{a;t}\leq \bar{\Phi}\big(\bar{n}_{a;t-1}^{1/2} [\Delta_a-c_1''/x] \big)+2\Phi(-t^{1/2}/(c_1''x\log t)).
	\end{align}
	\noindent (\emph{Lower bound}). Next we consider the lower bound. Using that for any $\epsilon>0$,
	\begin{align*}
	\big\{\bar{\theta}_{a;t}>\mu_\ast+\epsilon\big\}\cap  \Big\{\max_{b\neq a}\bar{\theta}_{b;t}<\mu_\ast+\epsilon\Big\}\subset \{A_t=a\},
	\end{align*}
	and noting that two events on the left hand side of the above display are independent conditional on $\mathscr{F}_{t-1}^Z$, it follows that 
	\begin{align}\label{ineq:cond_prob_gauss_lb1}
	\mathfrak{p}_{a;t}^\xi&\geq \Prob^\xi\big(\bar{\theta}_{a;t}>\mu_\ast+\epsilon|\mathscr{F}_{t-1}^Z\big)\cdot \prod\nolimits_{b\neq a} \Prob^\xi\big(\bar{\theta}_{b;t}<\mu_\ast+\epsilon|\mathscr{F}_{t-1}^Z\big).
	\end{align}
	Note that on the event $\mathcal{E}_{x}$,
	\begin{align*}
	&\Prob^\xi\big(\bar{\theta}_{a;t}>\mu_\ast+\epsilon|\mathscr{F}_{t-1}^Z\big)=\Prob_Z\Big(Z_{a;t}\geq \bar{n}_{a;t-1}^{1/2}\big[\Delta_a+\epsilon -  \bar{\xi}_{a;t-1}\big]\Big)\\
	&\geq \Prob_Z\big(Z_{a;t}\geq \bar{n}_{a;t-1}^{1/2}[\Delta_a+\epsilon+c_2/x]\big)\geq \bar{\Phi}\big(\bar{n}_{a;t-1}^{1/2}[\Delta_a+\epsilon+c_2/x]\big),
	\end{align*}
	and for $b\neq a$, $t\geq t_0+1$, 
	\begin{align*}
	\Prob^\xi\big(\bar{\theta}_{b;t}<\mu_\ast+\epsilon|\mathscr{F}_{t-1}^Z\big)&\geq \Prob_Z\Big(Z_{a;t}< \bar{n}_{b;t-1}^{1/2}\big[\epsilon -  \bar{\xi}_{b;t-1}\big]\Big)\\
	&\geq \Phi\big(\bar{\Phi}_\ast^{-}(2t^{-1/2})/(c_2 x)(\epsilon-c_2/x)\big).
	\end{align*}
	Combining the above estimates with (\ref{ineq:cond_prob_gauss_lb1}), by choosing $\epsilon=c_2'/x$ for sufficiently large $c_2'$, we arrive at 
	\begin{align}\label{ineq:cond_prob_gauss_lb}
	\mathfrak{p}_{a;t}^\xi\geq \bar{\Phi}\big(\bar{n}_{a;t-1}^{1/2}[\Delta_a+c_2''/x]\big)\cdot \big[1-\bar{\Phi}\big(\bar{\Phi}_\ast^{-}(2t^{-1/2})/(c_2'' x^2)\big)\big]^K.
	\end{align}
	The claimed estimate follows from (\ref{ineq:cond_prob_gauss_ub}) and (\ref{ineq:cond_prob_gauss_lb}). \qed

\subsection{Proof of Proposition \ref{prop:tau_n_char_eqn}}\label{subsection:proof_tau_n_char_eqn}

	Fix $0\leq m<n$. Consider the stopping times $\tau_{a;m}\leq \tau_{a;n}$. Using Lemmas \ref{lem:tau_property} and \ref{lem:var_bound_n_aT}, for some random variable $Y_{n,m}$ with uniform second moment estimate $\sup_{\{\xi_{a;i}\}} \E^\xi Y_{n,m}^2\leq 1$, 
	\begin{align}\label{ineq:tau_n_char_eqn_1}
	n-m= n_{a;\tau_{a;n}}-n_{a;\tau_{a;m}} =\sum_{t \in (\tau_{a;m}:\tau_{a;n}]} \mathfrak{p}_{a;t}^\xi +\sqrt{n-m}\cdot Y_{n,m}.
	\end{align}
	Fix $\epsilon \in (0,1/2)$. For notational simplicity, we write $\Delta_{a;\epsilon}^\pm\equiv (1\pm \epsilon)\Delta_a$.
	
	Using the upper estimate in Proposition \ref{prop:cond_prop_estimate}, by choosing $x\geq x_\ast(\epsilon)$ large enough, for all $\{\xi_{a;i}\} \in E_\xi(x_\ast(\epsilon))$, on an event $\mathcal{E}_{x_\ast(\epsilon)}\equiv \mathcal{E}_{x_\ast(\epsilon);m,n}(\{\xi_{a;i}\}) $ with $\Prob^\xi\big(\mathcal{E}_{x_\ast(\epsilon)}\big)\geq 1-\epsilon$,  there exists some $c_1=c_1(K,\Delta,\epsilon,\mathscr{L}(\mathsf{Z}))>0$ such that
	\begin{align*}
	n-m &\leq \sum_{t \in (\tau_{a;m}:\tau_{a;n}]} \bar{\Phi}\big(\bar{n}_{a;t-1}^{1/2} \Delta_{a;\epsilon}^-\big)+2\sum_{t \in (\tau_{a;m}:\tau_{a;n}]} \Phi\bigg(-\frac{t^{1/2}}{c_1\log t}\bigg)+ c_1\big(1+\sqrt{n-m} \cdot \abs{Y_{n,m}}\big).
	\end{align*}
	Using that
	\begin{enumerate}
		\item $\sum_{t \in (\tau_{a;m}:\tau_{a;n}]} \bar{\Phi}\big(\bar{n}_{a;t-1}^{1/2} \Delta_{a;\epsilon}^-\big)\leq \sum_{k \in (m:n]} (\tau_{a;k}-\tau_{a;k-1})\bar{\Phi}\big(k^{1/2}\Delta_{a;\epsilon}^-\big)$,
		\item $
		\sum_{t \in (\tau_{a;m}:\tau_{a;n}]} \Phi\big(-t^{1/2}/(c_1\log t)\big)\leq  \int_1^\infty \Phi(-t^{1/4}/c_1)\,\d{t}\leq c_1$ by (B2),
	\end{enumerate}
	we have 
	\begin{align}\label{ineq:tau_n_char_eqn_ub}
	n-m\leq \sum_{k \in (m:n]} (\tau_{a;k}-\tau_{a;k-1})\bar{\Phi}\big(k^{1/2}\Delta_{a;\epsilon}^-\big)+c_1\big(1+\sqrt{n-m}\cdot \abs{Y_{n,m}}\big). 
	\end{align}
	On the other hand, using (\ref{ineq:tau_n_char_eqn_1}) and the lower estimate in Proposition \ref{prop:cond_prop_estimate}, for the prescribed choice of $x_\ast(\epsilon)$, for all $\{\xi_{a;i}\} \in E_\xi(x_\ast(\epsilon))$, on the event $\mathcal{E}_{x_\ast(\epsilon)}$, for any $n_0 \in [n]$,
	\begin{align}\label{ineq:tau_n_char_eqn_lb_1}
	&n-m \geq -\sqrt{n_0} + \sum_{t \in (\sqrt{n_0}\vee \tau_{a;m}:\tau_{a;n}]} \mathfrak{p}_{a;t}^\xi + \sqrt{n-m}\cdot Y_{n,m}\\
	& \geq \bigg[1-\bar{\Phi}\bigg(\frac{\bar{\Phi}_\ast^{-}(2n_0^{-1/2})}{c_1} \bigg)\bigg]^K \sum_{t \in (\sqrt{n_0}\vee \tau_{a;m}:\tau_{a;n}]} \bar{\Phi}\big(\bar{n}_{a;t-1}^{1/2} \Delta_{a;\epsilon}^+\big) -  (\sqrt{n_0}+\sqrt{n-m}\cdot\abs{Y_{n,m}})\nonumber\\
	&\geq \bigg[1-c_1 \bar{\Phi}\bigg(\frac{\bar{\Phi}_\ast^{-}(2n_0^{-1/2})}{c_1} \bigg)\bigg]_+ \sum_{t \in (\tau_{a;m}:\tau_{a;n}]} \bar{\Phi}\big(\bar{n}_{a;t-1}^{1/2} \Delta_{a;\epsilon}^+\big) -  2(\sqrt{n_0}+\sqrt{n-m}\cdot\abs{Y_{n,m}}).\nonumber
	\end{align}
	Rearranging terms using Lemma \ref{lem:Phi_Psi_1}, we have for $n_0\geq c_1$, 
	\begin{align}\label{ineq:tau_n_char_eqn_lb_2}
	\sum_{t \in (\tau_{a;m}:\tau_{a;n}]} \bar{\Phi}\big(\bar{n}_{a;t-1}^{1/2} \Delta_{a;\epsilon}^+\big)\leq c_1\cdot (n-m+\sqrt{n_0}+\sqrt{n-m}\cdot \abs{Y_{n,m}}). 
	\end{align}
	Now combining (\ref{ineq:tau_n_char_eqn_lb_1}) and (\ref{ineq:tau_n_char_eqn_lb_2}), and using 
	\begin{align*}
	\sum_{t \in (\tau_{a;m}:\tau_{a;n}]} \bar{\Phi}\big(\bar{n}_{a;t-1}^{1/2} \Delta_{a;\epsilon}^+\big)\geq  \sum_{k \in (m:n]} (\tau_{a;k}-\tau_{a;k-1})\bar{\Phi}\big((k+1)^{1/2}\Delta_{a;\epsilon}^+\big),
	\end{align*}
	we have
	\begin{align}\label{ineq:tau_n_char_eqn_lb}
	n-m &\geq \sum_{k \in (m:n]} (\tau_{a;k}-\tau_{a;k-1})\bar{\Phi}\big((k+1)^{1/2}\Delta_{a;\epsilon}^+\big)\nonumber\\
	&\qquad - c_1 \cdot \big[(n-m)\bar{\Phi}\big(\bar{\Phi}_\ast^{-}(2n_0^{-1/2})/c_1\big)+\sqrt{n_0}+\sqrt{n-m}\cdot\abs{Y_{m,n}}\big].
	\end{align}
	The claim now follows from (\ref{ineq:tau_n_char_eqn_ub}) and (\ref{ineq:tau_n_char_eqn_lb}) by taking $n_0=n-m$. \qed

\subsection{Proof of Proposition \ref{prop:tau_n_est}}\label{subsection:proof_tau_n_est}

	Fix $\epsilon \in (0,1/4)$. Let the constants $c_0=c_0(K,\Delta,\epsilon,\mathscr{L}(\mathsf{Z}))>0$, $x_\ast(\epsilon)$, and the events $E_\xi(x_\ast(\epsilon))$, $\big\{\mathcal{E}_{x_\ast(\epsilon)}\equiv \mathcal{E}_{x_\ast(\epsilon);m,n}(\{\xi_{a;i}\})\big\}$ be as in the statement and the proof of Proposition \ref{prop:tau_n_char_eqn}. We further write
	\begin{align*}
	S_{a;n,m}^+ \equiv \sum_{k \in (m:n]} \beta_{a;k}\bar{\Phi}\big((k+1)^{1/2}\Delta_{a;\epsilon}^+\big),\quad S_{a;n,m}^- \equiv \sum_{k \in (m:n]} \beta_{a;k}\bar{\Phi}\big(k^{1/2}\Delta_{a;\epsilon}^-\big),
	\end{align*}
	with $\beta_{a;k}\equiv \tau_{a;k}-\tau_{a;k-1}$.

	\noindent (\emph{Upper bound}). Applying the negative control part of Proposition \ref{prop:tau_n_char_eqn} with $m=0$, we have on $\mathcal{E}_{x_\ast(\epsilon)}$, 
	\begin{align*}
	&\tau_{a;n}\cdot \bar{\Phi}\big((n+1)^{1/2}\Delta_{a;\epsilon}^+\big) = \bigg(\sum_{k \in [n]} \beta_{a;k}\bigg)\cdot \bar{\Phi}\big((n+1)^{1/2}\Delta_{a;\epsilon}^+\big)\\
	&\leq \sum_{k \in [n]} \beta_{a;k}\bar{\Phi}\big((k+1)^{1/2}\Delta_{a;\epsilon}^+\big)\leq  S_{a;n,0}^+\leq n+ c_1\cdot \big(\alpha_{n}(c_1)+\sqrt{n}\cdot \abs{Y_{n,0}}\big),
	\end{align*}
	which implies, for $n\geq c_1$,
	\begin{align*}
	\tau_{a;n}\leq c_1 \big(1\vee \abs{Y_{n,0}}\big)\cdot n/\bar{\Phi}\big((n+1)^{1/2}\Delta_{a;\epsilon}^+\big).
	\end{align*} 
	The claimed upper bound follows by a trivial large deviation estimate for $\abs{Y_{n,0}}$ via the uniform second moment control, and then bound $c_1 \cdot n /\bar{\Phi}\big((n+1)^{1/2}\Delta_{a;\epsilon}^+\big)$ by $1/\bar{\Phi}\big(n^{1/2}\Delta_{a;2\epsilon}^+\big)$ for $n\geq c_1$ using assumption (B2).
	
	\noindent (\emph{Lower bound}). Let $n_0=n_0(K,\Delta,\epsilon,\mathscr{L}(\mathsf{Z}))$ be a sufficiently large integer such that the event $\bar{E}_{n_0}\equiv \{n_0\geq c_0^2(1\vee \abs{Y_{n,n-n_0}})^2/4\}$ satisfies $\sup_{\{\xi_{a;i}\}} \Prob^\xi (\bar{E}_{n_0}^c)\leq \epsilon$. Applying the positive control part of Proposition \ref{prop:tau_n_char_eqn} with $m=n-n_0$, on the event  $\mathcal{E}_{x_\ast(\epsilon)}\cap \bar{E}_{n_0}$, 
	\begin{align*}
	\sum_{k \in (n-n_0:n]} \beta_{a;k}\bar{\Phi}\big(k^{1/2}\Delta_{a;\epsilon}^-\big)=S_{a;n,n-n_0}^-&\geq  n_0- c_0 \sqrt{n_0}\cdot (1\vee \abs{Y_{n,n-n_0}})\geq n_0/2.
	\end{align*}
	Consequently, there exists some $k_0 \in (n-n_0:n]$ such that
	\begin{align*}
	\beta_{a;k_0}\bar{\Phi}\big(k_0^{1/2}\Delta_{a;\epsilon}^-\big)\geq 1/2\,\Rightarrow\, \beta_{a;k_0}\geq 1/\big\{2 \bar{\Phi}\big(k_0^{1/2}\Delta_{a;\epsilon}^-\big)\big\}.
	\end{align*}
	Therefore for $n\geq c_1\vee n_0$, we have 
	\begin{align*}
		\tau_{a;n}\geq \beta_{a;k_0}\geq 1/\big\{2 \bar{\Phi}\big((n-n_0)_+^{1/2}\Delta_{a;\epsilon}^-\big)\big\}.
	\end{align*}
	The lower bound follows by (B2). \qed

\section{Proof of Proposition \ref{prop:sde_strong_sol}}\label{subsection:proof_sde_strong_sol}

\subsection{Standard It\^o  form of the SDE (\ref{def:sde_time_change_renor})}
The SDE (\ref{def:sde_time_change_renor}) may be rewritten in the standard It\^o form as follows. Let $X(t)\equiv (u_{\cdot}(t),w_{\cdot}(t))\in E_0$, and let the mappings $b: E_0 \to \R^{\mathcal{A}_0}\times \R^{\mathcal{A}_0}$, $\sigma:E_0\to \big(\R^{\mathcal{A}_0}\times \R^{\mathcal{A}_0}\big)^2$ be defined as
\begin{align}\label{def:sde_b_sigma}
\begin{cases}
b(u_\cdot,w_\cdot)\equiv \big( (p_a(u_\cdot,w_\cdot)-u_\cdot)_{a \in \mathcal{A}_0},(-w_a/2)_{a \in \mathcal{A}_0}\big),\\
\sigma(u_\cdot,w_\cdot)\equiv \mathrm{diag}\big(0_{\mathcal{A}_0}, (\sqrt{p_a(u_\cdot,w_\cdot)})_{a \in \mathcal{A}_0}\big).
\end{cases}
\end{align}
Then the SDE (\ref{def:sde_time_change_renor}) may be written in the standard It\^o diffusion in $E_0$:
\begin{align}\label{def:sde_standard_Ito_form}
\d{X(t)} = b(X(t))\,\d{t}+\sigma(X(t))\,\d{B(t)},\quad \forall t\in \R.
\end{align}

\subsection{Proof of Proposition \ref{prop:sde_strong_sol}}

We only need to prove the claim on the interval $[0,T]$ for any $T>0$. 

\noindent (\textbf{Step 1}). We first prove that any solution $(u_{\cdot}(t),w_{\cdot}(t))$ of the SDE (\ref{def:sde_time_change_renor}) on $[0,T]$ from an initial condition $(u_{\cdot}(0),w_{\cdot}(0))\in E_0$ must satisfy $(u_{\cdot}(t),w_{\cdot}(t)) \in E_0$ on $[0,T]$.

To see this, from the first equation of (\ref{def:sde_time_change_renor}), we have $u_a'(t)\geq - u_a(t)$, and therefore $u_a(t)\geq u_a(0) e^{-t}>0$ for all $t\geq 0$ and $a \in \mathcal{A}_0$. In particular, we have 
\begin{align}\label{ineq:sde_strong_sol_u_lb}
\min_{a \in \mathcal{A}_0}\inf_{t \in [0,T]}u_a(t)\geq \min_{a \in \mathcal{A}_0} u_a(0) e^{-T}\equiv \delta_0 \in (0,1).
\end{align}
Moreover, with $u\equiv \sum_{a \in \mathcal{A}_0} u_a$, using $\sum_{a \in \mathcal{A}_0} p_a=1$ and the first equation of (\ref{def:sde_time_change_renor}) yield that $u'(t)=1-u(t)$. Solving this ordinary differential equations leads to $u(t)=1+(u(0)-1) e^{-t}$. Since $(u_{\cdot}(0)) \in \Delta_{\mathcal{A}_0}^\circ$, we have  $u(0)=1$ and therefore $u(t)=1$ for all $t\geq 0$. This means $(u_{\cdot}(t)) \in \Delta_{\mathcal{A}_0}^\circ$ for all $t\geq 0$, proving the claim.

\noindent (\textbf{Step 2}). In this step, we construct a sequence of `localized' solutions $\{X^{(n)}(t): t \in [0,T]\}$ in $E_0$ to the SDE in its standard form (\ref{def:sde_standard_Ito_form}). 

To this end, we consider the localized space
\begin{align}\label{ineq:sde_strong_sol_En}
E_0^{(n)}\equiv \Big\{(u_a)_{a \in \mathcal{A}_0} \in \Delta_{\mathcal{A}_0}^\circ: \inf_{a \in \mathcal{A}_0}u_a\geq \delta_0\Big\}\times [-n,n]^{ \mathcal{A}_0},\quad \forall n=1,2,\ldots,
\end{align}
and let $\Pi^{(n)}: E_0 \to E_0^{(n)}$ be the natural `projection' in the sense that
\begin{align*}
\Pi^{(n)}(u_\cdot,w_\cdot)\equiv \Big((u_a\vee \delta_0)_{a \in \mathcal{A}_0}, \big((w_a\wedge n)\vee (-n)\big)_{a \in \mathcal{A}_0}\Big).
\end{align*}
Clearly $\Pi^{(n)}$ is $1$-Lipschitz. We further let
\begin{align*}
b^{(n)}(u_\cdot,w_\cdot)\equiv b\big(\Pi^{(n)}(u_\cdot,w_\cdot)\big),\quad \sigma^{(n)}(u_\cdot,w_\cdot)\equiv \sigma\big(\Pi^{(n)}(u_\cdot,w_\cdot)\big).
\end{align*}
Then $b^{(n)},\sigma^{(n)}$ are globally Lipschitz on $E_0$. Consequently, by the proven claim in Step 1 (with a trivial adaptation to $X^{(n)}$) and the classical SDE theory, cf. \cite[Chapter IX, Theorem 2.1, pp. 375]{revuz1999continuous} or \cite[Theorem 5.2.1, pp. 68]{oksendal2003stochastic}, with initial condition $(u_{\cdot}(0),w_{\cdot}(0))\in E_0$, the SDE
\begin{align*}
\d{X^{(n)}(t)} = b^{(n)}(X^{(n)}(t))\,\d{t}+\sigma^{(n)}(X^{(n)}(t))\,\d{B(t)},\quad \forall t \in \R
\end{align*}
admits a unique strong solution $X^{(n)}$ on $[0,T]$ in the space $E_0$. 

\noindent (\textbf{Step 3}). In this step, we shall use $\{X^{(n)}(t): t \in [0,2T]\}$ to define a global solution $\{X(t): t \in [0,T]\}$. Recall $E_0^{(n)}$ in (\ref{ineq:sde_strong_sol_En}). We define the stopping time
\begin{align*}
\tau_n^{(m)}\equiv \inf\big\{t \in [0,2T]: X^{(m)}(t) \notin E_0^{(n)}\big\},\quad \forall m,n\geq 1,
\end{align*}
with the convention that $\inf \emptyset = 2T$, so  $\tau_n^{(m)} \in [0,2T]$. Note that $\tau_n^{(m)}$ is well-defined (for $2T$) due to the continuity of $t\mapsto X^{(m)}(t)$ and the closedness of $E_0^{(n)}$. For $m\geq n$, as $b^{(m)}=b^{(n)}$ and $\sigma^{(m)}=\sigma^{(n)}$ on $E_0^{(n)}$, by the (pathwise) uniqueness of strong solutions, w.p.1., we have $
\tau^{(n)}\equiv \tau^{(n)}_n = \tau_n^{(m)}$ for all $m\geq n\geq 1$, and 
\begin{align}\label{ineq:sde_strong_sol_X_n}
X^{(m)}(t)=X^{(n)}(t),\quad \forall t \in [0, \tau^{(n)}],\, m\geq n.
\end{align}
Therefore w.p.1, for all $n\geq 1$, we have
\begin{align*}
\tau^{(n)}=\tau_n^{(n+1)}&= \inf\big\{t \in [0,2T]: X^{(n+1)}(t) \notin E_0^{(n)}\big\}\\
&\stackrel{(\ast)}{\leq} \inf\big\{t \in [0,2T]: X^{(n+1)}(t) \notin E_0^{(n+1)}\big\} = \tau^{(n+1)}.
\end{align*}
Here $(\ast)$ follows due to the nested property of $\{E_0^{(n)}\}$ in $n$. Consequently, 
\begin{align*}
\tau^{(\infty)}\equiv \uparrow\lim_{n \to \infty} \tau^{(n)} \in [0,2T]
\end{align*}
is a well-defined stopping time w.p.1. On $[0,\tau^{(\infty)})$, we define $X$ as follows: for any $t \in [0,\tau^{(\infty)})$, there exists some $n$ such that $t\leq \tau^{(n)}$, and we define $X(t)\equiv X^{(n)}(t)$. This is well-defined w.p.1, as for any $m\geq n$, as $t\leq \tau^{(n)}$, we have $X^{(n)}(t)=X^{(m)}(t)$ in view of (\ref{ineq:sde_strong_sol_X_n}).

It therefore remains to prove that $\tau^{(\infty)}>T$ w.p.1. To this end, note that $\{\tau^{(\infty)}\leq T\}=\cap_n \{\tau^{(n)}\leq T\}=\cap_n \{\max_{a \in \mathcal{A}_0}\sup_{t \in [0,T]}\abs{w_a^{(n)}(t)}\geq n\}$, with $W^{(n)}(t)\equiv \big\{\sum_{a \in \mathcal{A}_0} \big(w_a^{(n)}(t)\big)^2\big\}^{1/2}$, we have 
\begin{align}\label{ineq:sde_strong_sol_tau_infty}
\Prob(\tau^{(\infty)}\leq T)&\leq \abs{\mathcal{A}_0}\cdot\max_{a \in \mathcal{A}_0}  \liminf_{n \to \infty} \Prob \bigg(\sup_{t \in [0,T]}\abs{w_a^{(n)}(t)}\geq n\bigg)\nonumber\\
&\leq \abs{\mathcal{A}_0}\cdot \liminf_{n \to \infty} \Prob \bigg(\sup_{t \in [0,T]}W^{(n)}(t)\geq n\bigg).
\end{align}
On the other hand, by It\^o's formula (cf. \cite[Chapter IV, Theorem 3.3, pp. 147]{revuz1999continuous}) applied to $\big(w_a^{(n)}(t)\big)_{a \in \mathcal{A}_0}$ and $F(w_\cdot)=\sum_{a \in \mathcal{A}_0} w_a^2$, as the partial derivatives are $\partial_a F(w_\cdot)=2w_a$, $\partial_{ab} F(w_\cdot)=2\delta_{ab}$, and the quadratic variation process is $\iprod{w^{(n)}_a}{w^{(n)}_b}_\cdot = \delta_{ab} \int_0^\cdot \prod_{\# \in \{a,b\}}p_{\#}^{(n),1/2}(u_\cdot^{(n)}(t),w_\cdot^{(n)}(t))\,\d{t}$ where $p_\cdot^{(n)}\equiv p_\cdot \circ \Pi^{(n)}$, we may compute: 
\begin{align*}
\d{F(w^{(n)}_\cdot(t))}&=\sum_{a \in \mathcal{A}_0} \partial_a F (w^{(n)}_\cdot(t))\,\d{w^{(n)}_a(t)}+ \frac{1}{2}\sum_{a,b \in \mathcal{A}_0} \partial_{ab} F (w^{(n)}_\cdot(t))\,\d{\iprod{w^{(n)}_a}{w^{(n)}_b} }_t\\
& =  \sum_{a \in \mathcal{A}_0} \big( 2w^{(n)}_a(t)\,\d{w^{(n)}_a(t)}+ p_a^{(n)}(u_\cdot^{(n)}(t),w_\cdot^{(n)}(t))\,\d{t} \big)\\
& = \sum_{a \in \mathcal{A}_0} \big(-(w^{(n)}_a(t))^2+ p_a^{(n)}(u_\cdot^{(n)}(t),w_\cdot^{(n)}(t))\big)\,\d{t}\\
&\qquad + 2\sum_{a \in \mathcal{A}_0}  w^{(n)}_a(t)p_a^{(n),1/2}(u_\cdot^{(n)}(t),w_\cdot^{(n)}(t))\,\d{B_a(t)}. 
\end{align*}
So with $M^{(n)}(t)\equiv 2\sum_{a \in \mathcal{A}_0}  \int_0^t w^{(n)}_a(s)p_a^{(n),1/2}(u_\cdot^{(n)}(s),w_\cdot^{(n)}(s))\,\d{B_a(s)}$, using that $\sum_{a \in \mathcal{A}_0} p_a^{(n)}=1$, we arrive at 
\begin{align}\label{ineq:sde_strong_sol_tau_W_int}
(W^{(n)}(t))^2& = (W^{(n)}(0))^2+ \int_0^t \big[-(W^{(n)}(s))^2+1\big]\,\d{s}+M^{(n)}(t).
\end{align}
Consequently, 
\begin{align}\label{ineq:sde_strong_sol_sup_1}
\E\sup_{t \in [0,T]} (W^{(n)}(t))^2&\leq (W^{(n)}(0))^2+T+ \E\sup_{t \in [0,T]} (M^{(n)}(t))^2.
\end{align}
Using the Burkholder-Davis-Gundy inequality for the continuous martingale $\{M^{(n)}(t)\}$ that vanishes at $0$, cf. \cite[Chapter IV, Theorem 4.1 and Corollary 4.2, pp. 160]{revuz1999continuous}, we have for some universal constant $c>0$ whose numeric value may change from line to line below,
\begin{align*}
&\E\sup_{t \in [0,T]} (M^{(n)}(t))^2\leq c\cdot \E \langle M^{(n)} \rangle_{T} \\
&= 4c \sum_{a \in \mathcal{A}_0} \E \int_0^{T} (w^{(n)}_a(s))^2 p_a^{(n)}(u_\cdot^{(n)}(s),w_\cdot^{(n)}(s))\,\d{s}\leq 4c\cdot \int_0^{T} \E (W^{(n)}(s))^2\,\d{s}.
\end{align*}
On the other hand, in view of (\ref{ineq:sde_strong_sol_tau_W_int}), we have the simple estimate
\begin{align*}
\E (W^{(n)}(t))^2 = (W^{(n)}(0))^2+ \int_0^t \E\big[-(W^{(n)}(s))^2+1\big]\,\d{s}\leq (W^{(n)}(0))^2+t.
\end{align*}
Combining the above two displays, we have 
\begin{align}\label{ineq:sde_strong_sol_sup_2}
\E\sup_{t \in [0,T]} (M^{(n)}(t))^2\leq c T\cdot \big((W^{(n)}(0))^2+T\big).
\end{align}
Combining (\ref{ineq:sde_strong_sol_sup_1}) and (\ref{ineq:sde_strong_sol_sup_2}), and using $(W^{(n)}(0))^2=\pnorm{w_\cdot(0)}{}^2$,  it follows that
\begin{align*}
\E\sup_{t \in [0,T]} (W^{(n)}(t))^2\leq c\cdot (1+T)\cdot \big(\pnorm{w_\cdot(0)}{}^2+T\big). 
\end{align*}
Now combining the above estimate with (\ref{ineq:sde_strong_sol_tau_infty}), using Chebyshev's inequality we conclude that
\begin{align*}
\Prob(\tau^{(\infty)}\leq T)&\leq c\abs{\mathcal{A}_0}\cdot \lim_{n \to \infty}\big\{ n^{-2}\cdot (1+T)\cdot \big(\pnorm{w_\cdot(0)}{}^2+T\big)\big\}=0.
\end{align*}
This shows that w.p.1, $X$ is well-defined on $[0,T]$. In fact, the above argument shows that w.p.1, $\tau^{(\infty)}>2T-\epsilon$ and $X$ is well-defined on $[0,2T-\epsilon]$ for any $\epsilon>0$, and therefore actually w.p.1, $\tau^{(\infty)}=2T$ and $X$ is well-defined on $[0,2T)$.

\noindent (\textbf{Step 4}). Finally, we shall prove the (pathwise) uniqueness of $X$. Let $X'$ be another solution of the SDE (\ref{def:sde_standard_Ito_form}) with the same Brownian motion and initial condition $(u_{\cdot}(0),w_{\cdot}(0))\in E_0$. Recall again $E_0^{(n)}$ in (\ref{ineq:sde_strong_sol_En}). Let
\begin{align*}
\hat{\tau}^{(n)}\equiv \inf \big\{t \in [0,T]: X(t) \notin E_0^{(n)}\hbox{ or } X'(t) \notin E_0^{(n)}\big\},
\end{align*}
with the convention that $\inf\emptyset \equiv T$. As $\{E_0^{(n)}\}$ is nested in $n$, $\{\hat{\tau}^{(n)}\}$ is a non-decreasing sequence, and therefore $\hat{\tau}^{(\infty)}\equiv \uparrow \lim_{n \to \infty} \hat{\tau}^{(n)}$ is well-defined. As $X,X'$ are both continuous, for any $\omega \in \Omega_0$ with $\Prob(\Omega_0)=1$, there exists some $n_0=n_0(\omega)$ such that $X(t),X'(t) \in E_0^{(n)}$ for all $t \in [0,T]$ and $n\geq n_0$. This means that $\hat{\tau}^{(n)}=T$ for $n\geq n_0$. In other words, we must have 
\begin{align}\label{ineq:sde_strong_sol_tau_hat}
\uparrow \lim_{n \to \infty} \hat{\tau}^{(n)} = T,\quad \hbox{ a.s.}
\end{align}
Now with $\Delta_t^{(n)}\equiv X_{t\wedge \hat{\tau}^{(n)}}-X'_{t\wedge \hat{\tau}^{(n)}}$ and $\hat{M}^{(n)}(t)\equiv \int_0^{t\wedge \hat{\tau}^{(n)}  }  \big(\sigma(X_s)-\sigma(X_s')\big)\,\d{B(s)} $,  for any $T_0\in [0,T]$,
\begin{align}\label{ineq:sde_strong_sol_delta_1}
\E \sup_{t \in [0,T_0]}\pnorm{\Delta_t^{(n)}}{}^2& \leq 2 \E \sup_{t \in [0,T_0]}\biggpnorm{\int_0^{t\wedge \hat{\tau}^{(n)}  }  \big(b(X_s)-b(X_s')\big)\,\d{s} }{}^2+ 2\E \sup_{t \in [0,T_0]} \pnorm{\hat{M}^{(n)}(t)}{}^2\nonumber\\
&\equiv I_1(T_0)+I_2(T_0).
\end{align}
For $I_1(T_0)$, As $b$ is Lipschitz on $E_0^{(n)}$ (with Lipschitz constant depending on $n$), for some constant $c_1=c_1(\abs{\mathcal{A}_0},n,T)>0$,
\begin{align}\label{ineq:sde_strong_sol_delta_2}
I_1(T_0)&\leq c_1 \int_0^{T_0} \E\sup_{r \in [0,s]}\pnorm{\Delta_r^{(n)}}{}^2\,\d{s},\quad \forall T_0\in [0,T].
\end{align}
For $I_2(T_0)$, by using the Burkholder-Davis-Gundy inequality (cf. \cite[Chapter IV, Theorem 4.1 and Corollary 4.2, pp. 160]{revuz1999continuous}) component-wise for the vector-valued, continuous martingale $\{\hat{M}^{(n)}(t)\}$ that vanishes at $0$, and the fact that $\sigma$ is Lipschitz on $E_0^{(n)}$, 
\begin{align}\label{ineq:sde_strong_sol_delta_3}
I_2(T_0)&\leq c_1\cdot \E \pnorm{\langle \hat{M}^{(n)} \rangle_{T_0}}{}\leq c_1 \int_0^{T_0} \E\sup_{r \in [0,s]}\pnorm{\Delta_r^{(n)}}{}^2\,\d{s},\quad \forall T_0\in [0,T].
\end{align}
Combining (\ref{ineq:sde_strong_sol_delta_1})-(\ref{ineq:sde_strong_sol_delta_3}),  with $H^{(n)}(s)\equiv \E \sup_{t \in [0,s]}\pnorm{\Delta_t^{(n)}}{}^2$, we have 
\begin{align*}
H^{(n)}(T_0)\leq c_1 \int_0^{T_0} H^{(n)}(s)\,\d{s},\quad \forall T_0 \in [0,T].
\end{align*}
Using Gr\"onwall's inequality (cf. Lemma \ref{lem:gronwall_ineq}), we then conclude $H^{(n)}(T_0)=0$ for all $T_0 \in [0,T]$. In particular, 
\begin{align*}
\sup_{t \in [0,T]} \abs{X_{t\wedge \hat{\tau}^{(n)}}-X'_{t\wedge \hat{\tau}^{(n)}}}=0,\quad \hbox{a.s.}
\end{align*}
The almost sure pathwise uniqueness on $[0,T]$ now follows by letting $n\to \infty$ in  the above display with the help of (\ref{ineq:sde_strong_sol_tau_hat}).  \qed

\section{Proofs for Section \ref{section:proof_outline_optimal}}\label{section:proof_optimal}

\subsection{Proof of Lemma \ref{lem:r_lim}}\label{subsection:proof_r_lim}

	For an optimal arm $a \in \mathcal{A}_0$, let $p_{a;T}:\R_{>0}^K\times \R^K\to \R_{\geq 0}$ be defined as follows: for $\big(r_{[K]},\xi_{[K]}\big)\in \R_{>0}^K\times \R^K$ and $\{\mathsf{Z}_{[K]}\} \stackrel{\mathrm{i.i.d.}}{\sim} \mathsf{Z}$,
	\begin{align*}
	&p_{a;T}\big(r_{[K]},\xi_{[K]}\big)\\
	&\equiv \Prob_{\mathsf{Z}}\bigg(-\sqrt{T}\Delta_b+\frac{\xi_b}{r_b}+ \frac{\mathsf{Z}_b}{\sqrt{r_b}} \leq -\sqrt{T}\Delta_a+ \frac{\xi_a}{r_a}+ \frac{\mathsf{Z}_a}{\sqrt{r_a}},\forall b\in [K]\setminus \{a\}\bigg)\\
	& = \E_{\mathsf{Z}} \prod_{b \in [K]\setminus \{a\}} \Phi\,\bigg[\sqrt{r_b}\cdot \bigg(\frac{\xi_a}{r_a}-\frac{\xi_b}{r_b}+\frac{\mathsf{Z}}{\sqrt{r_a}}\bigg)+ \sqrt{r_b}\cdot \sqrt{T}(\Delta_b-\Delta_a)\bigg].
	\end{align*}
	With the definitions of  $(r_{\cdot;T},\xi_{\cdot;T})$ in (\ref{def:r_xi_T}) and of the conditional probability $\mathfrak{p}_{a;t}^\xi$ in (\ref{def:cond_prob}), we may write
	\begin{align*}
	\mathfrak{p}_{a;t}^\xi=\Prob^\xi(A_t=1|\mathscr{F}_{t-1}^Z)& = p_{a;T}\big(r_{[K];T}(z_{t-1}^{(T)}), \xi_{[K];T}(z_{t-1}^{(T)})\big),\quad t \in [T].
	\end{align*}
	In the arguments below, we shall fix $\epsilon \in (0,1/2]$.
	
	\noindent (\textbf{Step 1}). In this step, we show that for any $x\geq 1$ and $\{\xi_{a;i}\} \in E_\xi(x)$, there exist some constants $c_1=c_1>(K,\Delta,\mathscr{L}(\mathsf{Z}))>0$ and  $c_2=c_2(\epsilon,x,K,\Delta,\mathscr{L}(\mathsf{Z}))>0$ such that if $T\geq c_2$,
	\begin{align}\label{ineq:r_conv_step1}
	&\Prob^\xi\bigg(\max_{t\in [\epsilon T:T]}\bigabs{\mathfrak{p}_{a;t}^\xi-p_{a}\big(r_{\mathcal{A}_0;T}(z_{t-1}^{(T)}), \xi_{\mathcal{A}_0;T}(z_{t-1}^{(T)})\big) }\nonumber\\
	&\qquad > 2\Prob\big(\abs{\mathsf{Z}}>\{\bar{\Phi}_\ast^{-}(1\wedge c_2T^{-1/2})\}^{1/2}/c_2 \big)\bigg)\leq c_1 \exp\big(-1/\{c_1\bar{\Phi}_\ast(c_1x)\}\big).
	\end{align}
	To prove (\ref{ineq:r_conv_step1}), note that by Lemma \ref{lem:apriori_arm_pull_lower_bound}, for all $\{\xi_{a;i}\} \in E_\xi(x)$, there exists some event $\mathcal{E}_x$ with $\Prob^\xi(\mathcal{E}_x^c)\leq c_1 \exp(-1/\{c_1\bar{\Phi}_\ast(c_1x)\})$ on which $\min_{a \in [K]} n_{a;t}\geq \big[\bar{\Phi}_\ast^{-}(t^{-1/2})/(c_1 x)\big]^2 $ for all $t\geq 2/[\bar{\Phi}_\ast(c_1x)]^2$. So on the event $\mathcal{E}_x$, for some $c_2=c_2(\epsilon,x,K,\Delta,\mathscr{L}(\mathsf{Z}))>0$, if $t\geq c_2$,
	\begin{align*}
	&\bigabs{\mathfrak{p}_{a;t}^\xi-p_{a}\big(r_{\mathcal{A}_0;T}(z_{t-1}^{(T)}), \xi_{\mathcal{A}_0;T}(z_{t-1}^{(T)})\big)}\\
	&\leq \E_{\mathsf{Z}}\biggabs{1- \prod_{b \in \mathcal{A}_+} \Phi\bigg[{r_{b;T}^{1/2}(z_{t-1}^{(T)}) }\cdot \bigg(\frac{\xi_{a;T}(z_{t-1}^{(T)}) }{r_{a;T}(z_{t-1}^{(T)}) }-\frac{\xi_{b;T}(z_{t-1}^{(T)}) }{r_{b;T}(z_{t-1}^{(T)}) }+\frac{\mathsf{Z}}{{r_{a;T}^{1/2}(z_{t-1}^{(T)})}}+\sqrt{T}\Delta_b\bigg)\bigg] }\\
	&\leq \E_{\mathsf{Z}}\biggabs{1-   \Phi^{\abs{\mathcal{A}_+}} \bigg[ c_2^{-1} \bar{\Phi}_\ast^{-}(2t^{-1/2})\cdot \bigg(\Delta_b-\frac{c_2(1\vee \abs{\mathsf{Z}}) }{\{\bar{\Phi}_\ast^{-}(2t^{-1/2})\}^{1/2} }\bigg) \bigg]}.
	\end{align*}
	Consequently, for $t\geq c_2$, by considering the regime $\abs{\mathsf{Z}}\leq \{\bar{\Phi}_\ast^{-}(2t^{-1/2})\}^{1/2}/c_2$ and $\abs{\mathsf{Z}}> \{\bar{\Phi}_\ast^{-}(2t^{-1/2})\}^{1/2}/c_2$, the RHS of the above display can be further bounded by
	\begin{align*}
	\abs{1-\Phi\big(\bar{\Phi}_\ast^{-}(2t^{-1/2})/c_2\big)}+\Prob\big(\abs{\mathsf{Z}}>\{\bar{\Phi}_\ast^{-}(2t^{-1/2})\}^{1/2}/c_2 \big)\leq 2\Prob\big(\abs{\mathsf{Z}}>\{\bar{\Phi}_\ast^{-}(2t^{-1/2})\}^{1/2}/c_2 \big),
	\end{align*}
	proving the claimed estimate in (\ref{ineq:r_conv_step1}).

	\noindent (\textbf{Step 2}).
	For $z \in [\epsilon,1]$, note that
	\begin{align*}
	r_{a;T}(z)-r_{a;T}(\epsilon)& = \frac{1}{T}\sum_{s=\floor{\epsilon T}+1}^{ \floor{z T}} \mathfrak{p}_{a;s}^\xi+\frac{1}{T}\sum_{s=\floor{\epsilon T}+1}^{ \floor{z T}} \big(\bm{1}_{A_s=a}-\mathfrak{p}_{a;s}^\xi\big)\equiv I_1(\epsilon,z)+I_2(\epsilon,z).
	\end{align*}
	By the martingale Bernstein's inequality (cf. Lemma \ref{ineq:mtg_bern}),
	\begin{align*}
	\Prob^\xi\bigg(\sup_{z \in [\epsilon,1]}\abs{I_2(\epsilon,z)}\geq c_1\epsilon_T \bigg)\leq \epsilon_T^{100}.
	\end{align*}
	On the other hand, (\ref{ineq:r_conv_step1}) in Step 1 yields that for any $\{\xi_{a;i}\}  \in E_\xi(x)$, if $T\geq c_2$,
	\begin{align*}
	&\Prob^\xi\bigg( \sup_{z \in [\epsilon,1]}\biggabs{I_1(\epsilon,z)-\frac{1}{T}\sum_{s=\floor{\epsilon T}+1}^{ \floor{z T}} p_{a}\big(r_{\mathcal{A}_0;T}(z_{s-1}^{(T)}), \xi_{\mathcal{A}_0;T}(z_{s-1}^{(T)})\big) }\\
	&\qquad \qquad  >2\Prob\big(\abs{\mathsf{Z}}>\{\bar{\Phi}_\ast^{-}(1\wedge c_2T^{-1/2})\}^{1/2}/c_2 \big)\bigg)\leq c_1 \exp\big(-1/\{c_1\bar{\Phi}_\ast(c_1x)\}\big).
	\end{align*}
	Combining the above estimates concludes the desired estimate. The claim for the case $\Delta_a\equiv 0$ for all $a \in [K]$ is much simpler without the need of Step 1.\qed

\subsection{Proof of Lemma \ref{lem:xi_mtg}}\label{subsection:proof_xi_mtg}

	We note that $\{\xi_{a;T}(\cdot)\}$ is a martingale with respect to $\{\mathscr{F}_{\floor{zT}}: z \in [0,1]\}$, and its the predictable quadratic variation of $\{\xi_{a;T}(\cdot)\}$ is 
	\begin{align*}
	\langle\xi_{a;T}\rangle_z \equiv \sum_{s=1}^{\floor{zT}} \E \big[\big(\xi_{a;T}(s/T)-\xi_{a;T}((s-1)/T)\big)^2|\mathscr{F}_{s-1}\big] = \frac{1}{T} \sum_{s=1}^{\floor{zT}} \bm{1}_{A_s=a}=\bar{r}_{a;T}(z).
	\end{align*}
	The claim follows by uniqueness in Doob decomposition theorem.
	\qed

\subsection{Proof of Proposition \ref{prop:compact_conv_sde}}\label{subsection:proof_compact_conv_sde}
	
	Note that the sequence $\{(r_{a;T},\xi_{a;T}):a \in \mathcal{A}_0\}$ is tight in $(\ell^\infty[0,1])^{\mathcal{A}_0\times \mathcal{A}_0}$. In particular, tightness of $\{r_{a;T}\}$ is trivial due to the boundedness $r_{a;T} \in [0,1]$, and tightness of $\{\xi_{a;T}\}$ follows by noting that 
	\begin{align*}
	\sup_{z \in [0,1]}\abs{\xi_{a;T}(z)}\leq T^{-1/2}\max_{t \in [T]} \biggabs{\sum_{s \in [t]}\xi_s}=\bigop(1).
	\end{align*}
	Here the stochastic boundedness follows from, e.g., L\'evy's maximal inequality (cf. \cite[Theorem 1.1.1]{de2012decoupling}). By Prokhorov's theorem, let $\{(r_a,\xi_a):a \in \mathcal{A}_0\}$ be a sequential limit of $\{(r_{a;T},\xi_{a;T}):a \in \mathcal{A}_0\}$ in $(\ell^\infty[0,1])^{\mathcal{A}_0\times \mathcal{A}_0}$. By Skorokhod’s representation theorem (cf. \cite[Theorem 3.7.25]{gine2015mathematical}), we assume without loss of generality that on an event $E_0$ with probability $1$,
	\begin{align*}
	\sup_{z \in [0,1]} \abs{r_{a;T}(z)-r_a(z)} + \sup_{z \in [0,1]} \abs{\xi_{a;T}(z)-\xi_a(z)} \to 0.
	\end{align*}
	\noindent (\textbf{Step 1}). In this step, we shall show that for any $\epsilon>0$, $\{(r_{a}(z),\xi_{a}(z)):a \in \mathcal{A}_0, z \in [\epsilon,1]\}$ satisfies the stochastic differential equation (\ref{def:sde_singular}). 
	
	To this end, first note that by Lemma \ref{lem:r_lim}, on the event $E_0$, for any $\epsilon>0$, as $r_a(\epsilon)>0$ w.p.1, we have almost surely
	\begin{align}\label{ineq:compact_conv_sde_step1_1}
	r_a(z)-r_a(\epsilon)=\int_\epsilon^z p_a\big(r_{\cdot}(x),\xi_{\cdot}(x)\big)\,\d{x},\quad \forall a \in \mathcal{A}_0,\, z \in [\epsilon,1].
	\end{align}
	This gives the first equation of (\ref{def:sde_singular}).

	Next, by Lemma \ref{lem:xi_mtg} and the tightness of $\{\xi_{a;T}\}$, taking limit as $T \to \infty$ proves that both $\{\xi_a(z):z \in [0,1]\}$ and $\{\xi_a^{2}(z)-r_a(z):z \in [0,1]\}$ are continuous martingale with respect to the natural filtration. Note that $z\mapsto \bar{r}_{a;T}(z)$ is $1$-Lipschitz, so its limit $r_a(\cdot)$ in $\ell^\infty[0,1]$ must also be $1$-Lipschitz (in particular, absolutely continuous). Further note that in view of the proven (\ref{ineq:compact_conv_sde_step1_1}), $r_a'(z)>0$ w.p.1 for all $z \in (0,1]\cap \mathbb{Q}$, and therefore by continuity, $r_a'(z)>0$ w.p.1 for all $z \in (0,1]$. Therefore 
	\begin{align*}
		W_a(z)\equiv \int_0^z \frac{1}{\sqrt{r_a'(s)}}\,\d{\xi_a(s)}
	\end{align*}
	is well-defined w.p.1. Moreover, for any $z \in [0,1]$ and $a,b \in \mathcal{A}_0$, as $\langle \xi_a, \xi_{b} \rangle \equiv \delta_{ab} r_a$, a simple calculation shows that 
	\begin{align*}
	\langle W_a, W_{b} \rangle_z = \int_0^z \frac{ 1 }{ \{r_a'(s)r_b'(s)\}^{1/2} }  \,\d{\langle \xi_a, \xi_{b} \rangle_s}= \delta_{ab}\cdot \int_0^z \frac{ 1 }{ r_a'(s) } \,\d{r_a(s)} = \delta_{ab}\cdot z. 
	\end{align*}
	By L\'evy's characterization for Brownian motion, $\{W_a: a \in \mathcal{A}_0\}$ is a standard $\abs{\mathcal{A}_0}$-dimensional Brownian motion starting from $0$. 
	
	Finally, for fixed $a \in \mathcal{A}_0$, by the associativity of stochastic integral, cf., \cite[Chapter IV, Proposition 2.4, pp. 139]{revuz1999continuous}, with $h_s\equiv 1/\sqrt{r_a'(s)}$, $g_s\equiv \sqrt{r_a'(s)}$, and simplified notation $\xi_s\equiv \xi_a(s)$ and $W_s\equiv W_a(s)$, using $\cdot$ to denote stochastic integral with respect to a continuous martingale,
	\begin{align*}
	g\cdot W = g \cdot (h\cdot \xi) = (gh)\cdot \xi = 1\cdot \xi,\quad\hbox{a.s.} 
	\end{align*}
	Combined with (\ref{ineq:compact_conv_sde_step1_1}), we obtain the second equation of (\ref{def:sde_singular}).
	
	\noindent (\textbf{Step 2}). In this step, we shall perform the time change and renormalization. Recall $\{(u_{a;T},w_{a;T}):a \in \mathcal{A}_0\}$ defined in (\ref{def:u_w}), and let 
	\begin{align*}
	\begin{cases}
	u_{a}(t)\equiv e^{-t}\cdot r_{a}(e^t),\\
	w_{a}(t)\equiv e^{-t/2}\cdot \xi_{a}(e^t), 
	\end{cases}
	\forall a \in \mathcal{A}_0,\, t \in (-\infty,0].
	\end{align*}
	Then by the first line of (\ref{def:sde_singular}) and the definition of $p_a$ in (\ref{def:p_a}), we obtain the differential equation for $u_a$:
	\begin{align*}
	\d u_{a}(t) & = \d \big(e^{-t} r_{a}(e^t) \big) = e^{-t} \d {r_{a}(e^t) }-e^{-t} r_{a}(e^t)\,\d{t} \\
	& = \big(p_a(e^t u_\cdot(t), e^{t/2} w_\cdot(t) )-u_a(t)\big)\,\d{t} = \big(p_a(u_\cdot(t), w_\cdot(t) )-u_a(t)\big)\,\d{t}.
	\end{align*}
	Next we shall find the stochastic differential equation for $w_a$. For notational simplicity, we write $z\equiv e^t$. By It\^o's formula (cf. \cite[Chapter IV, Theorem 3.3, pp. 147]{revuz1999continuous}) applied to $X\equiv (\xi_a(z),z)$ and $F(x_1,x_2)=x_1/\sqrt{x_2}$, as for any $\epsilon>0$ and $z \in [\epsilon,1]$, 
	\begin{itemize}
		\item $F(\xi_a(z),z)- F(\xi_a(\epsilon),\epsilon)= \xi_a(z)/\sqrt{z}-\xi_a(\epsilon)/\sqrt{\epsilon}$,
		\item $\int_\epsilon^z \partial_1 F(\xi_a(u),u)\,\d{\xi_a(u)}= \int_\epsilon^z  u^{-1/2}\,\d{\xi_a(u)}=\int_\epsilon^z \sqrt{p_a(r_\cdot(u),\xi_\cdot(u))/u}\,\d{W_a(u)}$,
		\item $\int_\epsilon^z \partial_2 F(\xi_a(u),u)\,\d{u}= - \frac{1}{2}\int_\epsilon^z \xi_a(u)/u^{3/2}\,\d{u}$,
		\item $\int_\epsilon^z \partial_{11} F(\xi_a(u),u)\,\d{\langle\xi_a\rangle_u}=0$,
	\end{itemize}
	we have with $\epsilon=e^{t_\epsilon}$,
	\begin{align*}
	w_a(t)-w_a(t_\epsilon) &= \int_{t_\epsilon}^t {p_a^{1/2}(r_\cdot(e^s),\xi_\cdot(e^s))}\cdot e^{-s/2}\,\d{W_a(e^s)}-\frac{1}{2}\int_{t_\epsilon}^t \frac{\xi_a(e^s)}{e^{3s/2}}\,\d{e^s}\\
	& = \int_{t_\epsilon}^t {p_a^{1/2}(u_\cdot(s),w_\cdot(s))}\,\d{W_a'(s)}- \frac{1}{2}\int_{t_\epsilon}^t w_a(s)\,\d{s}.
	\end{align*}
	Here in the second identity, we have defined $W_a'(t)\equiv \int_0^t e^{-s/2}\,\d{W_a(e^s)}$ which is another standard, two-sided Brownian motion starting from 0, as $W_a'(0)=0$ and $\langle W_a'\rangle_t = \int_0^t e^{-s}\,\d{\langle W_a(e^\cdot)\rangle_s}=\int_0^t e^{-s}\,\d{e^s} = t$. This gives the desired stochastic differential equation for $w_a$. 
	
	\noindent (\textbf{Step 3}). Fix $T_0>0$. For any $t \in [0,T_0]$, let $(X(t))\equiv (u_a(-T_0+t),w_a(-T_0+t))_{a \in \mathcal{A}_0}$. Then by the definition of $(u_{a;T},w_{a;T})$ in (\ref{def:u_w}), any of its sequential limit must have the same marginal law, i.e., $\mathscr{L}(X(t_1))=\mathscr{L}(X(t_2))$ for any $t_1,t_2\in [0,T_0]$. We denote this law on $E_0$ as $\mu$, and we prove that $\mu$ is actually an invariant measure of the SDE (\ref{def:sde_time_change_renor}). Specifically, for $t \in [0,T_0]$, note that
	\begin{align*}
	(P_t^\ast \mu)(A) & = \int_{E_0} P_t(x,A)\,\mu(\d x) = \int_{E_0} \Prob\big(X(t) \in A|X(0)=x\big)\,\mu(\d x)\\
	& = \E \Prob\big(X(t) \in A|X(0)\big) = \Prob(X(t) \in A) \stackrel{(\ast)}{=} \mu(A). 
	\end{align*}
	Here in $(\ast)$ we used the fact that marginal laws remain $\mu$. As $T_0$ can be taken as arbitrary, it follows $P_t^\ast \mu = \mu$ for all $t\geq 0$. \qed

\subsection{Proof of Theorem \ref{thm:Doob}}\label{subsection:proof_Doob}

	We shall verify the uniqueness criteria presented in \cite[Corollary 7.8]{hairer2008ergodic}. In particular, due to the already assumed strong Feller property (\texttt{P}1), we only need to verify the existence of an \emph{accessible point} $x \in E_0$ in the sense of \cite[Definition 7.2]{hairer2008ergodic}: for any $y \in E_0$, any open neighborhood $U$ of $x$ and every $\lambda>0$,
	\begin{align}\label{ineq:doob_1}
	R_\lambda(y,U)\equiv \lambda \int_0^\infty e^{-\lambda t} P_t(y,U)\,\d{t}>0.
	\end{align}
	By the irreducibility (\texttt{P}2), there exists some $t_0\geq 0$ such that $P_{t_0}(y,U)>0$. By the continuity of $t \mapsto P_t(y,U)$ due to the continuity of $t \mapsto X(t)$, there exists a compact interval $t_0 \in I\subset [0,\infty)$ with $\mathrm{int}(I) \neq \emptyset$ such that $\inf_{t \in I} P_t(y,U)\equiv \delta_0>0$. This means that we may restrict the integral (\ref{ineq:doob_1}) on $I$ with $R_\lambda(y,U)\geq \lambda e^{-\lambda I_{\max}} \abs{I} \delta_0>0$, where $I_{\max}\equiv \sup_{z \in I} z$. This verifies (\ref{ineq:doob_1}), and therefore establishes the accessibility for (any) $x \in E_0$.  \qed

\subsection{Proof of Lemma \ref{lem:Ito_Stratonovich}}\label{subsection:proof_Ito_Stratonovich}

	Let $X(t)\equiv(u_a(t),w_a(t))_{a \in \mathcal{A}_0}$, $H_a(t)\equiv \sqrt{p_a(u_\cdot(t),w_\cdot(t))}$, and $F_a(u_\cdot,w_\cdot)\equiv \sqrt{p_a(u_\cdot,w_\cdot)}$. Then $H_a(t) = F_a(X(t))$. By It\^o's formula (cf. \cite[Chapter IV, Theorem 3.3, pp. 147]{revuz1999continuous}), 
	\begin{align*}
	\d{ F_a(X(t))} & = \sum_{b \in \mathcal{A}_0} \Big(\partial_{u_b} F_a(X(t))\,\d{u_b(t)}+ \partial_{w_b} F_a(X(t))\,\d{w_b(t)}\Big)\\
	&\qquad + \frac{1}{2} \sum_{b,b' \in \mathcal{A}_0} \sum_{\#,\& \in \{u,w\} }\partial_{\#_b \&_{b'}} F_a(X(t))\,\d{\iprod{\#_b}{\&_{b'}}_t}\equiv I_1+I_2.
	\end{align*}
	For $I_1$, as $\partial_{u_b} F_a =\partial_{u_b} p_a/(2\sqrt{p_a})$ and $\partial_{w_b} F_a =\partial_{w_b} p_a/(2\sqrt{p_a})$, 
	\begin{align*}
	I_1& = \big(\cdots\big)\,\d{t}+ \frac{1}{2}\sum_{b \in \mathcal{A}_0} \partial_{w_b}  p_a (u_\cdot(t),w_\cdot(t))\,\d{B_b(t)}.
	\end{align*}
	For $I_2$, note that $u_b$ has finite variation, and $\iprod{w_b}{w_{b'}}_\cdot=\delta_{bb'}\int_0^\cdot p_a (u_\cdot(t),w_\cdot(t))\,\d{t}$, 
	\begin{align*}
	I_2& = \frac{1}{2}\sum_{b \in \mathcal{A}_0} \partial_{w_b}^2 F_a(X(t)) \cdot  p_a (u_\cdot(t),w_\cdot(t))\,\d{t}=\big(\cdots\big)\,\d{t}.
	\end{align*}
	Combining the above displays, we arrive at
	\begin{align*}
	\d{H_a(t)} = \big(\cdots\big)\,\d{t}+ \frac{1}{2}\sum_{b \in \mathcal{A}_0} \partial_{w_b}  p_a (u_\cdot(t),w_\cdot(t))\,\d{B_b(t)},
	\end{align*}
	which therefore entails $
	\d \iprod{H_a}{B_a}_t  =\frac{1}{2}\partial_{w_a}  p_a (u_\cdot(t),w_\cdot(t))\,\d{t}$. The claim now follows by using the standard relation $
	H_a(t)\circ \d{B_a(t)} =  H_a(t)\, \d{B_a(t)}+ \frac{1}{2} \d \iprod{H_a}{B_a}_t$.\qed

\subsection{Proof of Proposition \ref{prop:sde_strong_feller}}\label{subsection:proof_sde_strong_feller}

For ease of notation in this proof, we write $\mathcal{A}_0\equiv \{1,\ldots,m\}$. Note that the open simplex $\Delta_{m}^\circ\equiv \Delta_{\mathcal{A}_0}^\circ$ (defined in (\ref{def:Delta_circ})) is diffeomorphic to $\blacktriangle_{m-1}^\circ\equiv \{u\in (0,1)^{m-1}: \pnorm{u}{1}<1\}$ by simply dropping the last coordinate, we may identify the state space $E_0$ as $E_0=\blacktriangle_{m-1}^\circ\times \R^m$, and $\{p_a\}_{a \in [m-1]}$ as functions defined on $\blacktriangle_{m-1}^\circ\times \R^m$ in that $p_a\big(u_{[m-1]},1-\sum_{a \in [m-1]} u_a, w_\cdot\big)$, and is $p_m$ identified as $1-\sum_{a \in [m-1]} p_a$. The semigroup $(P_t)$ is also identified accordingly.

As $\{p_a\}_{a \in [m-1]}$ may have blowing up derivatives near the boundary of $E_0$, we consider the localized space $E_0^{(n)}\subset E_0$ defined by 
\begin{align}\label{ineq:strong_feller_localized_space}
E_0^{(n)}\equiv \big\{u_\cdot \in [1/n,1]^{m-1}, \pnorm{u}{1}\leq 1-1/n\big\}\times [-n,n]^{m},\quad n=1,2,\ldots.
\end{align}
Let $\blacktriangle_{m-1}^{\circ,(n)}\equiv \big\{u_\cdot \in [1/n,1]^{m-1}, \pnorm{u}{1}\leq 1-1/n\big\}\subset \blacktriangle_{m-1}^\circ$. Let $\chi_{[u]}^{(n)}: \R^{m-1}\to \R^{m-1}$ and $\chi_{[w]}^{(n)}: \R\to \R$ be smooth maps such that (i) $\chi_{[u]}^{(n)}(\R^{m-1})\subset \blacktriangle_{m-1}^{\circ,(2n)}$, $\chi_{[u]}^{(n)}|_{\blacktriangle_{m-1}^{\circ,(n)}}=\mathrm{id}$, and (ii) $\chi_{[w]}^{(n)}(\R)\subset [-2n,2n]$, $\chi_{[w]}^{(n)}|_{[-n,n]}=\mathrm{id}$, and its first derivative is globally strictly positive on $\R$. Then let $\chi^{(n)}: \R^{2m-1}\to E_0$ be defined via
\begin{align*}
\chi^{(n)}(u_\cdot,w_\cdot)= \big(\chi^{(n)}_{[u]}(u_\cdot),\chi^{(n)}_{[w]}(w_1),\ldots,\chi^{(n)}_{[w]}(w_m)\big).
\end{align*}
Clearly, the image of $\chi^{(n)}$ satisfies $\chi^{(n)}(\R^{2m-1})\subset E_0^{(2n)}$ by definition.

Recall $b,\sigma$ defined in (\ref{def:sde_b_sigma}) and hereafter identified as functions on $E_0$. We then define their smoothed versions as
\begin{align}\label{ineq:strong_feller_b_sigma}
\begin{cases}
b^{(n)}(u_\cdot,w_\cdot)\equiv \big( (p_a(\chi^{(n)}(u_\cdot,w_\cdot))-u_a)_{a \in [m-1]},(-w_a/2)_{a \in [m]}\big),\\
\sigma^{(n)}(u_\cdot,w_\cdot)\equiv \mathrm{diag}\big(0_{m-1}, ({p_a^{1/2}(\chi^{(n)}(u_\cdot,w_\cdot))})_{a \in [m]}\big).
\end{cases}
\end{align}
As the derivatives of, $b^{(n)},\sigma^{(n)}$ have bounded derivatives of all orders (the constants of which may depend on $n$), and therefore the SDE in $\R^{2m-1}$
\begin{align}\label{ineq:strong_feller_sde}
\d{X^{(n)}(t)} = b^{(n)}(X^{(n)}(t))\,\d{t}+\sigma^{(n)}(X^{(n)}(t))\,\d{B(t)},\quad \forall t\in \R
\end{align}
admits a unique strong solution $X^{(n)}(t)=\big(u_\cdot^{(n)}(t), w_\cdot^{(n)}(t)\big)\in E_0$ on $[0,\infty)$ with initial condition $X^{(n)}(0)=(u_\cdot(0),w_\cdot(0)) \in E_0$. We use $X(t)=(u_\cdot(t),w_\cdot(t))$ to denote the unique strong solution to the SDE (\ref{def:sde_time_change_renor}) (identified as diffusion in $\R^{2m-1}$ as above) with the same initial condition and the same Brownian motion $B$.

\noindent (\textbf{Step 1}). Let $(P_t^{(n)})$ be the semigroup associated with $(X^{(n)}(t))$ in (\ref{ineq:strong_feller_sde}). In this step, we shall establish the strong Feller property of the semigroup $(P_t^{(n)})$ for every localization level $n$ and at any time $t$.

The Stratonovich form of the SDE (\ref{ineq:strong_feller_sde}) is given by 
\begin{align}\label{ineq:strong_feller_sde_Stratonovich}
\begin{cases}
\d{u_j^{(n)}} = \big\{p_j^{(n)}(u^{(n)}_\cdot,w^{(n)}_\cdot)- u_j^{(n)}\big\}\,\d{t},\\
\d{w_k}^{(n)} = \big\{-\frac{1}{2} w_k^{(n)} -\frac{1}{4}\partial_{w_k} p_k^{(n)}\big(u^{(n)}_\cdot,w^{(n)}_\cdot\big)\big\}\,\d{t}\\
\qquad \qquad\qquad +{p_k^{(n),1/2}\big(u^{(n)}_\cdot,w^{(n)}_\cdot\big)}\circ \d{B_{m-1+k}},
\end{cases}
\forall j \in [m-1], k \in [m].
\end{align}
In the notation of Definition \ref{def:hormander_cond}, by identifying the smooth vector fields $V_{0}^{(n)}:\R^{2m-1}\to \R^{2m-1}$ with elements in $C^\infty(\R^{2m-1}, T \R^{2m-1})$, for any $x=(u_\cdot,w_\cdot) \in \R^{2m-1}$,
\begin{align}\label{ineq:strong_feller_sde_step1_V}
V_{0}^{(n)}(x)&\equiv \sum_{j \in [m-1]} \big\{p_j^{(n)}\big(u_\cdot,w_\cdot\big)-u_j\big\}\, \partial_{u_j}+ \sum_{k \in [m]} \bigg\{-\frac{1}{2} w_k -\frac{1}{4}\partial_{w_k} p_k^{(n)}(u_\cdot,w_\cdot)\bigg\}\, \partial_{w_k},\nonumber\\
V_{k}^{(n)}(x)&\equiv p_k^{(n),1/2}(u_\cdot,w_\cdot)\, \partial_{w_k},\quad \forall k \in [m].
\end{align}
In order to verify the parabolic H\"ormander's condition (\ref{cond:hormander}), it suffices to prove
\begin{align}\label{ineq:strong_feller_sde_step1_hormander}
\mathrm{span}\bigg\{\big\{V_{k}^{(n)}(x)\big\}_{k\in [m]}\bigcup \big\{\big[V_{k}^{(n)},V_{0}^{(n)}\big](x)\big\}_{k\in [m]}\bigg\} = \R^{2m-1},\quad \forall x \in \R^{2m-1}.
\end{align}
In view of the form of $\{V_{k}^{(n)}(x)\}$ in (\ref{ineq:strong_feller_sde_step1_V}), the claim (\ref{ineq:strong_feller_sde_step1_hormander}) follows provided that 
\begin{align}\label{ineq:strong_feller_sde_step1_hormander_1}
\mathrm{span}\Big\{\big\{\big[V_{k}^{(n)},V_{0}^{(n)}\big]^{[u]}(x)\big\}_{k\in [m]}\Big\} = \mathrm{span}\big\{\partial_{u_1},\ldots, \partial_{u_{m-1}}\},\quad \forall x \in \R^{2m-1}.
\end{align}
Here for a vector field $U\in C^\infty(\R^{2m-1}, T \R^{2m-1})$, we write  $U=\sum_{j \in [m-1]} U_j^{[u]}\partial_{u_j}+\sum_{k \in [m]} U_k^{[w]} \partial_{w_k}\equiv U^{[u]}+U^{[w]}$.

\noindent \emph{Claim 1}. The Lie bracket  $\big[V_{k}^{(n)},V_{0}^{(n)}\big]^{[u]}$ is given by 
\begin{align}\label{ineq:strong_feller_sde_step1_Lie_bracket}
\big[V_{k}^{(n)},V_{0}^{(n)}\big]^{[u]}(x)=  \sum_{j \in [m-1]} \big(p_k^{(n),1/2} \partial_{w_k} p_j^{(n)}\big)(u_\cdot,w_\cdot)\, \partial_{u_j}.
\end{align}
The proof will be deferred towards the end. To compute the span of all these Lie brackets, it suffices to consider the matrix
\begin{align*}
\mathsf{M}^{(n)}(x)\equiv \Big(\big(p_k^{(n),1/2} \partial_{w_k} p_j^{(n)}\big)(u_\cdot,w_\cdot)\Big)_{j \in [m-1],k\in [m]} \in \R^{(m-1)\times m}.
\end{align*} 
\noindent \emph{Claim 2}. For $x \in \R^{2m-1}$ and $\mathsf{M}^{(n)}(x)$ defined above, we have
\begin{align}\label{ineq:strong_feller_sde_step1_M_rank}
\mathrm{rank}\big(\mathsf{M}^{(n)}(x)\big)= m-1.
\end{align}
The proof will  again be deferred towards the end.  

Combining Claims 1 and 2 above proves (\ref{ineq:strong_feller_sde_step1_hormander_1}), and thereby (\ref{ineq:strong_feller_sde_step1_hormander}). Consequently, we may apply H\"olmander's theorem in the form of Theorem \ref{thm:hormander} to conclude the strong Feller property of $(P_t^{(n)})$.

\noindent (\textbf{Step 2}). In this step, we prove that for any $t>0$, any compact $\mathcal{X} \subset E_0$ and any test function $f \in B(E_0)$,
\begin{align}\label{ineq:strong_feller_sde_step2}
\lim_{n \to \infty}\sup_{x \in \mathcal{X}} \abs{P_t f(x)-P_t^{(n)}f(x)} = 0.
\end{align}
Let the stopping times $\{\tau^{(n)}\}$ be defined by
\begin{align*}
\tau^{(n)}\equiv \inf\{t\geq 0: X(t) \notin E_0^{(n)}\},\quad \forall n=1,2,\ldots,
\end{align*}
with the convention that $\inf \emptyset = \infty$. As $(X(t))$ and $(X^{(n)}(t))$ share the same initial condition $x\equiv X(0)=X^{(n)}(0)=(u_{\cdot}(0),w_{\cdot}(0)) \in E_0$, by taking $n_0$ large enough such that $x \in \mathcal{X}\subset \mathrm{int}(E_0^{(n)})$ for $n\geq n_0$, w.p.1 we have $\tau^{(n)}>0$ and  $X(t)=X^{(n)}(t)$ for all $t \in [0,\tau^{(n)})$. Then for any $x \in \mathcal{X}$, 
\begin{align*}
\abs{P_t f(x)-P_t^{(n)}f(x)}&=\bigabs{\E_x \big(f(X(t))-f(X^{(n)}(t))\big)(\bm{1}_{\tau^{(n)}\leq t}+\bm{1}_{\tau^{(n)}> t})}\\
& \leq 2\pnorm{f}{\infty}\cdot \Prob_x(\tau^{(n)}\leq t).
\end{align*}
So in order to prove (\ref{ineq:strong_feller_sde_step2}), it suffices to prove
\begin{align}\label{ineq:strong_feller_sde_step2_tau}
\lim_{n \to \infty}\sup_{x \in \mathcal{X}}\Prob_{x} (\tau^{(n)}\leq t)=0.
\end{align}
To this end, using exactly the same argument as in the estimate (\ref{ineq:sde_strong_sol_u_lb}) and the natural identification for $u_m(\cdot), u_m^{(n)}(\cdot)$, with $\delta_0\equiv \min_{a \in [m]} u_a(0) e^{-t}$, w.p.1 we have
\begin{align}\label{ineq:strong_feller_sde_step2_lb}
\min_{a \in [m]}\inf_{s \in [0,t]} u_a^{(n)}(s)\wedge u_a(s)\geq \delta_0,\quad \forall n=1,2,\ldots.
\end{align}
This means that for $n$ large enough, $u,u^{(n)}\in \blacktriangle_{m-1}^{\circ,(n)}$ and therefore
\begin{align*}
\sup_{x \in \mathcal{X}}\Prob_{x} (\tau^{(n)}\leq t)\leq  \sup_{x \in \mathcal{X}}\Prob_x\Big(\max_{a \in [m]}\sup_{s \in [0,t]} \abs{w_a(s)}>n\Big).
\end{align*}
Now following exactly the same proof of (\ref{ineq:sde_strong_sol_tau_infty}) (where the suprema over the initial condition $x \in \mathcal{X}$ amounts to a uniform upper bound on $\max_{a \in [m]}\abs{w_a(0)}$) shows that the RHS of the above display vanishes as $n \to \infty$. This proves (\ref{ineq:strong_feller_sde_step2_tau}), and therefore the claim (\ref{ineq:strong_feller_sde_step2}).

\noindent (\textbf{Step 3}). In this step, we shall combine the claims in Step 1 and Step 2 to conclude. 

To this end, we fix $t>0$ and take a sequence $\{x_n\}\subset E_0$ such that $x_n \to x \in E_0$. Let us fix a compact neighbor $\mathcal{X}\subset E_0$ of $x$, so there exists some $n_0$ such that $x_n \in \mathcal{X}$ for $n\geq n_0$. Then for any $f \in B(E_0)$, and $N\in \N$, if $n\geq n_0$,
\begin{align*}
\abs{P_t f(x_n)-P_t f(x)}&\leq \abs{P_t f(x_n)-P_t^{(N)} f(x_n)}+\abs{P_t^{(N)} f(x_n)-P_t^{(N)} f(x)}\\
&\qquad + \abs{P_t^{(N)} f(x)-P_t f(x)}\\
&\leq 2\sup_{x \in \mathcal{X}} \abs{P_t^{(N)} f(x)-P_t f(x)}+\abs{P_t^{(N)} f(x_n)-P_t^{(N)} f(x)}.
\end{align*}
Now first letting $n\to \infty$ and using the strong Feller property of $P_t^{(N)}$ proved in Step 1, the second term of the above display vanishes. Then letting $N \to \infty$ and using (\ref{ineq:strong_feller_sde_step2}) proved in Step 2 to conclude that $\lim_n \abs{P_t f(x_n)-P_t f(x)}=0$. \qed

\begin{proof}[Proof of (\ref{ineq:strong_feller_sde_step1_Lie_bracket}) in Claim 1]
For the test function $f_j(x)=f_j(u_\cdot,w_\cdot)=u_j$, using (\ref{ineq:strong_feller_sde_step1_V}) with natural identification, for any $x = (u_\cdot,w_\cdot) \in \R^{2m-1}$, we have 
\begin{align*}
V_{0}^{(n)}(f_j)(x) = p_j^{(n)}\big(u_\cdot,w_\cdot\big)-u_j,\quad V_{k}^{(n)}(f_j)(x) = 0.
\end{align*}
This means, 
\begin{align*}
V_{k}^{(n)}\big(V_{0}^{(n)}(f_j)\big)(x)& = p_k^{(n),1/2} \partial_{w_k}\big(V_{0}^{(n)}(f_j)\big)(x) = \big(p_k^{(n),1/2} \partial_{w_k} p_j^{(n)}\big)(u_\cdot,w_\cdot),\\
V_{0}^{(n)}\big(V_{k}^{(n)}(f_j)\big)(x)& = \sum_{\alpha \in [m-1]} (*)_\alpha \partial_{u_\alpha} \big(V_{k}^{(n)}(f_j)\big)(x)+ \sum_{\beta \in [m]} (*)_\beta\, \partial_{w_\beta} \big(V_{k}^{(n)}(f_j)\big)(x)=0.
\end{align*}
By definition of the Lie bracket (cf. Definition \ref{def:vector_field_Lie}), we have
\begin{align*}
\big[V_{k}^{(n)},V_{0}^{(n)}\big](f_j)(x) & = V_{k}^{(n)}\big(V_{0}^{(n)}(f_j)\big)(x)-V_{0}^{(n)}\big(V_{k}^{(n)}(f_j)\big)(x)\\
& = \big(p_k^{(n),1/2} \partial_{w_k} p_j^{(n)}\big)(u_\cdot,w_\cdot).
\end{align*}
Using the identification (\ref{eqn:Lie_bracket_iden}), we have proven (\ref{ineq:strong_feller_sde_step1_Lie_bracket}) in Claim 1.
\end{proof}

We need the following for the proof of (\ref{ineq:strong_feller_sde_step1_M_rank}) in Claim 2.

\begin{lemma}\label{lem:M_ker}
	Suppose a matrix $M \in \R^{m\times m}$ satisfies the following properties:
	\begin{enumerate}
		\item $M_{jk}<0$ for $k\neq j$.
		\item $M_{jk}>0$ for $k=j$.
		\item $\sum_{k \in [m]} M_{jk} =0 $ for all $j \in [m]$. 
	\end{enumerate}
	Then $\mathrm{Ker}(M)= \mathrm{span}(\bm{1})$.
\end{lemma}
\begin{proof}
	Property (3) entails that $\mathrm{span}(\bm{1})\subset \mathrm{Ker}(M)$, so we only prove the converse direction. Let $x \in \R^m$ be such that $M x =0$. Let $j^\ast$ be such that $x_{j^\ast}=\max_{i \in [m]} x_i$. Then using $M_{j^\ast j^\ast} = -\sum_{k\neq j^\ast} M_{j^\ast k}$ by property (3),
	\begin{align*}
	0 = (Mx)_{j^\ast} = \sum_{k \in [m]} M_{j^\ast k} x_k = \sum_{k\neq j^\ast} M_{j^\ast k} x_k + M_{j^\ast j^\ast} x_{j^\ast} = \sum_{k\neq j^\ast} M_{j^\ast k}(x_k-x_{j^\ast}). 
	\end{align*}
	By (1), $M_{j^\ast k}<0$ for $k\neq j^\ast$. Moreover, $x_k\leq x_{j^\ast}$. The above equation then yields that $M_{j^\ast k}(x_k-x_{j^\ast})=0\Leftrightarrow x_k=x_j^\ast$ for $k\neq j^\ast$. In other words, $x=c\bm{1}$ for some $c \in \R$, and therefore $\mathrm{Ker}(M)\subset \mathrm{span}(\bm{1})$. 
\end{proof}

\begin{proof}[Proof of (\ref{ineq:strong_feller_sde_step1_M_rank}) in Claim 2]
As $p_k^{(n)}(x)>0$ for any $x \in \R^{2m-1}$, it suffices to prove that for any $(u_\cdot,w_\cdot) \in E_0$, the full Jacobian matrix satisfies
\begin{align}\label{ineq:strong_feller_sde_step1_J_rank}
\mathrm{rank}\Big(\big(\partial_{w_k} p_j^{(n)}(u_\cdot,w_\cdot)\big)_{j,k \in [m]}\Big) = m-1.
\end{align}
Consider the change of coordinates $(u_\cdot,z_\cdot)\to (u_\cdot,w_\cdot/u_\cdot)$ (where $/$ denotes coordinate-wise division). Let
\begin{align}\label{ineq:strong_feller_sde_step1_pbar}
\bar{p}_j(u_\cdot,z_\cdot)\equiv \E_{\mathsf{Z}} \prod_{k\neq j} \Phi\,\bigg[\sqrt{u_k}\cdot \bigg(z_j-z_k+\frac{\mathsf{Z}}{\sqrt{u_j}}\bigg) \bigg].
\end{align}
Using the derivative relation $
\partial_{w_k} p_j(u_\cdot,w_\cdot) = u_k^{-1} \partial_{z_k} \bar{p}_j(u_\cdot,w_\cdot/u_\cdot)$, we have
\begin{align}\label{ineq:strong_feller_sde_step1_pbar_2}
&\partial_{w_k} p_j^{(n)}(u_\cdot,w_\cdot) = \partial_{w_k} p_j\big(\chi^{(n)} (u_\cdot,w_\cdot)\big) \nonumber\\
&= \frac{(\chi_{[w]}^{(n)})'(w_k) }{\chi_{[u];k}^{(n)}(u_k)}\cdot  \partial_{z_k}\bar{p}_j \bigg[ (\chi_{[u];\alpha}^{(n)}(u_\alpha))_{\alpha \in [m]}, \bigg(\frac{\chi_{[w]}^{(n)}(w_\beta)}{\chi_{[u];\beta}^{(n)}(u_\beta)}\bigg)_{\beta \in [m]}  \bigg].
\end{align}
As $(\chi_{[w]}^{(n)})'(w_k)> 0$ for all $k \in [m]$, in order to prove (\ref{ineq:strong_feller_sde_step1_J_rank}), it suffices to prove that for any permissible $(u_\cdot,z_\cdot)$,
\begin{align}\label{ineq:strong_feller_sde_step1_Jbar_rank}
\mathrm{rank}\Big(\overline{\mathsf{J}}(u_\cdot,z_\cdot)\equiv \big(\partial_{z_k} \bar{p}_j(u_\cdot,z_\cdot)\big)_{j,k \in [m]}\Big) = m-1.
\end{align}
To this end, we note some properties of $\bar{p}_j$:
\begin{enumerate}
	\item $\partial_{z_k} \bar{p}_j(u_\cdot,z_\cdot)<0$ for $k\neq j$.
	\item $\partial_{z_k} \bar{p}_j(u_\cdot,z_\cdot)>0$ for $k=j$.
	\item $\sum_{k \in [m]} \partial_{z_k} \bar{p}_j(u_\cdot,z_\cdot) = 0$ for all $j \in [m]$.
\end{enumerate}
Indeed, properties (1)-(2) follow directly from differentiating (\ref{ineq:strong_feller_sde_step1_pbar}), and property (3) follows from differentiating the identity $\bar{p}_j(u_\cdot,z_\cdot+c\bm{1})=\bar{p}_j(u_\cdot,z_\cdot)$ with respect to $c \in \R$. Now applying Lemma \ref{lem:M_ker} shows that $\mathrm{Ker}\big(\overline{\mathsf{J}}(u_\cdot,z_\cdot)\big)=\mathrm{span}(\bm{1})$, which therefore proves (\ref{ineq:strong_feller_sde_step1_Jbar_rank}), and thereby the desired claim (\ref{ineq:strong_feller_sde_step1_J_rank}). This completes the proof of (\ref{ineq:strong_feller_sde_step1_M_rank}) in Claim 2.
\end{proof}

\subsection{Proof of Proposition \ref{prop:sde_irreducible}}\label{subsection:proof_sde_irreducible}

We shall use the same notation and the identification as in the proof of Proposition \ref{prop:sde_strong_feller}, except that  $E_0^{(n)}$ is now given as
\begin{align}\label{ineq:sde_irreducible_E_0}
E_0^{(n)}\equiv \blacktriangle_{m-1}^{\circ,(n)} \times [-\kappa_n,\kappa_n]^{m},\quad \forall n=1,2,\ldots.
\end{align}
Here $\kappa_n\geq n$ will be determined later, and the smoothing function $\chi_{[w]}^{(n)}:\R \to \R$ is defined such that $\chi_{[w]}^{(n)}(\R)\subset [-2\kappa_n,2\kappa_n]$, $\chi_{[w]}^{(n)}|_{[-\kappa_n,\kappa_n]}=\mathrm{id}$. Note that all proofs in Proposition \ref{prop:sde_strong_feller} will go through with minor modifications with this slight change.

We fix a tolerance level $\epsilon \in (0,1/2)$ to be chosen later, and the constant $\kappa_n$ in (\ref{ineq:sde_irreducible_E_0}) will be chosen as
\begin{align}\label{ineq:sde_irreducible_kappa}
\kappa_n\equiv 1+\bigg(\frac{2mn}{\epsilon}\bigg)^{1/2}\vee n.
\end{align}

With this choice of $\kappa_n$, we claim that with $L_{m,n,\epsilon}\equiv\big(2mn/\epsilon\big)^{1/2}$, the vectors $\big\{\bar{w}_{(k),\cdot}^{(n)}\equiv L_{m,n,\epsilon}(1-2\cdot \bm{1}_{\cdot\neq k}): k \in [m]\big\}\subset \R^m$ satisfy
\begin{align}\label{ineq:sde_irreducible_step1_wbar}
\sup_{u \in \blacktriangle_{m-1}^{\circ,(n)} }\max_{k \in [m]} \max\Big\{1-p_k^{(n)}(u,\bar{w}_{(k)}^{(n)}),\max_{j\neq k} p_j^{(n)}(u,\bar{w}_{(k)}^{(n)})\Big\}\leq \epsilon.
\end{align}
To prove (\ref{ineq:sde_irreducible_step1_wbar}), with $\{Z_j: j \in [m]\}$ denoting i.i.d. variables with law $\mathsf{Z}$,
\begin{align*}
&1-p_k(u,\bar{w}_{(k)}^{(n)}) = 1-\Prob\bigg(\frac{\bar{w}_{(k),k}^{(n)}}{u_k}+\frac{Z_k}{\sqrt{u_k}}\geq \frac{\bar{w}_{(k),j}^{(n)}}{u_j}+\frac{Z_j}{\sqrt{u_j}},\quad \forall j\neq k\bigg)\\
&\leq \sum_{j\neq k} \Prob\bigg[\frac{Z_k}{\sqrt{u_k}}-\frac{Z_j}{\sqrt{u_j}}<-L_{m,n,\epsilon} \bigg(\frac{1}{u_k}+\frac{1}{u_j}\bigg) \bigg]\leq \frac{1}{L_{m,n,\epsilon}^2}\cdot \sum_{j\neq k} \frac{1}{(u_k^{-1}+u_j^{-1})},
\end{align*}
and for $j\neq k$,
\begin{align*}
&p_j(u,\bar{w}_{(k)}^{(n)})\leq \Prob\bigg( \frac{\bar{w}_{(k),j}^{(n)}}{u_j}+\frac{Z_j}{\sqrt{u_j}} \geq \frac{\bar{w}_{(k),k}^{(n)}}{u_k}+\frac{Z_k}{\sqrt{u_k}}\bigg)\\
&= \Prob\bigg[\frac{Z_j}{\sqrt{u_j}}-\frac{Z_k}{\sqrt{u_k}} \geq L_{m,n,\epsilon} \bigg(\frac{1}{u_k}+\frac{1}{u_j}\bigg)\bigg]\leq \frac{1}{L_{m,n,\epsilon}^2}\cdot \frac{1}{(u_k^{-1}+u_j^{-1})}.
\end{align*}
This means that the right hand side of the above two displays can be further bounded by $2mn/L_{m,n,\epsilon}^2=\epsilon$.  By the choice of $\kappa_n$ in (\ref{ineq:sde_irreducible_kappa}), we may ensure $\bar{w}_{(k)}^{(n)} \in [-\kappa_n,\kappa_n]^m$ and therefore $p_\cdot(u,\bar{w}_{(k)}^{(n)})\equiv p_\cdot^{(n)}(u,\bar{w}_{(k)}^{(n)})$ for all $u \in \blacktriangle_{m-1}^{\circ,(n)}$. This proves (\ref{ineq:sde_irreducible_step1_wbar}).

\noindent (\textbf{Step 1}). Fix a tolerance level $\epsilon>0$ and a point $x_0=(u_{x_0,\cdot},w_{x_0,\cdot}) \in O$. In this step, we shall construct, for every $n$ large enough, a square integrable function $h^{(n)}\equiv (h_1^{(n)},\ldots,h_m^{(n)})$ such that the solution $(u_\cdot^{(n)}(\cdot),w_\cdot^{(n)}(\cdot))$ to the ODE 
\begin{align}\label{ineq:sde_irreducible_ode}
\begin{cases}
\d{u_j^{(n)}} = \big\{p_j^{(n)}(u^{(n)}_\cdot,w^{(n)}_\cdot)- u^{(n)}_j\big\}\,\d{t},\\
\d{w_k}^{(n)} = \big\{-\frac{1}{2} w_k^{(n)} -\frac{1}{4}\partial_{w_k} p_k^{(n)}\big(u^{(n)}_\cdot,w^{(n)}_\cdot\big)\big\}\,\d{t}\\
\qquad \qquad\qquad +{p_k^{(n),1/2}\big(u^{(n)}_\cdot,w^{(n)}_\cdot\big)} h_k^{(n)} \, \d{t},
\end{cases}
\forall j \in [m-1], k \in [m],
\end{align}
with initial condition $(u_\cdot^{(n)}(0),w_\cdot^{(n)}(0))=x=(u_{x,\cdot},w_{x,\cdot})$, satisfies the following: there exists some $t_0>0$ such that $(u_\cdot^{(n)}(t),w_\cdot^{(n)}(t)) \in E_0^{(\floor{n/2})}$ for all $t \in [0,t_0]$, and
\begin{align}\label{ineq:sde_irreducible_step1}
\pnorm{(u_\cdot^{(n)}(t_0),w_\cdot^{(n)}(t_0))-x_0}{}\leq 4m\epsilon.
\end{align}
Our construction of $h^{(n)}$ will be based on two phases.

\noindent (\emph{Phase I}). For a tolerance level $\epsilon>0$, we claim that there exist $t_\ast>0$ and a sequence $0=t_0<t_1<\cdots<t_{m-1}=t_m=t_\ast$ which depend on $t_\ast,u_{x_0}$, such that with the piecewise constant function $q^\ast(t)\equiv \sum_{k \in [m]} e_k \bm{1}_{t \in [t_{k-1},t_k)} \in \R^m$ on $[0,t_\ast)$ and $q^\ast(t_\ast) \equiv e_m \in \R^m$, the solution $u^\ast(\cdot)\in \R^m$ to the ODE
\begin{align}\label{ineq:sde_irreducible_step1_ODE}
\d{u^\ast(t)} = q^\ast(t)-u^\ast(t),\quad u^\ast(0)=\binom{u_x}{1-\sum_{j \in [m-1]} u_{x,j}} \in \R^m
\end{align} 
satisfies
\begin{align}\label{ineq:sde_irreducible_step1_err}
\pnorm{u_{[m-1]}^\ast(t_\ast)-u_{x_0}}{1}\leq \epsilon. 
\end{align}
The proof of this claim will be deferred towards the end. 

Recall $\bar{w}_{(k)}^{(n)}\in \R^m$ defined in (\ref{ineq:sde_irreducible_step1_wbar}). Let us now define $\tilde{w}^{\ast,(n)}\in \R^m$ on $[0,t_\ast]$ by
\begin{align*}
\tilde{w}^{\ast,(n)}(t)\equiv \sum_{k \in [m]}\bar{w}_{(k)}^{(n)} \bm{1}_{q^\ast(t)=e_k} = \sum_{k \in [m]}\bar{w}_{(k)}^{(n)} \bm{1}_{t \in [t_{k-1},t_k)} \in \R^m,\quad \forall t \in [0,t_\ast),
\end{align*}
and $\tilde{w}^{\ast,(n)}(t_\ast)\equiv \bar{w}_{(m)}^{(n)}$. In other words, $\tilde{w}_k^{\ast,(n)}(t)= L_{m,n,\epsilon}(2\cdot \bm{1}_{t \in [t_{k-1},t_k)}-1)$ for all $k \in [m]$ and $t \in [0,t_\ast)$, with the end point $\tilde{w}_k^{\ast,(n)}(t_\ast)$ extended by continuity.

For $\eta>0$, let $\chi_\eta: \R \to [0,1]$ be a $C^\infty$, smoothed function of $\bm{1}_{[-1,1]}$ such that $\chi|_{[-1,1]}=\mathrm{id}$ and $\mathrm{supp}(\chi)=[-1-2\eta,1+2\eta]$.  Now with $\Delta t_k\equiv t_k-t_{k-1}$ and $\eta>0$ being a small enough constant to be specified later, let
\begin{align*}
w_{\eta;k}^{\ast,(n)}(t)&\equiv L_{m,n,\epsilon}\cdot\bigg[2\cdot \chi_\eta\bigg(\frac{2t-(t_{k-1}+t_k)}{\Delta t_k}\bigg)-1\bigg],\quad \forall k \in [m],\, t \in [0,t_\ast].
\end{align*}
Then we may estimate: for all $k \in [m]$ and $t \in [0,t_\ast]$,
\begin{align}\label{ineq:sde_irreducible_step1_w_err}
\abs{\tilde{w}_k^{\ast,(n)}(t)-w_{\eta;k}^{\ast,(n)}(t)}\leq  2L_{m,n,\epsilon}\cdot \big(\bm{1}_{t\in [t_{k-1}-\eta \Delta t_k,t_{k-1}]}+\bm{1}_{t\in [t_k,t_k+\eta \Delta t_k]}\big).
\end{align}
Let $u_\eta^{\ast,(n)}\in \R^{m-1}$ be the solution to the first equation of (\ref{ineq:sde_irreducible_ode}) with $w^{(n)}\equiv w_\eta^{\ast,(n)}$ and the initial condition specified therein, whose existence and uniqueness is guaranteed by the classical Picard-Lindel\"orf theorem for ODE. Then with $\Delta^{(n)} u_\eta^{\ast}\equiv u_\eta^{\ast,(n)}-u^\ast_{[m-1]} \in \R^{m-1}$, by comparing (\ref{ineq:sde_irreducible_ode}) and (\ref{ineq:sde_irreducible_step1_ODE}), 
\begin{align*}
\d \big(\Delta^{(n)} u_\eta^{\ast}\big)_j & = \big\{\big(p_j^{(n)}(u^{\ast,(n)}_{\eta;\cdot},w^{\ast,(n)}_{\eta;\cdot})-q^\ast_j \big)- \big(\Delta^{(n)}  u_\eta^\ast\big)_j \big\}\,\d{t},\quad \forall t \in [0,t_\ast],\, j \in [m-1].
\end{align*}
Solving this ODE, with $G_{\eta;j}^{\ast,(n)}\equiv p_j^{(n)}(u^{\ast,(n)}_{\eta;\cdot},w^{\ast,(n)}_{\eta;\cdot})-q^\ast_j $, we have
\begin{align*}
\big(\Delta^{(n)} u_\eta^{\ast}\big)_j(t)= e^{-t} \big(\Delta^{(n)} u_\eta^{\ast}\big)_j(0)+\int_0^t e^{-(t-s)}G_{\eta;j}^{\ast,(n)}(s)\,\d{s},\quad \forall t \in [0,t_\ast],\, j \in [m-1].
\end{align*}
Using $\Delta^{(n)} u_\eta^{\ast}=0$, we then have for any $t \in [0,t_\ast]$,
\begin{align*}
&\pnorm{\big(\Delta^{(n)} u_\eta^{\ast}\big)(t)}{1}\leq \int_0^t e^{-(t-s)}\pnorm{G_\eta^{\ast,(n)}(s)}{1}\,\d{s}\nonumber\\
&\leq \sup_{s \in [0,t]} \pnorm{p_\cdot^{(n)}(u^{\ast,(n)}_{\eta;\cdot},\tilde{w}^{\ast,(n)}_\cdot)-q^\ast}{1}+c_{m,n}\cdot\max_{k \in [m]}\int_0^t \abs{\tilde{w}_k^{\ast,(n)}(t)-w_{\eta;k}^{\ast,(n)}(t)}\,\d{t}\nonumber\\
&\equiv I_{\eta;1}(t)+I_{\eta;2}(t).
\end{align*}
We handle $I_{\eta;1}(t),I_{\eta;2}(t)$ as follows:
\begin{itemize}
	\item Using the same uniform lower estimate as in (\ref{ineq:strong_feller_sde_step2_lb}), there exists some $n_\ast=n_\ast(\epsilon,x)$ such that for all $n\geq n_\ast$, $u_\eta^{\ast,(n)}(t)\in \blacktriangle_{m-1}^{\circ,(n)}$ for all $\eta>0$ and $t \in [0,t_\ast]$, and therefore (\ref{ineq:sde_irreducible_step1_wbar}) entails that $\sup_{\eta>0}\sup_{t \in [0,t_\ast]}I_{\eta;1}(t)\leq (m-1)\epsilon$ whenever $n\geq n_\ast$.
	\item Using (\ref{ineq:sde_irreducible_step1_w_err}), we have $\sup_{t \in [0,t_\ast]} I_{\eta;2}(t)\leq c_{m,n,\epsilon}'\cdot \eta$. 
\end{itemize}
By choosing $\eta=\eta_\ast\equiv \epsilon/c_{m,n,\epsilon}'$, for $n\geq n_\ast$, we have $\sup_{t \in [0,t_\ast]}\pnorm{\big(\Delta^{(n)} u_{\eta_\ast}^{\ast}\big)(t)}{1}\leq m\epsilon$. Consequently, combined with (\ref{ineq:sde_irreducible_step1_err}), for $n\geq n_\ast$,
\begin{align}\label{ineq:sde_irreducible_step1_phase1_u}
\pnorm{u_{\eta_\ast}^{\ast,(n)}(t_\ast)-u_{x_0}}{1}\leq (m+1)\epsilon. 
\end{align}
Note that at the end of this phase I, both $u_{\eta_\ast}^{\ast,(n)}$ and $w_{\eta_\ast}^{\ast,(n)}$ are well-defined on $[0,t_\ast]$ with the end point of $u_{\eta_\ast}^{\ast,(n)}$ at time $t_\ast$ close to the target $u_{x_0}$. 

\noindent (\emph{Phase II}). In the second phase, we consider $t \in [t_\ast,t_\ast+\epsilon]$, and define $w_{\eta_\ast}^{\ast,(n)}$ on $[t_\ast,t_\ast+\epsilon]$ as the linear interpolation from $w_{\eta_\ast}^{\ast,(n)}(t_\ast)$ at time $t_\ast$ to the target $w_{x_0}$ at time $t_\ast+\epsilon$. Concretely,
\begin{align*}
w_{\eta_\ast}^{\ast,(n)}(t)\equiv w_{\eta_\ast}^{\ast,(n)}(t_\ast)+\epsilon^{-1}(t-t_\ast)\cdot \big(w_{x_0}-w_{\eta_\ast}^{\ast,(n)}(t_\ast)\big),\quad \forall t \in [t_\ast,t_\ast+\epsilon].
\end{align*}
Then $w_{\eta_\ast}^{\ast,(n)}$ is globally Lipschitz on $[0,t_\ast+\epsilon]$ with an a.e. well-defined derivative component-wise. With thus defined $w_{\eta_\ast}^{\ast,(n)}$, we then solve $u_{\eta_\ast}^{\ast,(n)}$ on $[t_\ast,t_\ast+\epsilon]$ via the ODE in the first line of (\ref{ineq:sde_irreducible_ode}) with initial condition at time $t_\ast$ given by $u_{\eta_\ast}^{\ast,(n)}(t_\ast)$ obtained in Phase I. As $\max_{j \in [m-1]}\abs{ (u^{\ast,(n)}_{\eta_\ast;j})'}\leq 1$, we have $
\pnorm{u^{\ast,(n)}_{\eta_\ast}(t_\ast+\epsilon)-u^{\ast,(n)}_{\eta_\ast}(t_\ast)}{1}\leq (m-1)\epsilon$. Combined with (\ref{ineq:sde_irreducible_step1_phase1_u}), 
\begin{align}\label{ineq:sde_irreducible_step1_phase2_u}
\max_{\# \in \{u,w\}}\pnorm{\#_{\eta_\ast}^{\ast,(n)}(t_\ast+\epsilon)-\#_{x_0}}{1}\leq 2m\epsilon.
\end{align}
By the second equation of (\ref{ineq:sde_irreducible_ode}), for $k \in [m]$, let
\begin{align*}
h_{\eta_\ast;k}^{\ast,(n)}\equiv \frac{ 1 }{ p_k^{(n),1/2}(u^{(n)}_{\eta_\ast;\cdot},w^{(n)}_{\eta_\ast;\cdot} )  }\cdot \bigg((w_{\eta_\ast;k}^{\ast,(n)})' + \frac{1}{2}w_{\eta_\ast;k}^{\ast,(n)}+\frac{1}{4} \partial_{w_k} p_k^{(n)} (u_{\eta_\ast;\cdot}^{\ast,(n)},w_{\eta_\ast;\cdot}^{\ast,(n)}) \bigg)
\end{align*}
be defined on $[0,t_\ast+\epsilon]$. Clearly $h_{\eta_\ast}^{\ast,(n)} \in L^2$ due to localization, and therefore proving the claim (\ref{ineq:sde_irreducible_step1}).

\noindent (\textbf{Step 2}). Let us choose $\epsilon>0$ small enough such that $B(x_0,8m\epsilon)\subset O$, $n$ large enough depending on $\epsilon,x$ and $t_0$ be as specified in Step 1. Recall $X^{(n)}$ is the solution to the localized SDE (\ref{ineq:strong_feller_sde}) with initial condition $X^{(0)}=x$. Let $\mu^{(n)}\equiv \Prob_x\circ (X^{(n)})^{-1}$ be the associated measure on $(C^{\alpha},\pnorm{\cdot}{\alpha})$ for $\alpha \in [0,1/2)$. By Theorem \ref{thm:support}, we have $\gamma^{(n)}\equiv \mathscr{S}_x(h^{(n)}) \in \mathrm{supp}_\alpha(\mu^{(n)})$. Using the alternative open set characterization for the support, for any $r>0$,
\begin{align*}
\Prob_x\big(X^{(n)} \in B_\alpha(\gamma^{(n)},r) \big)=\mu^{(n)}\big(B_\alpha(\gamma^{(n)},r)\big)=\mu^{(n)}\big(\big\{\gamma \in C^\alpha: \pnorm{\gamma-\gamma^{(n)}}{}< r\big\}\big)>0.
\end{align*}
Now we may choose $r_0 = r_0(\epsilon,x) \in (0,4m\epsilon)$ small enough such that for any $\gamma \in B_\alpha(\gamma^{(n)},r_0)$, $\gamma(t) \in E_0^{(n)}$ for all $t \in [0,t_0]$, and by definition and the proven claim in Step 1, $\gamma(t_0) \in B(x_0,8m\epsilon)\subset O$. Consequently,
\begin{align*}
\Prob_x(X(t_0)\in O)\geq \Prob_x\big(X^{(n)} \in B_a(\gamma^{(n)},r_0) \big)>0,
\end{align*}
proving the desired claim. \qed

\begin{proof}[Proof of (\ref{ineq:sde_irreducible_step1_err})]
To see this, note that the ODE (\ref{ineq:sde_irreducible_step1_ODE}) has an explicit solution
\begin{align*}
u^\ast(t)= e^{-t} u^\ast(0)+\int_0^t e^{-(t-s)}q^\ast(s)\,\d{s} = e^{-t}u_x+(1-e^{-t}) u(t,q^\ast),
\end{align*}
where with $\mu_t (\d s) \equiv \frac{e^{-(t-s)}}{1-e^{-t}}\,\d{s}$ being a probability measure on $[0,t]$, we write $u(t,q^\ast)\equiv \int_0^t q^\ast (s)\,\mu_t(\d s)$.

Consequently, it suffices to show that, for any $t>0$, there exists some piecewise constant function $q_t$ such that $u(t,q_t)=u_{x_0}$. To this end, for $k \in [m]$, let
\begin{align*}
y_k\equiv e^{-t}+(1-e^{-t})\sum_{j \in [k]} (u_{x_0})_j \in [0,1],\quad t_k\equiv t-\log(1/y_k), 
\end{align*}
and let $t_0\equiv 0$ for notational consistency. Note that $t_m=t$ due to $\sum_k (u_{x_0})_k=1$. Let $q_{t_{[m]}}(s)\equiv \sum_{k \in [m]} e_k \bm{1}_{s \in [t_{k-1},t_k)}$. We may then compute 
\begin{align*}
u(t,q_{t_{[m]}})&=\sum_{k \in [m]} e_k \int_{t_{k-1}}^{t_k} \mu_t(\d s) = \sum_{k \in [m]} \frac{e^{-t}(e^{t_k}-e^{t_{k-1}})}{1-e^{-t}}\cdot e_k\\
& = \sum_{k \in [m]} \frac{y_k-y_{k-1}}{1-e^{-t}}\cdot e_k = \sum_{k \in [m]} (u_{x_0})_k\cdot e_k = u_{x_0},
\end{align*}
proving the reduced claim, and therefore (\ref{ineq:sde_irreducible_step1_err}). 
\end{proof}

\section{Remaining proofs}\label{section:proof_adaptive_inf}

\subsection{Proof of Corollary \ref{cor:suboptimal_arm_rates}}\label{subsection:proof_suboptimal_arm_rates}

	For (\ref{ineq:suboptimal_arm_pull_GTS}), with $x=\bar{\Phi}^-(1/T)$, and using the asymptotic formula $1/T=\bar{\Phi}(x)= (1+\smallo(1))e^{-x^2/2}/(\sqrt{2\pi} x)$, we have $\log T = x^2/2+\log x+\bigo(1)$. This gives $x^2\sim 2\log T$.
	
	For (\ref{ineq:suboptimal_arm_pull_slow}), let $m$ be such that $\limsup_T \log_{[m]}^+ T/r_T<\infty$, where $\log_{[m]}^+ T\equiv \log^+(\log_{[m-1]}^+ T)$ is defined recursively with $\log_{[1]}^+T=1\vee \log T$. Then we may choose $\mathsf{Z}$ with $\bar{\Phi}(x)\sim c_m \exp(-\exp_{[m-1]}(x))$ as $x \uparrow \infty$. Here $\exp_{[m]}(x)=\exp(\exp_{[m-1]}(x))$ is similarly recursively defined with $\exp_{[1]}(x)=\exp(x)$. Using the same calculations as above, we have $\bar{\Phi}^-(1/T)\sim \log_{[m]}^+ T$ as $T \to \infty$. \qed

\subsection{Proof of Theorem \ref{thm:arm_mean_dist}}\label{subsection:proof_arm_mean_dist}

	The claim for optimal arms $a \in \mathcal{A}_0$ follows by Theorem \ref{thm:optimal_arm_pull} upon noting the relation
	\begin{align*}
	\big(n_{a;T}/\sigma^2\big)^{1/2}\cdot \big(\hat{\mu}_{a;T}-\mu_a\big) = \bigg\{\frac{1}{T^{1/2}}\sum_{s \in [T]}\bm{1}_{A_s=a}\xi_s\bigg\}\bigg/ \bigg\{\frac{n_{a;T}}{T}\bigg\}^{1/2},
	\end{align*}
	and applying the continuous mapping theorem. Therefore we will focus on Gaussian approximations for suboptimal arms $a \in \mathcal{A}_+$.
	
	To this end, we recall the filtration $\big\{\mathscr{F}_t\equiv \sigma\big(\{\xi_{s_1},Z_{a;s_2}\}_{a\in [K], s_1 \in [t],s_2 \in [t+1]}\big)\big\}$, and the theoretical sample size $n_{a;T}^\ast\equiv \sigma^2 [\bar{\Phi}^-(1/T)/\Delta_a]^2$. Let us define $X_{T,s}\equiv (n_{a;T}^\ast)^{-1/2} \bm{1}_{A_s=a} \xi_s$ and $\mathscr{F}_{T,s}\equiv \mathscr{F}_s$ for $s \in [T]$. Then $\{(X_{T,s},\mathscr{F}_{T,s}): s \in [T], T\in \mathbb{N}\}$ is a martingale difference array with the desired nested property. 
	
	For condition (1) in Lemma \ref{lem:mtg_clt}, we have for any $\epsilon>0$, 
	\begin{align*}
	&\sum_{s \in [T]} \E[X_{T,s}^2\bm{1}_{\abs{X_{T,s}}>\epsilon }|\mathscr{F}_{T,s-1}] \\
	&= \frac{1}{n_{a;T}^\ast}\sum_{s \in [T]} \E\big[\bm{1}_{A_s=a} \xi_s^2\bm{1}_{\abs{\xi_s}>\epsilon (n_{a;T}^\ast)^{1/2}} |\mathscr{F}_{s-1} \big]= \frac{n_{a;T}}{n_{a;T}^\ast} \cdot \E \xi_1^2\bm{1}_{\abs{\xi_1}>\epsilon (n_{a;T}^\ast)^{1/2}} .
	\end{align*}
	As $n_{a;T}^\ast \to \infty$ and $n_{a;T}/n_{a;T}^\ast \to 1$ in probability as $T \to \infty$ by Theorem \ref{thm:suboptimal_arm_pull}, the right hand side of the above display vanishes as $T \to \infty$ by Lemma \ref{lem:first_moment_tail}. This verifies condition (1) in Lemma \ref{lem:mtg_clt}. 
	
	For condition (2) in Lemma \ref{lem:mtg_clt}, we may easily compute 
	\begin{align*}
	\sum_{s \in [T]} \E[X_{T,s}^2|\mathscr{F}_{T,s-1}] = \frac{n_{a;T}}{n_{a;T}^\ast} \stackrel{\Prob}{\to} 1.
	\end{align*}
	Consequently, Lemma \ref{lem:mtg_clt} applies to conclude that 
	\begin{align*}
	\sum_{s \in [T]} X_{T,s} = (n_{a;T}^\ast)^{-1/2} \sum_{s \in [T]}\bm{1}_{A_s=a} \xi_s\rightsquigarrow \mathcal{N}(0,1).
	\end{align*}
	The claim now follows by noting the relation $
	\big(n_{a;T}/\sigma^2\big)^{1/2} \big(\hat{\mu}_{a;T}-\mu_a\big) =({n_{a;T}^\ast}/{n_{a;T}}) \cdot \sum_{s \in [T]} X_{T,s}$ 
	and another application of Theorem \ref{thm:suboptimal_arm_pull}. \qed

\subsection{Proof of Proposition \ref{prop:var_consist}}\label{subsection:proof_var_consist}

	It suffices to prove that, for any $a \in \mathcal{A}_+$,
	\begin{align}\label{ineq:var_consist_1}
	\hat{\sigma}_a^2\equiv \frac{1}{n_{a;T}}\sum_{t \in [T]} (R_t-\hat{\mu}_{a;t})^2\bm{1}_{A_t = a} \stackrel{\mathbb{P}}{\to} \sigma^2. 
	\end{align}
	Note that with $X_t\equiv \bm{1}_{A_t=a}(\xi_t^2-\sigma^2)\equiv \bm{1}_{A_t=a} \eta_t$, and the filtration $\{\mathscr{F}_t\}$ defined in the proof of Theorem \ref{thm:arm_mean_dist}, $\{(X_t,\mathscr{F}_t)\}$ is a martingale difference sequence. We define $b_T\equiv n_{a;T}^\ast$. Clearly $b_T \uparrow \infty$. 
	
	We shall now verify the conditions in Lemma \ref{lem:mtg_wlln}. For condition (1), note that 
	\begin{align*}
	\sum_{t \in [T]} \Prob\big(\abs{X_t}> b_T\big) &= \sum_{t \in [T]} \E \Big(\bm{1}_{A_t = a}\cdot \E\big[\bm{1}_{ \abs{
			\eta_t}>b_T }|\mathscr{F}_{t-1}\big]\Big) =  \frac{\E n_{a;T}}{n_{a;T}^\ast} \cdot \big[b_T \Prob(\abs{\eta_1}>b_T)\big].
	\end{align*}
	Now using the assumption and Lemma \ref{lem:first_moment_tail}, the right hand side of the above display vanishes as $T \to \infty$. This verifies (1).
	
	For (2), note that 
	\begin{align*}
	\frac{1}{b_T} \sum_{t \in [T]} \E[X_t\bm{1}_{\abs{X_t}\leq b_T}|\mathscr{F}_{t-1}]& = \frac{1}{b_T} \sum_{t \in [T]} \bm{1}_{A_t=a}\cdot  \E\eta_t\bm{1}_{\abs{\eta_t}\leq b_T}= \frac{n_{a;T}}{n_{a;T}^\ast}\cdot \E \eta_1\bm{1}_{\abs{\eta_1}\leq b_T}.
	\end{align*}
	By Theorem \ref{thm:suboptimal_arm_pull} and dominated convergence theorem, the right hand side of the above display vanishes as $T \to \infty$. This verifies (2).
	
	For (3), note that 
	\begin{itemize}
		\item $\E X_t^2\bm{1}_{\abs{X_t}\leq b_T}=\E \bm{1}_{A_t=a}\big(\eta_t^2 \bm{1}_{\abs{\eta_t}\leq b_T}\big)= \E \bm{1}_{A_t=a}\cdot \E \eta_1^2 \bm{1}_{\abs{\eta_1}\leq b_T} $.
		\item $\E\big(\E[X_t\bm{1}_{\abs{X_t}\leq b_T}|\mathscr{F}_{t-1}]\big)^2=\E\big(\bm{1}_{A_t=a}\cdot  \E{\eta_1}\bm{1}_{\abs{\eta_1}\leq b_T}\big)^2 = \E \bm{1}_{A_t=a}\cdot \big(\E {\eta_1}\bm{1}_{\abs{\eta_1}\leq b_T}\big)^2$.
	\end{itemize}
	So we have 
	\begin{align*}
	&\frac{1}{b_T^2} \sum_{t \in [T]} \Big\{\E X_t^2\bm{1}_{\abs{X_t}\leq b_T}-\E\big(\E[X_t\bm{1}_{\abs{X_t}\leq b_T}|\mathscr{F}_{t-1}]\big)^2\Big\}\\
	&= \frac{1}{b_T^2} \sum_{t \in [T]} \E \bm{1}_{A_t=a}\cdot \var\big({\eta_1} \bm{1}_{\abs{\eta_1}\leq b_T}\big)\leq \frac{\E n_{a;T}}{n_{a;T}^\ast}\cdot \frac{\E \eta_1^2 \bm{1}_{\abs{\eta_1}\leq b_T}}{b_T}. 
	\end{align*}
	On the other hand, with $g_T\equiv b_T^{-1}\abs{\eta_1}\bm{1}_{\abs{\eta_1}\leq b_T} \in [0,1]$ and $g_T\to 0$ almost surely, by dominated convergence we have $b_T^{-1} \E \eta_1^2 \bm{1}_{\abs{\eta_1}\leq b_T} = \E\abs{\eta_1} g_T \to 0$. Using the assumption, the right hand side of the above display vanishes as $T \to \infty$. This verifies (3). Thus Lemma \ref{lem:mtg_wlln} applies to conclude that
	\begin{align*}
	\hat{\sigma}_a^2 -\sigma^2= \frac{1}{n_{a;T}}\sum_{t \in [T]} X_t = \frac{n_{a;T}^\ast}{n_{a;T}}\cdot \frac{1}{b_T} \sum_{t \in [T]}X_t \stackrel{\Prob}{\to} 0,
	\end{align*}
	where in the last step we used Theorem \ref{thm:suboptimal_arm_pull} again. \qed

\appendix

\section{Basics of Lie brackets}\label{section:Lie_bracket}

\begin{definition}\label{def:vector_field_Lie}
	Let $M \subset \R^n$ be a smooth (i.e., $C^\infty$) manifold.
	\begin{enumerate}
		\item $U: C^\infty(M)\to C^\infty(M)$ is a smooth \emph{vector field}, if $U$ is a linear map and is a derivation in that $U(fg) = f U(g)+U(f) g$ for all $f,g \in C^\infty(M)$.
		\item For two smooth vector fields $U,V: C^\infty(M)\to C^\infty(M)$, their \emph{Lie bracket} $[U,V]: C^\infty(M)\to C^\infty(M)$ is defined as $[U,V](f)\equiv U(V(f))-V(U(f))$ for all $f \in C^\infty(M)$. 
	\end{enumerate}
\end{definition}

Let  $T_x M$ be the tangent space at $x \in M$ that can be identified as all derivations at $x$ in the sense that
\begin{align*}
T_x M&\equiv \big\{D: C^\infty(M)\to \R, \hbox {s.t. $D$ is a linear map},\\
&\quad\quad\hbox{and } D(fg)=(Df) g(x)+ f(x) D(g),\forall f,g \in C^\infty (M)\big\}.
\end{align*} 
Let $TM$ be the tangent bundle of $M$ that collects direct sum of $\{T_x M: x \in M\}$.

A vector field $U: C^\infty(M)\to C^\infty(M)$ can then be naturally identified as an element in $C^\infty(M,TM)$. Specifically, suppose $x=(x_1,\ldots,x_k): O \to \R^k$ is a local coordinate chart on $O\subset M$, where $k\leq n$. Then a smooth vector field $U$ may represented on $O$ as 
\begin{align*}
U(x) = \sum_{i \in [k]} U_i(x) \partial_{x_i},\quad \forall x \in O
\end{align*}
for some smooth functions $\{U_i \in C^\infty(O)\}$ and coordinate vector fields $\{\partial_{x_i}\}$.

The Lie bracket of two smooth vector fields $U,V:C^\infty(M)\to C^\infty(M)$ with local representation $U=\sum_i U_i \partial_{x_i},V=\sum_i V_i \partial_{x_i}$ on $O$ can then be represented locally as 
\begin{align}\label{eqn:Lie_bracket_coord_form}
[U,V](x)=\sum_{\ell \in [k]}  \bigg(\sum_{i \in [k]} U_i(x) \partial_{x_i} V_\ell(x) - V_i(x) \partial_{x_i} U_\ell(x) \bigg) \,\partial_{x_\ell},\quad \forall x \in O.
\end{align}
In the simplest case $M=\R^n$, a smooth vector field $U: C^\infty(\R^n)\to C^\infty(\R^n)$ can  be identified as a mapping from $\R^n$ to itself (denoted $U^*: \R^n\to \R^n$), with the following relation:
\begin{align}\label{eqn:Lie_bracket_iden}
U(f) (x)\equiv \sum_{i \in [n]} U_i^*(x) \partial_i f(x)=\iprod{U^\ast(x)}{\nabla f(x)},\quad \forall x \in \R^n, f \in C^\infty(\R^n).
\end{align}

\section{Technical tools}

\subsection{Some martingale results}
The following version of the martingale weak law of large numbers is taken from \cite[Theorem 2.13, pp. 29]{hall1980martingale}.
\begin{lemma}\label{lem:mtg_wlln}
	Suppose $\{(X_{n},\mathscr{F}_{n}): n \in \mathbb{N}\}$ is a martingale difference sequence, and $\{b_n\}$ is a sequence of positive numbers such that $b_n\uparrow \infty$. With $\big\{X_{n,i}\equiv X_i \bm{1}_{\abs{X_i}\leq b_n}: i \in [n]\big\}$, and further assuming the following:
	\begin{enumerate}
		\item $\sum_{i \in [n]} \Prob(\abs{X_i}>b_n)\to 0$.
		\item $b_n^{-1} \sum_{i \in [n]} \E[X_{n,i}|\mathscr{F}_{i-1}]\stackrel{\Prob}{\to} 0$.
		\item $b_n^{-2} \sum_{i \in [n]} \big\{\E X_{n,i}^2-\E\big(\E[X_{n,i}|\mathscr{F}_{i-1}]\big)^2\big\}\to 0$.
	\end{enumerate}
	Then $b_n^{-1} \sum_{i \in [n]} X_i \stackrel{\Prob}{\to} 0$. 
\end{lemma}

The following version of the martingale central limit theorem is taken from \cite[Corollary 3.1, pp. 58]{hall1980martingale}.
\begin{lemma}\label{lem:mtg_clt}
	Suppose $\{(X_{n,i},\mathscr{F}_{n,i}): i \in [k_n], n \in \mathbb{N}\}$ is a martingale difference array with the nested property $\mathscr{F}_{n,i}\subset \mathscr{F}_{n+1,i}$ for all $i \in [k_n]$ and $n \in \mathbb{N}$. Further assume the following:
	\begin{enumerate}
		\item For any $\epsilon>0$, $\sum_{i \in [k_n]} \E[X_{n,i}^2\bm{1}_{\abs{X_{n,i}}>\epsilon }|\mathscr{F}_{n,i-1}]\stackrel{\mathbb{P}}{\to} 0$.
		\item $\sum_{i \in [k_n]} \E[X_{n,i}^2|\mathscr{F}_{n,i-1}]\stackrel{\mathbb{P}}{\to} 1$.
	\end{enumerate}
	Then $\sum_{i \in [k_n]} X_{n,i}\rightsquigarrow \mathcal{N}(0,1)$. 
\end{lemma}

The following martingale Bernstein's inequality due to \cite{freedman1975tail} will be useful.

\begin{lemma}\label{ineq:mtg_bern}
	Let $X_1,X_2,\ldots$ be a martingale difference sequence with respect to the filtration $\{\mathscr{F}_n:n\geq 0\}$ such that $\sup_n \abs{X_i}\leq b$ almost surely. For $n\geq 1$, let $S_n\equiv \sum_{i=1}^n X_i$ be the associated partial sum, and $\langle S\rangle_n\equiv \sum_{i=1}^n \E(X_i^2|\mathscr{F}_{i-1})$. Then for any finite stopping time $\tau$ with respect to $\{\mathscr{F}_n\}$, and any $x\geq 0$,
	\begin{align*}
	\Prob\bigg(\max_{n\leq \tau} \abs{S_n}>x, \langle S\rangle_\tau\leq v^2\bigg)\leq 2\exp\bigg(-\frac{x^2}{2v^2+2bx/3}\bigg).
	\end{align*}
\end{lemma}

\subsection{A generalized Gr\"onwall's inequality}

The following is a generalized version of the classical Gr\"onwall's inequality. 

\begin{lemma}[Generalized Gr\"onwall's inequality]\label{lem:gronwall_ineq}
	Let $I$ be an interval of $\R$ with left end point $a$. Suppose that the integral inequality for $u:I\to \R_{\geq 0}$
	\begin{align}\label{ineq:gronwall_ineq_cond}
	u(t)\leq \alpha(t)+\int_a^t \beta(s)\cdot \mathsf{F}(u(s))\,\d{s},\quad \forall t \in I. 
	\end{align}
	is satisfied for non-negative functions $\alpha,\beta: I\to \R_{\geq 0}$ and $\mathsf{F}\in C^1(\R_{\geq 0}\to \R_{>0})$ such that $\mathsf{F}$ is  non-decreasing on $\R_{\geq 0}$. Then, for any $t_0\geq 0$, with $\mathsf{G}(t)\equiv \int_{t_0}^t (\mathsf{F}(x))^{-1}\d{x}$, 
	\begin{align*}
	u(t)\leq \mathsf{G}^{-1}\bigg[\mathsf{G}\bigg(\sup_{s \in [a,t]}\alpha(s)\bigg)+\int_a^t \beta(s)\,\d{s}\bigg],\quad \forall t \in I.
	\end{align*}
	In particular, if $\mathsf{F}(x)=x^p$ for some $p \in (0,1]$, 
	\begin{align*}
	u(t)\leq 
	\begin{cases}
	\big[\big(\sup_{s \in [a,t]}\alpha(s)\big)^{1-p}+(1-p)\int_a^t \beta(s)\,\d{s}\big]^{1/(1-p)}, & p \in (0,1);\\
	\big(\sup_{s \in [a,t]}\alpha(s)\big)\cdot \exp\big(\int_a^t \beta(s)\,\d{s}\big), & p=1.
	\end{cases}	
	\end{align*}
\end{lemma}

\begin{proof}
	Fix $t \in I$. Let $A(t)\equiv \sup_{s \in [a,t]}\alpha(s)$, and 
	\begin{align*}
	w(s)\equiv A(t)+\int_a^s \beta(r) \cdot \mathsf{F}(u(r))\,\d{r},\quad \forall s \in [a,t].
	\end{align*}
	Then  (\ref{ineq:gronwall_ineq_cond}) entails that $u(s)\leq w(s)$ for all $s \in [a,t]$. Using that $\beta(\cdot)\geq 0$ and $\mathsf{F}$ is non-decreasing, we have $w'(s)=\beta(s)\mathsf{F}(u(s))\leq \beta(s)\mathsf{F}(w(s))$ for $s \in [a,t]$. As $\mathsf{G}'=1/\mathsf{F}'\geq 0$, it then follows that
	\begin{align*}
	\frac{\d}{\d s} \mathsf{G}(w(s))=\mathsf{G}'(w(s))w'(s)\leq \mathsf{G}'(w(s))\cdot \beta(s)\mathsf{F}(w(s)) = \beta(s),\quad \forall s \in [a,t].
	\end{align*}
	Integrating the above display for $s \in [a,t]$, using that $w(a)=A(t)$, we have 
	\begin{align*}
	\mathsf{G}(w(t))\leq \mathsf{G}(A(t))+\int_a^t \beta(s)\,\d{s}\equiv \mathsf{G}(A(t))+B(t). 
	\end{align*}
	As $\mathsf{G}:\R_{\geq 0}\to \R_{\geq 0}$ is continuously increasing, we then have
	\begin{align*}
	u(t)\leq w(t)\leq \mathsf{G}^{-1}\big(\mathsf{G}(A(t))+B(t)\big),
	\end{align*}
	proving the desired general inequality.
	
	If $\mathsf{F}(x)=x^p$ for $p \in (0,1)$, then $\mathsf{G}(t)=\int_{t_0}^t x^{-p}\,\d{x}=(1-p)^{-1}\big(t^{1-p}-t_0^{1-p}\big)$ and $\mathsf{G}^{-1}(y)=\big((1-p)y+t_0^{1-p}\big)^{1/(1-p)}$. Consequently, with $c_p\equiv (1-p)^{-1}t_0^{1-p}$,
	\begin{align*}
	u(t)\leq \mathsf{G}^{-1}\big((1-p)^{-1}(A(t))^{1-p}+B(t)-c_p\big)\leq \big((A(t))^{1-p}+(1-p)B(t)\big)^{1/(1-p)}.
	\end{align*}
	If $\mathsf{F}(x)=x$, then $\mathsf{G}(t)=\log(t/t_0)$ and $\mathsf{G}^{-1}(y)=t_0 e^y$. Consequently, we recover the standard Gr\"onwall's inequality with 
	\begin{align*}
	u(t)\leq \mathsf{G}^{-1}\big(\log(A(t))+B(t)-\log t_0\big)=\exp\big(\log(A(t))+B(t)\big)=A(t)e^{B(t)},
	\end{align*}
	as desired.
\end{proof}

\section{Auxiliary results}

The following lemma provides a uniform control for the error process.
\begin{lemma}\label{lem:max_err}
	Let $\xi_1,\ldots,\xi_n$ be independent random variables with mean $0$ and variance $1$. Then there exists some universal constant $c>0$ such that
	\begin{align*}
	\Prob\bigg(\max_{k \geq 1} \biggabs{\frac{1}{\sqrt{k}\log k}\sum_{i \in [k]}\xi_i}\geq t\bigg)\leq c\cdot t^{-2},\quad \forall t> 0.
	\end{align*}
	In particular, for $E_{\xi}(x)$ defined in (\ref{def:event_E_xi}), we have $
	\Prob\big(E_{\xi}^c(x)\big)\leq cK\cdot x^{-2}$.
\end{lemma}
\begin{proof}
	We prove the claim by a standard blocking argument. Note that
	\begin{align*}
	\Prob\bigg(\max_{k\geq 1} \biggabs{\frac{1}{\sqrt{k}\log k}\sum_{i \in [k]}\xi_i}\geq t\bigg)
	&\leq \sum_{j =1}^\infty  \Prob\bigg(\max_{e^{j-1}\leq k\leq e^{j}} \biggabs{\sum_{i \in [k]}\xi_i}\geq t e^{(j-1)/2}(j-1)\bigg).
	\end{align*}
	By L\'evy's maximal inequality (cf. \cite[Theorem 1.1.1]{de2012decoupling}), the right hand side of the above display can be further bounded by 
	\begin{align*}
	c\sum_{j=1}^\infty \Prob\bigg(\biggabs{\sum_{i \in [e^j]}\xi_i}\geq te^{(j-1)/2} (j-1)\bigg)\leq c \sum_{j=1}^\infty \frac{1}{t^2j^2}\leq c\cdot t^{-2},
	\end{align*} 
	as desired. 
\end{proof}

The following provides certain regularity condition for the c.d.f. of the convoluted random variables.

\begin{lemma}\label{lem:Phi_Psi_1}
	Let $\mathsf{Z}$ be a random variable with c.d.f. $\Phi$. With $\mathsf{Z}'$ denoting an independent copy of $\mathsf{Z}$, let
	\begin{align*}
	\Phi_\ast(z) \equiv \sup_{(a,a')\in\partial B_2(1)} \Prob(a\mathsf{Z}+a'\mathsf{Z}'\le z),\,	\Psi_\ast(z) \equiv \inf_{(a,a')\in\partial B_2(1)} \Prob(a\mathsf{Z}+a'\mathsf{Z}'\le z).
	\end{align*}
	Suppose $\Phi(z) \in (0,1)$ for all $z \in \R$. Then for every $B>0$, 
	\begin{align*}
	\sup_{z \in [-B,B]} \Phi_\ast(z)\vee \big(1-\Psi_\ast(z)\big)<1.
	\end{align*}
\end{lemma}
\begin{proof}
	\noindent (1). We first prove the claim for $\Phi_\ast$. Let $u_z\equiv 2\abs{z}+1$, $p_z\equiv \Prob(\abs{\mathsf{Z}}\leq u_z)>0$, and let $T_z\equiv \sqrt{2}(z+u_z)$. Fix $a,a' \in \partial B_2(1)$. Without loss of generality, we assume $\abs{a}\geq 1/\sqrt{2}$. Then:
	\begin{itemize}
		\item Suppose $a>0$. Then on the event $\{\mathsf{Z}>T_z, \abs{\mathsf{Z}'}\leq u_z\}$, we have $
		a \mathsf{Z}+a' \mathsf{Z}'> a T_z - \abs{a'} u_z\geq (1/\sqrt{2}) T_z-u_z= z$. 
		Thus, by independence of $\mathsf{Z},\mathsf{Z}'$, 
		\begin{align*}
		\Prob\big(a \mathsf{Z}+a' \mathsf{Z}'>z\big)\geq \Prob\big(\mathsf{Z}>T_z, \abs{\mathsf{Z}'}\leq u_z\big)\geq p_z\cdot \Prob(\mathsf{Z}>T_z).
		\end{align*}
		\item Suppose $a<0$. Then on the event $\{\mathsf{Z}<-T_z, \abs{\mathsf{Z}'}\leq u_z\}$ we may similar conclude $a \mathsf{Z}+a' \mathsf{Z}'>z$, and therefore $
		\Prob\big(a \mathsf{Z}+a' \mathsf{Z}'>z\big)\geq p_z\cdot  \Prob(\mathsf{Z}<-T_z)$. 
	\end{itemize}
	Combining the two cases, we have 
	\begin{align*}
	\inf_{(a,a')\in\partial B_2(1)}\Prob\big(a \mathsf{Z}+a' \mathsf{Z}'>z\big)\geq p_z\cdot \big(\Prob(\mathsf{Z}>T_z)\wedge \Prob(\mathsf{Z}<-T_z)\big)\equiv \delta(z).
	\end{align*}
	Equivalently, $\Phi_\ast(z)\leq 1-\delta(z)$. The first claim now follows as $\inf_{z \in [-B,B]} \delta(z)>0$ for all $B>0$.
	
	\noindent (2). Next, we prove the claim for $\Psi_\ast$. The argument is genuinely similar. Using the same $u_z,p_z$ as defined above and $S_z\equiv \sqrt{2}(u_z-z)$, and assume $\abs{a}\geq 1/\sqrt{2}$. Then:
	\begin{itemize}
		\item Suppose $a>0$. Then on the event $\{\mathsf{Z}\leq -S_z, \abs{\mathsf{Z}'}\leq u_z\}$, we have $
		a \mathsf{Z}+a' \mathsf{Z}'\leq -a S_z + \abs{a'} u_z\leq -(1/\sqrt{2}) S_z+u_z=z$. This leads to $\Prob\big(a \mathsf{Z}+a' \mathsf{Z}'\leq z\big)\geq p_z\cdot \Prob(\mathsf{Z}\leq -S_z)$. 
		\item Suppose $a<0$. Then on the event $\{\mathsf{Z}\geq S_z, \abs{\mathsf{Z}'}\leq u_z\}$, we again have $a \mathsf{Z}+a' \mathsf{Z}'\leq z$, which leads to $\Prob\big(a \mathsf{Z}+a' \mathsf{Z}'\leq z\big)\geq p_z\cdot \Prob(\mathsf{Z}\geq S_z)$. 
	\end{itemize} 
	Combining the two cases and arguing as in (1) to conclude the proof for $\Psi_\ast$.
\end{proof}

The following tail estimate is classical.
\begin{lemma}\label{lem:first_moment_tail}
Suppose $X$ has a finite first moment $\E\abs{X}<\infty$. Then we have $\lim_{t \to \infty} t\Prob(\abs{X}>t)=0$.
\end{lemma}
\begin{proof}
We include below a quick proof. Let $F(t)\equiv \Prob(\abs{X}>t)$. Then $F$ is a non-decreasing function on $[0,\infty)$ and $\int_0^\infty F(t)\,\d{t} = \E\abs{X}<\infty$. Suppose on the contrary that for some sequence $\{t_n\}$, $\liminf_n t_n F(t_n)\geq 2\epsilon>0$. Then by passing to a further subsequence if necessary, we assume $t_{n+1}> 2t_n$. This means that the intervals $\{[t_n/2,t_n]\}$ are disjoint. Since $F$ is non-increasing, for some large $n_0 \in \N$ and all $n\geq n_0$, for all $t \in [t_n/2,t_n]$, we have $F(t)\geq F(t_n)\geq \epsilon/t_n$. This means that $\int_0^\infty F(t)\,\d{t}\geq \sum_{n\geq n_0} \int_{t_n/2}^{t_n} F(t)\,\d{t}\geq \sum_{n\geq n_0} (\epsilon/2)=\infty$, a contradiction.
\end{proof}

The following lemma formally proves that $\mathcal{N}(0,1)$ is the unique invariant probability measure for the (rescaled) Ornstein–Uhlenbeck semigroup.

\begin{lemma}\label{lem:sde_1d_inv}
The semigroup associated with the stochastic differential equation 
\begin{align}\label{ineq:sde_1d}
\d{w(t)} = -\frac{1}{2}w(t)+\d{B(t)},\quad \forall t \in \R.
\end{align}
has $\mathcal{N}(0,1)$ as its unique invariant distribution.
\end{lemma}
\begin{proof}
First, note that the SDE (\ref{ineq:sde_1d}) has an explicit solution
\begin{align*}
w(t) = e^{-t/2} w(0) + \int_0^t e^{-(t-s)/2}\,\d{B(s)}\sim \mathcal{N}\big(e^{-t/2}w(0)+1-e^{-t}\big).
\end{align*}
This means that the semigroup $(P_t)$ associated with (\ref{ineq:sde_1d}) can be computed as
\begin{align*}
P_t f(x) = \E_x f(w(t)) = \E f\big(e^{-t/2}x+\sqrt{1-e^{-t} }\cdot \mathsf{Z}_0\big),\quad \mathsf{Z}_0\sim \mathcal{N}(0,1). 
\end{align*}
Clearly $\mathcal{N}(0,1)$ is an invariant distribution to $(P_t)$. Now for any invariant measure $\mu$ of $(P_t)$, by Definition \ref{def:semigroup} we have for all $t\geq 0$ and any $A \in \mathcal{B}(\R)$,
\begin{align*}
\mu(A)=\int \Prob_x\big(w(t) \in A \big)\,\mu(\d x) = \int \E\bm{1}_A\big(e^{-t/2}x+\sqrt{1-e^{-t} }\cdot \mathsf{Z}_0\big)\, \mu(\d x).
\end{align*}
So a standard measure theoretic argument shows that for any $f \in B(\R)$,
\begin{align*}
\mu(f)\equiv \int f(x)\,\mu(\d x) = \int \E f\big(e^{-t/2}x+\sqrt{1-e^{-t} }\cdot \mathsf{Z}_0\big)\, \mu(\d x).
\end{align*}
Sending $t \to \infty$, we have $\mu(f) = \E f(\mathsf{Z}_0)$ for all $f \in B(\R)$. This shows that any invariant measure $\mu$ must have law $\mathcal{N}(0,1)$. 
\end{proof}

\section{Simulation methods for the SDE}

\begin{algorithm}[!t]
	\caption{Euler-Maruyama simulation of the SDE (\ref{def:sde_time_change_renor})}
	\label{alg:EM-uw}
	\textbf{Input}
	\begin{itemize}
		\item index set of optimal arms $\mathcal A_0$;
		\item  step size $\Delta t>0$, terminal time $T>0$, burn-in time $T_{\mathrm{burn}}\in[T]$; 
		\item thinning factor $q\in\mathbb N$, Monte-Carlo size $M\in\mathbb N$ for evaluating $p_a$;
		\item clipping parameters $\varepsilon\in(0,1)$ and $\delta\in(0,1)$;
		\item initial condition $(u^0,w^0)$ with $u^0\in\Delta_{\mathcal{A}_0}^\circ$.
	\end{itemize}
	
	\medskip
	\textbf{Procedure}
	\begin{enumerate}
		\item Set $N \leftarrow \lfloor T/\Delta t \rfloor$, $k_{\mathrm{burn}} \leftarrow \lfloor T_{\mathrm{burn}}/\Delta t \rfloor$.
		Initialize an empty list $\texttt{Samples}$.
		
		\item \textbf{For} $k=0,1,\dots,N-1$ \textbf{do}
		\begin{enumerate}
			\item (\emph{Compute $p^k$}). For each $a\in\mathcal A_0$, compute
			\[
			\widehat p_a^k
			\;\leftarrow\;
			\frac1M\sum_{\ell=1}^M
			\prod_{b\in\mathcal A_0\setminus\{a\}}
			\Phi\,\bigg[\sqrt{u_b^k}\cdot\bigg(\frac{w_a^k}{u_a^k}-\frac{w_b^k}{u_b^k}
			+\frac{Z^{(\ell)}}{\sqrt{u_a^k}}\bigg)\bigg],
			\]
			where $Z^{(\ell)}\stackrel{\text{i.i.d.}}{\sim}\mathcal N(0,1)$.
			Then set $p_a^k \leftarrow  p_a^k\vee\delta$ for all $a\in\mathcal A_0$, and renormalize
			$p^k\leftarrow p^k/\sum_{c\in\mathcal A_0}p_c^k$.
			
			\item (\emph{Brownian increment}). Draw i.i.d.\ $\xi_a^{k+1}\sim\mathcal N(0,1)$ for all $a\in\mathcal A_0$,
			and set $\Delta B_a^{k+1}\leftarrow \sqrt{\Delta t}\,\xi_a^{k+1}$.
			
			\item (\emph{Euler-Maruyama update}). For each $a\in\mathcal A_0$, set
			\begin{align*}
			u_a^{k+1} &\leftarrow u_a^k + (p_a^k-u_a^k)\Delta t,\\
			w_a^{k+1} &\leftarrow w_a^k - \tfrac12 w_a^k\,\Delta t + \sqrt{p_a^k}\,\Delta B_a^{k+1}.
			\end{align*}
			
			\item (\emph{Project $u^{k+1}$ to the simplex}). For each $a\in\mathcal A_0$, set
			$\widetilde u_a^{k+1}\leftarrow (u_a^{k+1}\vee \epsilon)\wedge (1-\epsilon)$,
			and then renormalize
			\[
			u^{k+1}\leftarrow \widetilde u^{k+1}\Big/\sum_{c\in\mathcal A_0}\widetilde u_c^{k+1}.
			\]
			
			\item (\emph{Store after burn-in with thinning}).
			If $k\ge k_{\mathrm{burn}}$ and $(k-k_{\mathrm{burn}})\equiv 0 \pmod q$, append $(u^k,w^k)$ to $\texttt{Samples}$ and  store $\{w_a^k/\sqrt{u_a^k}\}_{a\in\mathcal A_0}$.
		\end{enumerate}
		\item \textbf{end for}
		
		\item \textbf{return} $\texttt{Samples}$.
	\end{enumerate}
     
     	\textbf{Output}
     	\begin{itemize}
     		\item stored samples $\{(u^k,w^k)\}$ approximating the invariant law;
     		\item $\{w_a^k/\sqrt{u_a^k}\}$ for each $a\in\mathcal A_0$ used in Figure \ref{fig:comp_N_normal}.
     	\end{itemize}
\end{algorithm}

Recall $\{p_a\}_{a \in \mathcal{A}_0}$ defined in (\ref{def:p_a}) and the SDE in (\ref{def:sde_time_change_renor}). Below we shall only describe the simulation method for Gaussian Thompson sampling with $\mathsf{Z}\sim \mathcal{N}(0,1)$; the general sampling scheme follows from obvious changes. Our simulation method, based on Euler-Maruyama discretization scheme, is summarized in Algorithm \ref{alg:EM-uw}.

In our simulations used in Figure \ref{fig:arm_pull} and \ref{fig:comp_N_normal}, we make the following choice of the simulation parameters in Algorithm \ref{alg:EM-uw}: 
\begin{itemize}
	\item step size $\Delta t = 0.02$, terminal time $T=10^5$, burn-in time $T_{\mathrm{burn}}=2\times 10^4$;
	\item thinning factor $q=5$, Monte-Carlo size $M=400$ (only when $\abs{\mathcal{A}_0}\geq 3$);
	\item clipping factor $\epsilon=10^{-6}$ and $\delta = 10^{-10}$;
	\item initial condition $u^0=\abs{\mathcal{A}_0}^{-1}\bm{1}_{\mathcal{A}_0}$ and $w^0=0_{\mathcal{A}_0}$.
\end{itemize}
For the samples from Gaussian Thompson sampling obtained in Figure \ref{fig:arm_pull}, we use a smaller number $T=2\times 10^4$ and the histogram is plotted with $2\times 10^4$ Monte-Carlo simulations.

\section*{Acknowledgments}
The research of Q. Han is partially supported by NSF grant DMS-2143468.

\bibliographystyle{alpha}
\bibliography{mybib}

@article {kusuoka1987applications,
	AUTHOR = {Kusuoka, Shigeo and Stroock, Daniel},
	TITLE = {Applications of the {M}alliavin calculus. {III}},
	JOURNAL = {J. Fac. Sci. Univ. Tokyo Sect. IA Math.},
	FJOURNAL = {Journal of the Faculty of Science. University of Tokyo.
	Section IA. Mathematics},
	VOLUME = {34},
	YEAR = {1987},
	NUMBER = {2},
	PAGES = {391--442},
	ISSN = {0040-8980},
	MRCLASS = {60J60 (58G32 60H10)},
	MRNUMBER = {914028},
	MRREVIEWER = {Michael\ Cranston},
}

@article {kusuoka1985applications,
	AUTHOR = {Kusuoka, Shigeo and Stroock, Daniel},
	TITLE = {Applications of the {M}alliavin calculus. {II}},
	JOURNAL = {J. Fac. Sci. Univ. Tokyo Sect. IA Math.},
	FJOURNAL = {Journal of the Faculty of Science. University of Tokyo.
	Section IA. Mathematics},
	VOLUME = {32},
	YEAR = {1985},
	NUMBER = {1},
	PAGES = {1--76},
	ISSN = {0040-8980},
	MRCLASS = {60H07 (58G32 60G30)},
	MRNUMBER = {783181},
	MRREVIEWER = {Michael\ Cranston},
}

@incollection {kusuoka1984applications,
	AUTHOR = {Kusuoka, Shigeo and Stroock, Daniel},
	TITLE = {Applications of the {M}alliavin calculus. {I}},
	BOOKTITLE = {Stochastic analysis ({K}atata/{K}yoto, 1982)},
	SERIES = {North-Holland Math. Library},
	VOLUME = {32},
	PAGES = {271--306},
	PUBLISHER = {North-Holland, Amsterdam},
	YEAR = {1984},
	ISBN = {0-444-87588-3},
	MRCLASS = {60H07 (58G32 60G30)},
	MRNUMBER = {780762},
	MRREVIEWER = {Michael\ Cranston},
	DOI = {10.1016/S0924-6509(08)70397-0},
	URL = {https://doi-org.proxy.libraries.rutgers.edu/10.1016/S0924-6509(08)70397-0},
}

@inproceedings {malliavin1978stochastic,
	AUTHOR = {Malliavin, Paul},
	TITLE = {Stochastic calculus of variation and hypoelliptic operators},
	BOOKTITLE = {Proceedings of the {I}nternational {S}ymposium on {S}tochastic
	{D}ifferential {E}quations ({R}es. {I}nst. {M}ath. {S}ci.,
	{K}yoto {U}niv., {K}yoto, 1976)},
	SERIES = {Wiley-Intersci. Publ.},
	PAGES = {195--263},
	PUBLISHER = {John Wiley \& Sons, New York-Chichester-Brisbane},
	YEAR = {1978},
	ISBN = {0-471-05375-9},
	MRCLASS = {60H05 (35H05 58G32 60J60)},
	MRNUMBER = {536013},
	MRREVIEWER = {Kiyosi\ It\^o},
}

@article {bubeck2023first,
	AUTHOR = {Bubeck, S\'ebastien and Sellke, Mark},
	TITLE = {First-order {B}ayesian regret analysis of {T}hompson sampling},
	JOURNAL = {IEEE Trans. Inform. Theory},
	FJOURNAL = {Institute of Electrical and Electronics Engineers.
	Transactions on Information Theory},
	VOLUME = {69},
	YEAR = {2023},
	NUMBER = {3},
	PAGES = {1795--1823},
	ISSN = {0018-9448,1557-9654},
	MRCLASS = {62B10 (62C10)},
	MRNUMBER = {4564682},
	DOI = {10.1109/tit.2022.3213630},
	URL = {https://doi-org.proxy.libraries.rutgers.edu/10.1109/tit.2022.3213630},
}

@article {russo2016information,
	AUTHOR = {Russo, Daniel and Van Roy, Benjamin},
	TITLE = {An information-theoretic analysis of {T}hompson sampling},
	JOURNAL = {J. Mach. Learn. Res.},
	FJOURNAL = {Journal of Machine Learning Research (JMLR)},
	VOLUME = {17},
	YEAR = {2016},
	PAGES = {Paper No. 68, 30},
	ISSN = {1532-4435,1533-7928},
	MRCLASS = {62C05 (62B10 62L05 93E35)},
	MRNUMBER = {3517091},
}

@article {russo2014learning,
	AUTHOR = {Russo, Daniel and Van Roy, Benjamin},
	TITLE = {Learning to optimize via posterior sampling},
	JOURNAL = {Math. Oper. Res.},
	FJOURNAL = {Mathematics of Operations Research},
	VOLUME = {39},
	YEAR = {2014},
	NUMBER = {4},
	PAGES = {1221--1243},
	ISSN = {0364-765X,1526-5471},
	MRCLASS = {93E35 (62C10 62L05)},
	MRNUMBER = {3279764},
	MRREVIEWER = {Oleg\ N.\ Granichin},
	DOI = {10.1287/moor.2014.0650},
	URL = {https://doi-org.proxy.libraries.rutgers.edu/10.1287/moor.2014.0650},
}

@article{korda2013thompson,
	title={Thompson sampling for 1-dimensional exponential family bandits},
	author={Korda, Nathaniel and Kaufmann, Emilie and Munos, Remi},
	journal={Advances in {N}eural {I}nformation {P}rocessing {S}ystems},
	volume={26},
	year={2013}
}

@inproceedings{agrawal2012analysis,
	title={Analysis of thompson sampling for the multi-armed bandit problem},
	author={Agrawal, Shipra and Goyal, Navin},
	booktitle={Conference on learning theory},
	pages={39--1},
	year={2012},
	organization={JMLR Workshop and Conference Proceedings}
}

@article{chapelle2011empirical,
	title={An empirical evaluation of thompson sampling},
	author={Chapelle, Olivier and Li, Lihong},
	journal={Advances in {N}eural {I}nformation {P}rocessing {S}ystems},
	volume={24},
	year={2011}
}

@article{fan2025diffusion,
	title={Diffusion Approximations for {T}hompson Sampling in the Small Gap Regime},
	author={Fan, Lin and Glynn, Peter W.},
	journal={arXiv preprint arXiv:2105.09232v5},
	year={2025}
}

@article{kuang2024weak,
	title={Weak signal asymptotics for sequentially randomized experiments},
	author={Kuang, Xu and Wager, Stefan},
	journal={Management Science},
	volume={70},
	number={10},
	pages={7024--7041},
	year={2024},
	publisher={INFORMS}
}

@article{praharaj2025instability,
	title={On Instability of Minimax Optimal Optimism-Based Bandit Algorithms},
	author={Praharaj, Samya and Khamaru, Koulik},
	journal={arXiv preprint arXiv:2511.18750},
	year={2025}
}

@article{halder2025stable,
	title={Stable Thompson Sampling: Valid Inference via Variance Inflation},
	author={Halder, Budhaditya and Pan, Shubhayan and Khamaru, Koulik},
	journal={arXiv preprint arXiv:2505.23260},
	year={2025}
}

@article{fan2025statistical,
	title={Statistical Inference under Adaptive Sampling with Lin{UCB}},
	author={Fan, Wei and Tan, Kevin and Wei, Yuting},
	journal={arXiv preprint arXiv:2512.00222},
	year={2025}
}

@article{han2024ucb,
	title={{UCB} algorithms for multi-armed bandits: Precise regret and adaptive inference},
	author={Han, Qiyang and Khamaru, Koulik and Zhang, Cun-Hui},
	journal={arXiv preprint arXiv:2412.06126},
	year={2024}
}

@article{hairer2008ergodic,
	title={Ergodic theory for stochastic {PDE}s},
	author={Hairer, Martin},
	journal={preprint},
	year={2008}
}

@inproceedings {stroock1972support,
	AUTHOR = {Stroock, Daniel W. and Varadhan, S. R. S.},
	TITLE = {On the support of diffusion processes with applications to the
	strong maximum principle},
	BOOKTITLE = {Proceedings of the {S}ixth {B}erkeley {S}ymposium on
	{M}athematical {S}tatistics and {P}robability ({U}niv.
	{C}alifornia, {B}erkeley, {C}alif., 1970/1971), {V}ol. {III}:
	{P}robability theory},
	PAGES = {333--359},
	PUBLISHER = {Univ. California Press, Berkeley, CA},
	YEAR = {1972},
	MRCLASS = {60J60},
	MRNUMBER = {400425},
}

@article {benarous1994holder,
	AUTHOR = {Ben Arous, G\'erard and Gradinaru, Mihai and Ledoux,
	Michel},
	TITLE = {H\"older norms and the support theorem for diffusions},
	JOURNAL = {Ann. Inst. H. Poincar\'e{} Probab. Statist.},
	FJOURNAL = {Annales de l'Institut Henri Poincar\'e. Probabilit\'es et
	Statistiques},
	VOLUME = {30},
	YEAR = {1994},
	NUMBER = {3},
	PAGES = {415--436},
	ISSN = {0246-0203},
	MRCLASS = {60J60 (60H10)},
	MRNUMBER = {1288358},
	MRREVIEWER = {Pierre\ Vallois},
	URL = {http://www.numdam.org.proxy.libraries.rutgers.edu/item?id=AIHPB_1994__30_3_415_0},
}

@incollection {millet1994simple,
	AUTHOR = {Millet, Annie and Sanz-Sol\'e, Marta},
	TITLE = {A simple proof of the support theorem for diffusion processes},
	BOOKTITLE = {S\'eminaire de {P}robabilit\'es, {XXVIII}},
	SERIES = {Lecture Notes in Math.},
	VOLUME = {1583},
	PAGES = {36--48},
	PUBLISHER = {Springer, Berlin},
	YEAR = {1994},
	ISBN = {3-540-58331-9},
	MRCLASS = {60J60},
	MRNUMBER = {1329099},
	MRREVIEWER = {S.\ Ramasubramanian},
	DOI = {10.1007/BFb0073832},
	URL = {https://doi-org.proxy.libraries.rutgers.edu/10.1007/BFb0073832},
}

@article {hormander1967hypoelliptic,
	AUTHOR = {H\"ormander, Lars},
	TITLE = {Hypoelliptic second order differential equations},
	JOURNAL = {Acta Math.},
	FJOURNAL = {Acta Mathematica},
	VOLUME = {119},
	YEAR = {1967},
	PAGES = {147--171},
	ISSN = {0001-5962,1871-2509},
	MRCLASS = {35.48 (47.00)},
	MRNUMBER = {222474},
	MRREVIEWER = {Joel\ Smoller},
	DOI = {10.1007/BF02392081},
	URL = {https://doi-org.proxy.libraries.rutgers.edu/10.1007/BF02392081},
}

@article {hairer2011malliavin,
	AUTHOR = {Hairer, Martin},
	TITLE = {On {M}alliavin's proof of {H}\"ormander's theorem},
	JOURNAL = {Bull. Sci. Math.},
	FJOURNAL = {Bulletin des Sciences Math\'ematiques},
	VOLUME = {135},
	YEAR = {2011},
	NUMBER = {6-7},
	PAGES = {650--666},
	ISSN = {0007-4497,1952-4773},
	MRCLASS = {60H07 (60H30 60J60)},
	MRNUMBER = {2838095},
	MRREVIEWER = {Shi\ Zan\ Fang},
	DOI = {10.1016/j.bulsci.2011.07.007},
	URL = {https://doi-org.proxy.libraries.rutgers.edu/10.1016/j.bulsci.2011.07.007},
}

@book {daprato1996ergodicity,
	AUTHOR = {Da Prato, G. and Zabczyk, J.},
	TITLE = {Ergodicity for infinite-dimensional systems},
	SERIES = {London Mathematical Society Lecture Note Series},
	VOLUME = {229},
	PUBLISHER = {Cambridge University Press, Cambridge},
	YEAR = {1996},
	PAGES = {xii+339},
	ISBN = {0-521-57900-7},
	MRCLASS = {60H15 (28D05 60J25 60J35)},
	MRNUMBER = {1417491},
	MRREVIEWER = {Bohdan\ Maslowski},
	DOI = {10.1017/CBO9780511662829},
	URL = {https://doi-org.proxy.libraries.rutgers.edu/10.1017/CBO9780511662829},
}

@book {oksendal2003stochastic,
	AUTHOR = {{\O}ksendal, Bernt},
	TITLE = {Stochastic differential equations},
	SERIES = {Universitext},
	EDITION = {Sixth},
	NOTE = {An introduction with applications},
	PUBLISHER = {Springer-Verlag, Berlin},
	YEAR = {2003},
	PAGES = {xxiv+360},
	ISBN = {3-540-04758-1},
	MRCLASS = {60H10 (60G44 60J60)},
	MRNUMBER = {2001996},
	DOI = {10.1007/978-3-642-14394-6},
	URL = {https://doi-org.proxy.libraries.rutgers.edu/10.1007/978-3-642-14394-6},
}

@book {revuz1999continuous,
	AUTHOR = {Revuz, Daniel and Yor, Marc},
	TITLE = {Continuous martingales and {B}rownian motion},
	SERIES = {Grundlehren der mathematischen Wissenschaften [Fundamental
	Principles of Mathematical Sciences]},
	VOLUME = {293},
	EDITION = {Third},
	PUBLISHER = {Springer-Verlag, Berlin},
	YEAR = {1999},
	PAGES = {xiv+602},
	ISBN = {3-540-64325-7},
	MRCLASS = {60G44 (60G07 60H05)},
	MRNUMBER = {1725357},
	DOI = {10.1007/978-3-662-06400-9},
	URL = {https://doi-org.proxy.libraries.rutgers.edu/10.1007/978-3-662-06400-9},
}

@article {freedman1975tail,
	AUTHOR = {Freedman, David A.},
	TITLE = {On tail probabilities for martingales},
	JOURNAL = {Ann. Probab.},
	FJOURNAL = {The Annals of Probability},
	VOLUME = {3},
	YEAR = {1975},
	PAGES = {100--118},
	ISSN = {0091-1798},
	MRCLASS = {60G45 (60F05)},
	MRNUMBER = {380971},
	MRREVIEWER = {D.\ Siegmund},
	DOI = {10.1214/aop/1176996452},
	URL = {https://doi-org.proxy.libraries.rutgers.edu/10.1214/aop/1176996452},
}

@book {karatzas1991brownian,
	AUTHOR = {Karatzas, Ioannis and Shreve, Steven E.},
	TITLE = {Brownian motion and stochastic calculus},
	SERIES = {Graduate Texts in Mathematics},
	VOLUME = {113},
	EDITION = {Second},
	PUBLISHER = {Springer-Verlag, New York},
	YEAR = {1991},
	PAGES = {xxiv+470},
	ISBN = {0-387-97655-8},
	MRCLASS = {60J65 (35K99 35R60 60G44 60H10 60J60)},
	MRNUMBER = {1121940},
	DOI = {10.1007/978-1-4612-0949-2},
	URL = {https://doi-org.proxy.libraries.rutgers.edu/10.1007/978-1-4612-0949-2},
}

@article{khamaru2024inference,
	title={Inference with the upper confidence bound algorithm},
	author={Khamaru, Koulik and Zhang, Cun-Hui},
	journal={arXiv preprint arXiv:2408.04595},
	year={2024}
}

@article{fan2024fragility,
	title={The fragility of optimized bandit algorithms},
	author={Fan, Lin and Glynn, Peter W.},
	journal={Operations Research},
	year={2024},
	publisher={INFORMS}
}

@article{fan2022typical,
	title={The typical behavior of bandit algorithms},
	author={Fan, Lin and Glynn, Peter W.},
	journal={arXiv preprint arXiv:2210.05660},
	year={2022}
}

@article {white1958limiting,
	AUTHOR = {White, John S.},
	TITLE = {The limiting distribution of the serial correlation
	coefficient in the explosive case},
	JOURNAL = {Ann. Math. Statist.},
	FJOURNAL = {Annals of Mathematical Statistics},
	VOLUME = {29},
	YEAR = {1958},
	PAGES = {1188--1197},
	ISSN = {0003-4851},
	MRCLASS = {62.00},
	MRNUMBER = {100952},
	MRREVIEWER = {P.\ Whittle},
	DOI = {10.1214/aoms/1177706450},
	URL = {https://doi.org/10.1214/aoms/1177706450},
}

@article {dickey1979distribution,
	AUTHOR = {Dickey, David A. and Fuller, Wayne A.},
	TITLE = {Distribution of the estimators for autoregressive time series
	with a unit root},
	JOURNAL = {J. Amer. Statist. Assoc.},
	FJOURNAL = {Journal of the American Statistical Association},
	VOLUME = {74},
	YEAR = {1979},
	NUMBER = {366},
	PAGES = {427--431},
	ISSN = {0162-1459,1537-274X},
	MRCLASS = {62M02 (62M05)},
	MRNUMBER = {548036},
	MRREVIEWER = {E.\ J.\ Hannan},
	URL =
	{http://links.jstor.org/sici?sici=0162-1459(197906)74:366<427:DOTEFA>2.0.CO;2-3&origin=MSN},
}

@article{fan2024precise,
	title={Precise Asymptotics and Refined Regret of Variance-Aware {UCB}},
	author={Fan, Yingying and Han, Yuxuan and Lv, Jinchi and Xu, Xiaocong and Zhou, Zhengyuan},
	journal={arXiv preprint arXiv:2412.08843},
	year={2024}
}

@article{bibaut2021post,
	title={Post-contextual-bandit inference},
	author={Bibaut, Aur{\'e}lien and Dimakopoulou, Maria and Kallus, Nathan and Chambaz, Antoine and van Der Laan, Mark},
	journal={Advances in Neural Information Processing Systems},
	volume={34},
	pages={28548--28559},
	year={2021}
}

@inproceedings{zhan2021off,
	title={Off-policy evaluation via adaptive weighting with data from contextual bandits},
	author={Zhan, Ruohan and Hadad, Vitor and Hirshberg, David A and Athey, Susan},
	booktitle={Proceedings of the 27th ACM SIGKDD Conference on Knowledge Discovery \& Data Mining},
	pages={2125--2135},
	year={2021}
}

@article{syrgkanis2023post,
	title={Post-Episodic Reinforcement Learning Inference},
	author={Syrgkanis, Vasilis and Zhan, Ruohan},
	journal={arXiv preprint arXiv:2302.08854},
	year={2023}
}

@article {de2004self,
	AUTHOR = {de la Pe\~na, Victor H. and Klass, Michael J. and Lai, Tze
	Leung},
	TITLE = {Self-normalized processes: exponential inequalities, moment
	bounds and iterated logarithm laws},
	JOURNAL = {Ann. Probab.},
	FJOURNAL = {The Annals of Probability},
	VOLUME = {32},
	YEAR = {2004},
	NUMBER = {3A},
	PAGES = {1902--1933},
	ISSN = {0091-1798,2168-894X},
	MRCLASS = {60E15 (60G40 60G42 60G44)},
	MRNUMBER = {2073181},
	MRREVIEWER = {Thierry\ Edmond\ Huillet},
	DOI = {10.1214/009117904000000397},
	URL = {https://doi.org/10.1214/009117904000000397},
}

@article {waudby2024anytime,
	AUTHOR = {Waudby-Smith, Ian and Wu, Lili and Ramdas, Aaditya and
	Karampatziakis, Nikos and Mineiro, Paul},
	TITLE = {Anytime-valid off-policy inference for contextual bandits},
	JOURNAL = {ACM/IMS J. Data Sci.},
	FJOURNAL = {ACM/IMS Journal of Data Science},
	VOLUME = {1},
	YEAR = {2024},
	NUMBER = {3},
	PAGES = {Art. 10, 42},
	ISSN = {2831-3194},
	MRCLASS = {68T07 (62D20 62L10)},
	MRNUMBER = {4922618},
}

@article{shin2019bias,
	title={On the bias, risk and consistency of sample means in multi-armed bandits},
	author={Shin, Jaehyeok and Ramdas, Aaditya and Rinaldo, Alessandro},
	journal={arXiv preprint arXiv:1902.00746},
	year={2019}
}

@article {deshpande2023online,
	AUTHOR = {Deshpande, Yash and Javanmard, Adel and Mehrabi, Mohammad},
	TITLE = {Online debiasing for adaptively collected high-dimensional
	data with applications to time series analysis},
	JOURNAL = {J. Amer. Statist. Assoc.},
	FJOURNAL = {Journal of the American Statistical Association},
	VOLUME = {118},
	YEAR = {2023},
	NUMBER = {542},
	PAGES = {1126--1139},
	ISSN = {0162-1459,1537-274X},
	MRCLASS = {99-01},
	MRNUMBER = {4595482},
	DOI = {10.1080/01621459.2021.1979011},
	URL = {https://doi.org/10.1080/01621459.2021.1979011},
}

@article{ying2024adaptive,
	title={Adaptive linear estimating equations},
	author={Ying, Mufang and Khamaru, Koulik and Zhang, Cun-Hui},
	journal={Advances in Neural Information Processing Systems},
	volume={36},
	year={2024}
}

@article{lin2024statistical,
	title={Statistical limits of adaptive linear models: low-dimensional estimation and inference},
	author={Lin, Licong and Ying, Mufang and Ghosh, Suvrojit and Khamaru, Koulik and Zhang, Cun-Hui},
	journal={Advances in Neural Information Processing Systems},
	volume={36},
	year={2024}
}

@article {agrawal2017near,
	AUTHOR = {Agrawal, Shipra and Goyal, Navin},
	TITLE = {Near-optimal regret bounds for {T}hompson sampling},
	JOURNAL = {J. ACM},
	FJOURNAL = {Journal of the ACM},
	VOLUME = {64},
	YEAR = {2017},
	NUMBER = {5},
	PAGES = {Art. 30, 24},
	ISSN = {0004-5411,1557-735X},
	MRCLASS = {68Q87 (60G42 62C10 62L10 68W20)},
	MRNUMBER = {3716885},
	DOI = {10.1145/3088510},
	URL = {https://doi.org/10.1145/3088510},
}

@inproceedings{kaufmann2012thompson,
	title={Thompson sampling: An asymptotically optimal finite-time analysis},
	author={Kaufmann, Emilie and Korda, Nathaniel and Munos, R{\'e}mi},
	booktitle={International Conference on Algorithmic Learning Theory},
	pages={199--213},
	year={2012},
	organization={Springer}
}

@article{russo2018tutorial,
	title={A tutorial on thompson sampling},
	author={Russo, Daniel J and Van Roy, Benjamin and Kazerouni, Abbas and Osband, Ian and Wen, Zheng},
	journal={Foundations and Trends{\textregistered} in Machine Learning},
	volume={11},
	number={1},
	pages={1--96},
	year={2018},
	publisher={Now Publishers, Inc.}
}

@article {agrawal1995sample,
	AUTHOR = {Agrawal, Rajeev},
	TITLE = {Sample mean based index policies with {$O(\log n)$} regret for
	the multi-armed bandit problem},
	JOURNAL = {Adv. in Appl. Probab.},
	FJOURNAL = {Advances in Applied Probability},
	VOLUME = {27},
	YEAR = {1995},
	NUMBER = {4},
	PAGES = {1054--1078},
	ISSN = {0001-8678,1475-6064},
	MRCLASS = {62L05 (60F10 93E35)},
	MRNUMBER = {1358906},
	MRREVIEWER = {John\ C.\ Gittins},
	DOI = {10.2307/1427934},
	URL = {https://doi.org/10.2307/1427934},
}

@article {auer2002nonstochastic,
	AUTHOR = {Auer, Peter and Cesa-Bianchi, Nicol\`o{} and Freund, Yoav and
	Schapire, Robert E.},
	TITLE = {The nonstochastic multiarmed bandit problem},
	JOURNAL = {SIAM J. Comput.},
	FJOURNAL = {SIAM Journal on Computing},
	VOLUME = {32},
	YEAR = {2002/03},
	NUMBER = {1},
	PAGES = {48--77},
	ISSN = {0097-5397,1095-7111},
	MRCLASS = {91A20 (68Q32 68T05 91A60)},
	MRNUMBER = {1954855},
	MRREVIEWER = {Mark\ R.\ Jerrum},
	DOI = {10.1137/S0097539701398375},
	URL = {https://doi.org/10.1137/S0097539701398375},
}

@article {robbins1952some,
	AUTHOR = {Robbins, Herbert},
	TITLE = {Some aspects of the sequential design of experiments},
	JOURNAL = {Bull. Amer. Math. Soc.},
	FJOURNAL = {Bulletin of the American Mathematical Society},
	VOLUME = {58},
	YEAR = {1952},
	PAGES = {527--535},
	ISSN = {0002-9904},
	MRCLASS = {62.0X},
	MRNUMBER = {50246},
	MRREVIEWER = {H.\ B.\ Mann},
	DOI = {10.1090/S0002-9904-1952-09620-8},
	URL = {https://doi.org/10.1090/S0002-9904-1952-09620-8},
}

@article{thompson1933likelihood,
	title={On the likelihood that one unknown probability exceeds another in view of the evidence of two samples},
	author={Thompson, William R},
	journal={Biometrika},
	volume={25},
	number={3-4},
	pages={285--294},
	year={1933},
	publisher={Oxford University Press}
}

@article {lai1987adaptive,
	AUTHOR = {Lai, Tze Leung},
	TITLE = {Adaptive treatment allocation and the multi-armed bandit
	problem},
	JOURNAL = {Ann. Statist.},
	FJOURNAL = {The Annals of Statistics},
	VOLUME = {15},
	YEAR = {1987},
	NUMBER = {3},
	PAGES = {1091--1114},
	ISSN = {0090-5364,2168-8966},
	MRCLASS = {62L05 (62F15)},
	MRNUMBER = {902248},
	MRREVIEWER = {Peter\ Watts\ Jones},
	DOI = {10.1214/aos/1176350495},
	URL = {https://doi.org/10.1214/aos/1176350495},
}

@article {hadad2021confidence,
	AUTHOR = {Hadad, Vitor and Hirshberg, David A. and Zhan, Ruohan and
	Wager, Stefan and Athey, Susan},
	TITLE = {Confidence intervals for policy evaluation in adaptive
	experiments},
	JOURNAL = {Proc. Natl. Acad. Sci. USA},
	FJOURNAL = {Proceedings of the National Academy of Sciences of the United
	States of America},
	VOLUME = {118},
	YEAR = {2021},
	NUMBER = {15},
	PAGES = {Paper No. e2014602118, 10},
	ISSN = {0027-8424,1091-6490},
	MRCLASS = {62L05 (60F05 62F03 62F10)},
	MRNUMBER = {4294058},
	DOI = {10.1073/pnas.2014602118},
	URL = {https://doi.org/10.1073/pnas.2014602118},
}

@article{zhang2020inference,
	title={Inference for batched bandits},
	author={Zhang, Kelly and Janson, Lucas and Murphy, Susan},
	journal={Advances in Neural Information Processing Systems},
	volume={33},
	pages={9818--9829},
	year={2020}
}

@article{abbasi2011improved,
	title={Improved algorithms for linear stochastic bandits},
	author={Abbasi-Yadkori, Yasin and P{\'a}l, D{\'a}vid and Szepesv{\'a}ri, Csaba},
	journal={Advances in Neural Information Processing Systems},
	volume={24},
	year={2011}
}

@book {de2009self,
	AUTHOR = {de la Pe\~na, Victor H. and Lai, Tze Leung and Shao, Qi-Man},
	TITLE = {Self-normalized processes},
	SERIES = {Probability and its Applications (New York)},
	NOTE = {Limit theory and statistical applications},
	PUBLISHER = {Springer-Verlag, Berlin},
	YEAR = {2009},
	PAGES = {xiv+275},
	ISBN = {978-3-540-85635-1},
	MRCLASS = {60-01 (60E15 60F05 60F10 60G07 62F40 62G20 62L05)},
	MRNUMBER = {2488094},
	MRREVIEWER = {Fuchang\ Gao},
	DOI = {10.1007/978-3-540-85636-8},
	URL = {https://doi.org/10.1007/978-3-540-85636-8},
}

@article {lai1985asymptotically,
	AUTHOR = {Lai, Tze Leung and Robbins, Herbert},
	TITLE = {Asymptotically efficient adaptive allocation rules},
	JOURNAL = {Adv. in Appl. Math.},
	FJOURNAL = {Advances in Applied Mathematics},
	VOLUME = {6},
	YEAR = {1985},
	NUMBER = {1},
	PAGES = {4--22},
	ISSN = {0196-8858,1090-2074},
	MRCLASS = {62L05},
	MRNUMBER = {776826},
	MRREVIEWER = {Albrecht\ Irle},
	DOI = {10.1016/0196-8858(85)90002-8},
	URL = {https://doi.org/10.1016/0196-8858(85)90002-8},
}

@book{lattimore2020bandit,
	title={Bandit algorithms},
	author={Lattimore, Tor and Szepesv{\'a}ri, Csaba},
	year={2020},
	publisher={Cambridge University Press}
}

@article {auer2002finite,
	AUTHOR = {Auer, Peter and Cesa-Bianchi, Nicolò and Fischer, Paul},
	TITLE = {Finite-time Analysis of the Multiarmed Bandit Problem},
	JOURNAL = {Machine Learning},
	VOLUME = {47},
	YEAR = {2002},
	NUMBER = {2},
	PAGES = {235--256},
	ISSN = {1573-0565},
	URL =
	{https://doi.org/10.1023/A:1013689704352},
}

@article {lai1982least,
	AUTHOR = {Lai, Tze Leung and Wei, Ching Zong},
	TITLE = {Least squares estimates in stochastic regression models with
	applications to identification and control of dynamic systems},
	JOURNAL = {Ann. Statist.},
	FJOURNAL = {The Annals of Statistics},
	VOLUME = {10},
	YEAR = {1982},
	NUMBER = {1},
	PAGES = {154--166},
	ISSN = {0090-5364,2168-8966},
	MRCLASS = {62J05 (62M10)},
	MRNUMBER = {642726},
	MRREVIEWER = {Hilmar\ Drygas},
	URL =
	{http://links.jstor.org/sici?sici=0090-5364(198203)10:1<154:LSEISR>2.0.CO;2-U&origin=MSN},
}

@article {lin2025semi,
	AUTHOR = {Lin, Licong and Khamaru, Koulik and Wainwright, Martin J.},
	TITLE = {Semiparametric inference based on adaptively collected data},
	JOURNAL = {Ann. Statist.},
	FJOURNAL = {The Annals of Statistics},
	VOLUME = {53},
	YEAR = {2025},
	NUMBER = {3},
	PAGES = {989--1014},
	ISSN = {0090-5364,2168-8966},
	MRCLASS = {62J12 (62G08 62G20)},
	MRNUMBER = {4925113},
	DOI = {10.1214/24-aos2485},
	URL = {https://doi-org.proxy.libraries.rutgers.edu/10.1214/24-aos2485},
}

@article {khamaru2025near,
	AUTHOR = {Khamaru, Koulik and Deshpande, Yash and Lattimore, Tor and
	Mackey, Lester and Wainwright, Martin J.},
	TITLE = {Near-optimal inference in adaptive linear regression},
	JOURNAL = {Ann. Statist.},
	FJOURNAL = {The Annals of Statistics},
	VOLUME = {53},
	YEAR = {2025},
	NUMBER = {6},
	ISSN = {0090-5364,2168-8966},
	MRCLASS = {62J05 (62G20)},
	MRNUMBER = {5005618},
	DOI = {10.1214/24-aos2450},
	URL = {https://doi-org.proxy.libraries.rutgers.edu/10.1214/24-aos2450},
}

@inproceedings{deshpande2018accurate,
	title={Accurate inference for adaptive linear models},
	author={Deshpande, Yash and Mackey, Lester and Syrgkanis, Vasilis and Taddy, Matt},
	booktitle={International Conference on Machine Learning},
	pages={1194--1203},
	year={2018},
	organization={PMLR}
}

@book {hall1980martingale,
	AUTHOR = {Hall, P. and Heyde, C. C.},
	TITLE = {Martingale limit theory and its application},
	SERIES = {Probability and Mathematical Statistics},
	PUBLISHER = {Academic Press, Inc. [Harcourt Brace Jovanovich, Publishers],
	New York-London},
	YEAR = {1980},
	PAGES = {xii+308},
	ISBN = {0-12-319350-8},
	MRCLASS = {60-02 (60B12 60F05 60G42)},
	MRNUMBER = {624435},
	MRREVIEWER = {David\ J.\ Aldous},
}

@article{kalvit2021closer,
	title={A closer look at the worst-case behavior of multi-armed bandit algorithms},
	author={Kalvit, Anand and Zeevi, Assaf},
	journal={Advances in Neural Information Processing Systems},
	volume={34},
	pages={8807--8819},
	year={2021}
}

@book {bakry2014analysis,
	AUTHOR = {Bakry, Dominique and Gentil, Ivan and Ledoux, Michel},
	TITLE = {Analysis and geometry of {M}arkov diffusion operators},
	SERIES = {Grundlehren der mathematischen Wissenschaften [Fundamental
	Principles of Mathematical Sciences]},
	VOLUME = {348},
	PUBLISHER = {Springer, Cham},
	YEAR = {2014},
	PAGES = {xx+552},
	ISBN = {978-3-319-00226-2; 978-3-319-00227-9},
	MRCLASS = {60J25 (58J65 60J35 60J60)},
	MRNUMBER = {3155209},
	MRREVIEWER = {Ming Liao},
	DOI = {10.1007/978-3-319-00227-9},
	URL = {https://doi.org/10.1007/978-3-319-00227-9},
}

@book {de2012decoupling,
	AUTHOR = {de la Pe{\~n}a, V{\'{\i}}ctor H. and Gin{\'e}, Evarist},
	TITLE = {Decoupling},
	SERIES = {Probability and its Applications (New York)},
	NOTE = {From dependence to independence,
	Randomly stopped processes. $U$-statistics and processes.
	Martingales and beyond},
	PUBLISHER = {Springer-Verlag, New York},
	YEAR = {1999},
	PAGES = {xvi+392},
	ISBN = {0-387-98616-2},
	MRCLASS = {60E15 (60B12 60F05 60F15 60F17 60G40)},
	MRNUMBER = {1666908},
	MRREVIEWER = {Miguel A. Arcones},
	DOI = {10.1007/978-1-4612-0537-1},
	URL = {http://dx.doi.org/10.1007/978-1-4612-0537-1},
}

@book {gine2015mathematical,
	AUTHOR = {Gin\'{e}, Evarist and Nickl, Richard},
	TITLE = {Mathematical foundations of infinite-dimensional statistical
	models},
	SERIES = {Cambridge Series in Statistical and Probabilistic Mathematics,
	[40]},
	PUBLISHER = {Cambridge University Press, New York},
	YEAR = {2016},
	PAGES = {xiv+690},
	ISBN = {978-1-107-04316-9},
	MRCLASS = {62-02 (60F05 60F17 60G15 62C20 62G07 62G08 62Gxx)},
	MRNUMBER = {3588285},
	MRREVIEWER = {Natalie Neumeyer},
	DOI = {10.1017/CBO9781107337862},
	URL = {https://doi-org.proxy.libraries.rutgers.edu/10.1017/CBO9781107337862},
}

\end{document}